\documentclass[10pt,reqno,dvipdfmx]{amsart}

\usepackage{amsmath,amssymb,amsthm,cite}

\usepackage[abbrev]{amsrefs}
\usepackage{graphicx}

\usepackage[top=30mm, bottom=30mm, left=25mm, right=25mm]{geometry}

\newtheorem{theo}{Theorem}[section]
\newtheorem{lemm}[theo]{Lemma}

\newtheorem{prop}[theo]{Proposition}

\numberwithin{equation}{section}

\theoremstyle{definition}

\newtheorem{rema}[theo]{Remark}

\newcommand{\BB}{\mathbb{B}}
\newcommand{\BC}{\mathbb{C}}
\newcommand{\BE}{\mathbb{E}}
\newcommand{\BF}{\mathbb{F}}

\newcommand{\BI}{\mathbb{I}}
\newcommand{\BK}{\mathbb{K}}
\newcommand{\BN}{\mathbb{N}}

\newcommand{\BR}{\mathbb{R}}
\newcommand{\BS}{\mathbb{S}}
\newcommand{\BT}{\mathbb{T}}
\newcommand{\BU}{\mathbb{U}}

\newcommand{\ba}{\boldsymbol{a}}

\newcommand{\bb}{\boldsymbol{b}}
\newcommand{\bB}{\boldsymbol{B}}

\newcommand{\bD}{\boldsymbol{D}}

\newcommand{\bff}{\boldsymbol{f}}
\newcommand{\bF}{\boldsymbol{F}}
\newcommand{\bG}{\boldsymbol{G}}
\newcommand{\bg}{\boldsymbol{g}}
\newcommand{\bH}{\boldsymbol{H}}
\newcommand{\bh}{\boldsymbol{h}}
\newcommand{\bk}{\boldsymbol{k}}
\newcommand{\bK}{\boldsymbol{K}}
\newcommand{\bI}{\boldsymbol{I}}

\newcommand{\bN}{\boldsymbol{N}}
\newcommand{\bn}{\boldsymbol{n}}

\newcommand{\bu}{\boldsymbol{u}}

\newcommand{\bv}{\boldsymbol{v}}

\newcommand{\CB}{\mathcal{B}}

\newcommand{\CE}{\mathcal{E}}
\newcommand{\CF}{\mathcal{F}}
\newcommand{\CH}{\mathcal{H}}

\newcommand{\CL}{\mathcal{L}}
\newcommand{\CM}{\mathcal{M}}

\newcommand{\CO}{\mathcal{O}}
\newcommand{\CP}{\mathcal{P}}

\newcommand{\CR}{\mathcal{R}}
\newcommand{\CS}{\mathcal{S}}
\newcommand{\CT}{\mathcal{T}}
\newcommand{\CU}{\mathcal{U}}
\newcommand{\CV}{\mathcal{V}}

\newcommand{\fp}{\mathfrak{p}}

\newcommand{\pd}{\partial}
\newcommand{\wt}{\widetilde}
\newcommand{\wh}{\widehat}
\newcommand{\ov}{\overline}

\newcommand{\dv}{\mathrm{div}\,}
\newcommand{\Hol}{\mathrm{Hol}}

\newcommand{\dx}{\,\mathrm{d}x}

\newcommand{\dt}{\,\mathrm{d}t}
\newcommand{\RE}{\mathrm{Re}\,}

\begin{document}
\title[]{Local well-posedness of incompressible viscous fluids \\ in bounded cylinders with $90^\circ$-contact angle}

\author{Keiichi Watanabe}

\address{Global Center for Science and Engineering, Waseda University, 3-4-1 Ookubo, Shinjuku-ku, Tokyo, 169-8555, Japan}		


\email{keiichi-watanabe@akane.waseda.jp}

\thanks{This research was partly supported by JSPS Grant-in-aid for Research Activity Start-up Grant Number 20K22311 and Waseda University Grant for Special Research Projects.}



\maketitle

\begin{abstract}
We consider a free boundary problem of the Navier--Stokes equations in the three-dimensional Euclidean space with moving contact line, where the 90$^\circ$-contact angle condition is posed. We show that for given $T > 0$ the problem is local well-posed on $(0, T)$ provided that the initial data are small. In contrast to the strategy in Wilke (2013), we study the transformed problem in an $L^p$-in-time and $L^q$-in-space setting, which yields the optimal regularity of the initial data. 
\end{abstract}

\section{Introduction}
\noindent
\subsection{Description of the problem}
We consider a fixed rigid body with a bounded, cylindrical, and simply connected cavity $\CV \subset \BR^3$, \textit{partially} filled with an incompressible viscous Newtonian fluid that fills a region $\Omega (t)$ at time $t > 0$. Here, the domain of cavity $\CV$ is defined by
\begin{equation*}
\CV := \{(x_1, x_2, x_3) \mid (x_1, x_2) \in D, \enskip x_3 \in (- 2 H, 2 H)\},
\end{equation*}
where $D \subset \BR^2$ is a bounded open set and $H$ is a given positive constant. We suppose that $D$ is surrounded by a smooth boundary $\pd D$ of class $C^4$. A sharp interface $\Gamma (t)$ separates the cavity $\CV$ into the fluid part $\Omega (t)$ and the vacant part $\CV \setminus \Omega (t)$. In addition, the boundary of the free interface $\pd \Gamma (t)$ separates the boundary of cavity $\pd \CV$ into the wetting part $\Sigma (t)$ and the drying part $\Sigma^*(t)$. The contact line $S (t)$ is defined by $S (t) = \pd \CV \cap \pd \Gamma (t)$. By abuse of notation, we will write $\Omega_t = \Omega (t)$, $\Gamma_t = \Gamma (t)$, $\Sigma_t = \Sigma (t)$, $\Sigma_t^* = \Sigma^* (t)$, and $S_t = S (t)$. If the initial position $\Gamma_0 = \Gamma (0)$ of $\Gamma_t$ can be approximate $D$ in the sense that the Hausdorff distance of the second order bundles of $\Gamma_t$ and $D$ is small enough, we may assume that the unknown free surface $\Gamma_t$ can be parameterized by means of an unknown height function $\eta$ such that
\begin{equation}
\label{Gamma_t}
\Gamma_t := \{(x_1, x_2, x_3) \mid (x_1, x_2) \in D, \enskip x_3 = \eta (x_1, x_2, t), \enskip t \ge 0\}
\end{equation}
whenever $\lvert \eta \rvert_{L^\infty (D)}$ and $\lvert \nabla_{x'} \eta \rvert_{L^\infty (D)}$ are suitably small. Here, we have set $x' = (x_1, x_2)$ for short. We may suppose that $\Omega_t$, $\Sigma_t$, and $S_t$ are represented by
\begin{align*}
\Omega_t & := \{(x_1, x_2, x_3) \mid (x_1, x_2) \in D, \enskip x_3 \in (- H, \eta (x_1, x_2, t)), \enskip t \ge 0\}, \\
\Sigma_t & := \{(x_1, x_2, x_3) \mid (x_1, x_2) \in \pd D, \enskip x_3 \in (- H, \eta (x_1, x_2, t)), \enskip t \ge 0\}, \\
S_t & := \{(x_1, x_2, x_3) \mid (x_1, x_2) \in \pd D, \enskip x_3 = \eta (x_1, x_2, t), \enskip t \ge 0\},
\end{align*}
respectively. Furthermore, we define the bottom $B$ of the cavity $\CV$ by
\begin{equation*}
B := \{(x_1, x_2, x_3) \mid (x_1, x_2) \in D, \enskip x_3 = - H\}.
\end{equation*}
This article considers the following free boundary problem of the Navier--Stokes equations: Given $\Gamma_0 \subset \CV$ and $\bv_0 \colon \Omega_0 \to \BR^3$, find a family $\{\Gamma_t\}_{t \ge 0}$ and a pair of functions $\bv (\cdot, t) \colon \Omega_t \to \BR^3$ and $\fp (\cdot, t) \colon \Omega_t \to \BR$ satisfying
\begin{align}
\label{eq-main}
\left\{
\begin{aligned}
\pd_t \bv + (\bv \cdot \nabla) \bv & = \dv \BT (\bv, \fp), & \quad & \text{in $\Omega_t$,} \\
\dv \bv & = 0, & \quad & \text{in $\Omega_t$,} \\
\BT (\bv, \fp) \bn_{\Gamma_t} & = \sigma \CH_{\Gamma_t} \bn_{\Gamma_t} - p_0 \bn_{\Gamma_t}, & \quad & \text{in $\Gamma_t$,} \\
V_{\Gamma_t} & = \langle \bv, \bn_{\Gamma_t} \rangle, & \quad & \text{in $\Gamma_t$,} \\
P_{\Sigma_t} (2 \mu \bD (\bv) \bn_{\Sigma_t}) & = 0, & \quad & \text{on $\Sigma_t$,} \\
\langle \bv, \bn_{\Sigma_t} \rangle & = 0, & \quad & \text{on $\Sigma_t$,} \\
P_B (2 \mu \bD (\bv) \bn_B) & = 0, & \quad & \text{on $B$,} \\
\langle \bv, \bn_B \rangle & = 0, & \quad & \text{on $B$,} \\
\langle \bn_{\Gamma_t}, \bn_{\Sigma_t} \rangle & = 0, & \quad & \text{on $S_t$}, \\
\bv (0) & = \bv_0 & \quad & \text{in $\Omega_0$}, \\ 
\Gamma_t \vert_{t = 0} & = \Gamma_0,
\end{aligned}
\right.
\end{align}
where we use the notation $\langle \cdot, \cdot \rangle$ to describe the dot product of vector fields. In this paper, we consider the case where the initial position $\Gamma_0$ of the free boundary is the graph of a height function $\eta_0$ on $D$, i.e.,
\begin{equation*}
\Gamma_0 := \{(x_1, x_2, x_3) \mid (x_1, x_2) \in D, \enskip x_3 \in \eta_0 (x_1, x_2)\}.
\end{equation*}
Here, $\bv$ and $\fp$ are unknown functions describing the velocity fields and the pressure of incompressible viscous fluid, respectively. In the system \eqref{eq-main}, the symbol $\BT (\bv, \fp)$ stands for the viscous stress tensor defined by
\begin{align*}
\BT (\bv, \fp) = \BS (\bv) - \fp \bI = 2 \mu \bD (\bv) - \fp \bI,
\end{align*}
where $\bD (\bv) = 2^{- 1} (\nabla \bv + (\nabla \bv)^\top)$ is the deformation tensor; $\mu > 0$ stands for a constant denoting the viscosity coefficient; $\bn_{\Gamma_t}$, $\bn_{\Sigma_t}$, and $\bn_B$ are the outward unit normal field on $\Gamma_t$, $\Sigma_t$, and $B$, respectively; $\sigma$ stands for the surface tension coefficient, which is a given positive constant; $\CH_{\Gamma_t}$ stands for the double mean curvature of $\Gamma_t$ given by $\CH_{\Gamma_t} = - \dv_{\Gamma_t} \bn_{\Gamma_t}$, where $\dv_{\Gamma_t}$ is the surface divergence on $\Gamma_t$; $p_0$ stands for an external pressure, which is a given positive \textit{constant}; $V_{\Gamma_t}$ stands for the normal velocity of $\Gamma_t$; $P_{\Sigma_t} := \bI - \bn_{\Sigma_t} \otimes \bn_{\Sigma_t}$ and $P_B := \bI - \bn_B \otimes \bn_B$ stand for the orthogonal projections onto the tangent bundle of $\Sigma_t$ and $B$, respectively. The \textit{contact angle} $\theta = \theta (x, t)$ is defined by $\cos \theta = - \langle \bn_{\Gamma_t}, \bn_{\Sigma_t} \rangle$ on $S_t$. Since we assume $\theta \equiv \pi \slash 2$, the boundary condition on the contact line can be read as $\langle \bn_{\Gamma_t}, \bn_{\Sigma_t} \rangle = 0$ on $S_t$. Although it is not difficult to consider the problem in the $n$-dimensional cases with $n \ge 3$, we here only deal with the three-dimensional case for simplicity.
\subsection{Historical remarks}
If the contact angle $\theta$ is fixed into the trivial cases, say, $\theta = \pi$ or $\theta = \pi\slash2$, there are the pioneering contributions by Solonnikov~\cite{S95} and Wilke~\cite{W17}. The key observations of their studies were that these contact angles remove the singularities at the contact lines. However, from the classical Young law, the contact angles seem to depend on the time if the initial contact angles are not equal to the contact angles at the equilibria, denoted by $\theta_\infty$. Hence, this shows that fixing the contact angle is a kind of idealization. Compared with their studies, the \textit{moving} contact angle problem for the two-dimensional \textit{Stokes} flow was considered by~\cite{GT18,ZT17}. To be more precise, as explained in \cite[Sec.~1.3]{GT18} and~\cite[Sec.~1]{ZT17}, they applied the Ren-E model~\cite{RE07}, i.e, there exists an increasing function $F$ such that
\begin{equation}
\label{cond-linevelocity_RenE}
F (0) = 0, \qquad V_{S_t} = F (\cos \theta - \cos \theta_\infty).
\end{equation}
Notice that these conditions guarantee the relation $(\cos \theta - \cos \theta_\infty) \langle \bv, \bn_{S_t}\rangle \le 0$, which infers
\begin{equation}
\label{ineq-1.4}
\int_{S_t} (\cos \theta - \cos \theta_\infty) \langle \bv, \bn_{S_t} \rangle \,\mathrm{d} l \le 0,
\end{equation}
i.e., the negative total available energy is a strict Lyapunov functional. Notice that they left an open question whether or not the boundary condition at the contact line has to be included to obtain a well-posed problem. In \cite{RE07}, Ren and E showed that if one supposes \eqref{cond-linevelocity_RenE} then the viscous dissipation rate is finite, but it was unclear whether the pressure is still logarithmically singular at the moving contact line, where this type of singularity is often referred as a \textit{weak singularity}. Recently, Fricke et al.~\cite{FKB19} showed that a weak singularity needs to be present at the contact line even if the contact angle is permitted to vary because the contact angle evolution is unphysical provided that the slip length of the rigid surface is positive and finite. The key observation of the study in \cite{FKB19} was to derive a kinematic evolution law for the dynamical contact angle, where the contact angle depends on $t$. To do this, they imposed the following additional conditions on the contact line instead of \eqref{cond-linevelocity_RenE}: Supposed that the contact angle $\theta$ and the contact line normal velocity $V_{S_t} := \langle \bv, \bn_{S_t} \rangle$ are related via $\theta = f (V_{S_t})$, where $f$ is some function satisfying
\begin{equation}
\label{cond-linevelocity}
f (0) = \theta_\infty, \qquad V_{S_t} (f (V_{S_t}) - \theta_\infty) \ge 0.
\end{equation}
Here, the contact line normal vector $\bn_{S_t}$ is defined via projection $P_{\pd \Sigma_t} = \bI - \bn_{\pd \Sigma_t} \otimes \bn_{\pd \Sigma_t}$ as
\begin{equation*}
\bn_{S_t} := \frac{P_{\pd \Sigma_t} \bn_{\pd \Gamma_t}}{\lvert P_{\pd \Sigma_t} \bn_{\pd \Gamma_t} \rvert},
\end{equation*}
cf. Fricke~\cite[Def.~2]{FKB19}, where $\bn_{\pd \Sigma_t}$ is a unit outer vector of $\pd \Sigma_t$ that is perpendicular to the lateral of the cavity $\CV$. The condition $V_{S_t} (f (V_{S_t}) - \theta_\infty) \ge 0$ ensures the energy dissipation of the system. This also implies that, in the absence of an external force, the contact line should only advance if the contact angle is greater than or equal to $\theta_\infty$ (and vice versa), which seems to be reasonable if one thinks of the spreading droplet problem. Roughly speaking, in their model, the moving direction of the contact line is monotone. Hence, in their model, the contact line cannot advance and ``go through'' the equilibrium contact line with the contact angle less than $\theta_\infty$. Summing up, as far as the author knows, the perfect prescription for the boundary condition on the contact line has not yet shown up. To evade this lack of clarity, we assume throughout this paper that the contact angle is constant and equal to 90 degrees. This is a kind of idealization but it is known that the total available energy is a strict Lyapunov functional. \par
Since the position of the free surface $\Gamma_t$ is a priori unknown, it will be convenient to transform the problem for the velocity and the pressure on a fixed domain. To this end, we apply the direct mapping method via a \textit{Hanzawa transformation}, where we can obtain precise regularity information for the free surface. Namely, we do not consider Lagrangian coordinates to derive the transformed problem. Here, the position of the fixed surface $\Gamma_*$ should be close to the unknown free surface $\Gamma_t$ in the sense that the Hausdorff distance of the second-order bundles of $\Gamma_t$ and $\Gamma_*$ is small enough. In this case, the Hanzawa transformation is a diffeomorphism mapping, so that we may obtain the well-posedness result of the free boundary problem from the fixed boundary problem. This shows that it is crucial to consider the problem in the appropriate fixed domain. For simplicity, in our study, we choose the right circular cylinder as a fixed domain, i.e., the domain $\Omega_t$ will be given as a perturbation of the right circular cylinder. The advantage of this setting is that we can apply the standard reflection arguments as was applied in \cite{W17}.   \par
As we will see in Section~\ref{sect-transform}, the transformed system via the Hanzawa transform becomes a \textit{quasilinear} parabolic system. Hence, in the present paper, we devote to prove solvability results based on the theory of maximal regularity, in which we can solve the nonlinear problem by the contraction mapping principle. The main difficulty of maximal $L^p - L^q$-regularity approach is that, if we study the corresponding linearized problem, the boundary data have to be in the intersection space
\begin{equation}
\label{intersection}
F^s_{p, q} (J; L^q (\pd D)) \cap L^p (J; B^{s \slash 2}_{q, q} (\pd D)) \qquad 0 < s < 1,
\end{equation}
where $F^s_{p, q}$ and $B^{s\slash2}_{q, q}$ denote the vector-valued inhomogeneous Triebel-Lizorkin space and the scalar-valued inhomogeneous Besov space, respectively, cf., \cite{DHP,PS16}. Hence, it is required to establish estimates for the nonlinear terms in this intersection space, which has not been established so far when $p \ne q$. Notice that if $p = q$, then the vector-valued Triebel-Lizorkin space becomes the vector-valued Sobolev-Slobodecki\u{\i} space, where the estimations for the nonlinear terms in this space are well-known, see, e.g., \cite[Appendix]{PS11}. Here, in~\cite{KPW13,PS10,PS11}, they studied the two-phase free boundary problem for the Navier--Stokes equations in the case $p = q$, where the free boundary can be understood in the classical sense. See also, e.g., \cite{MS91,PS02,S84,S89,S91} for results of the free boundary problem for the Navier--Stokes equations, which were established in anisotropic Sobolev-Slobodetski\u{\i} as well as in H\"older spaces. Compared with their studies, Shibata~\cite{S14,S16,S20} obtained the maximal $L^p - L^q$ regularity results for the linearized problem of the free boundary problem of the one-phase Navier--Stokes equations in the case $p \ne q$, see also \cite{MS17,SS20pre,SS11,SS20} for the two-phase case. Unfortunately, in their arguments, the boundary data were not lying in~\eqref{intersection}, and thus their results are not optimal in view of trace theorems. To overcome these fallacious, we use the recent contributions established by Meyries and Veraar~\cite{MV12,MV14} and Lindemulder~\cite{L19}. Using their results, we will show the principal linearization has maximal $L^p - L^q$ regularity in the case $p \ne q$, where the boundary data belong to the intersection space \eqref{intersection}. We also succeed to establish the estimates for the nonlinearities appeared on the boundaries in the intersection space~\eqref{intersection}. It should be emphasized that our approach completely works for refining the previous studies of the free boundary problems for the one-phase or two-phase Navier--Stokes equations in domains surronded by smooth boundary (e.g., bounded or exterior domains) to obtain optimal $L^p - L^q$-regularity space-time estimates for the corresponding linearized equation with $p \ne q$, which also yields well-posedness results for the free boundary problems with initial data possessing optimal regularity. \par
To work in the maximal regularity framework, it is crucial to study the \textit{model problems} in the whole space, a half space, a wedge domain with angle equal to $\pi\slash2$, see Section~\ref{sect-model-prob}. The model problems in half spaces are well-studied by many authors, e.g., \cite{PS16,S16,S20,SS12}, but we give more sophisticated results in view of the trace theory. For the model problem in a wedge domain with an angle equals to $\pi\slash2$, there are several studies by~\cite{BKK,K13,W17}. Their arguments relied on the standard refection argument but they are missing the condition for the existence of ``trace of trace," i.e., the trace onto an intersection of the boundaries. Hence, in this paper, we give optimal conditions for the existence of the trace onto an intersection of the boundaries --- this will be related to the compatibility conditions for the Cauchy problem and the precise regularity information of its solution on the contact line.
\subsection{Main result}
As the main result in the article we establish the local well-posedness of the problem~\eqref{eq-main}.
\begin{theo}
\label{th-main}
Let $p$, $q$, $\delta$ satisfy
\begin{equation}
\label{cond-thmain}
2 < p < \infty, \quad 3 < q < \infty, \quad \frac{1}{p} + \frac{3}{2 q} < \delta - \frac{1}2 \le \frac{1}2.
\end{equation}
Then given $T > 0$, there exists $\varepsilon_0 = \varepsilon_0 (T) > 0$ such that for any initial data
\begin{equation*}
(\bv_0, \eta_0) \in B^{2 (\delta - 1\slash p)}_{q, p} (\Omega_0)^3 \times B^{2 + \delta - 1\slash p - 1\slash q}_{q, p} (D),
\end{equation*}
satisfying the compatibility conditions
\begin{equation*}
\left\{\begin{aligned}
\dv \bv_0 & = 0, & \quad & \text{in $\Omega_0$}, \\
P_{\Gamma_0} (2 \mu \bD (\bv_0) \bn_{\Gamma_0}) & = 0, & \quad & \text{on $\ov{\Gamma_0}$}, \\
P_{\Sigma_0} (2 \mu \bD (\bv_0) \bn_{\Sigma_0}) = 0, \quad \langle \bv_0, \bn_{\Sigma_0} \rangle & = 0, & \quad & \text{on $\ov{\Sigma_0}$}, \\
P_B (2 \mu \bD (\bv_0) \bn_B) = 0, \quad \langle \bv_0, \bn_B \rangle & = 0, & \quad & \text{on $\ov B$}, \\
\langle \bn_{\Gamma_0}, \bn_{\Sigma_0} \rangle & = 0, & \quad & \text{on $S_0$},
\end{aligned}\right.
\end{equation*}
and the smallness condition
\begin{equation*}
\lvert \bv_0 \rvert_{B^{2 (\delta - 1\slash p)}_{q, p} (\Omega_0)} + \lvert \eta_0 \rvert_{B^{2 + \delta - 1\slash p - 1\slash q}_{q, p} (D)} \le \varepsilon_0, 
\end{equation*}
the problem~\eqref{eq-main} admits a unique classical solution $(\bv, \fp, \Gamma)$ on $(0, T)$, where $P_{\Gamma_0} := \bI - \bn_{\Gamma_0} \otimes \bn_{\Gamma_0}$, $P_{\Sigma_0} := \bI - \bn_{\Sigma_0} \otimes \bn_{\Sigma_0}$. Furthermore, the free boundary $\Gamma_t$ is the graph of a function $\eta (t)$ on $D$, the set $\CM = \bigcup_{t \in (0, T)} (\Gamma_t \times \{t\})$ is a real analytic manifold, and the function $(\bv, \fp) \colon \{(x, t) \in \Omega_t \times (0, T)\} \to \BR^4$ is real analytic.
\end{theo}
\begin{rema}
The condition \eqref{cond-thmain} induces the embeddings
\begin{equation*}
B^{2 (\delta - 1\slash p)}_{q, p} (\Omega_0) \hookrightarrow \mathrm{BUC}^1 (\ov{\Omega_0}), \quad B^{2 + \delta - 1\slash p - 1\slash q}_{q, p} (D) \hookrightarrow \mathrm{BUC}^2 (\ov{D}).
\end{equation*}
\end{rema}
\begin{rema}
The restriction \eqref{cond-thmain} implies the existence of the trace onto an intersection of the boundaries. Hence, we need the compatibility conditions on the initial contact line $S_0 := \pd \Gamma_0 \cap \pd \Sigma_0$ and on the corner of the boundary $\pd \Sigma_0 \cap \pd B$.
\end{rema}
\begin{rema}
The smallness condition for the initial velocity $\bv_0$ is due to the nonlinear term $D (\bu', \eta) = - \langle \bu', \nabla_{x'} \eta \rangle$ appeared in the transformed problem~\eqref{eq-fixed}, where $\bu' = (u_1, u_2)^\top$. In fact, this term cannot be small in the norm of $\BF_{3, \delta} (J; \Gamma_*)$, defined in \eqref{def-spaces}, even if $\lvert \nabla_{x'} \eta \rvert_{L^\infty (D)}$ is small. To avoid this harmful nonlinearity, we have to consider the modified term $\langle \mathbf{b} - \bu', \nabla_{x'} \eta \rangle$, where $\mathbf{b}$ is taken such that $\mathbf{b} (0) = \mathrm{Tr}_{\Gamma_*} \bu'_0$. We refer to Pr\"uss and Simonett~\cite{PS11} for the details, see also \cite{KPW13,PSW14,S20,SY17,S11}. However, to make our arguments simple, we keep the smallness condition for the initial velocity $\bv_0$.
\end{rema}
The plan of this paper is constructed as follows: Section~\ref{sect-transform} introduces the Hanzawa transformation to transform the free boundary problem~\eqref{eq-main} into a domain with a fixed boundary problem. Section~\ref{sect-model-prob} gives the results for the model problems in the whole space, a half space, a quarter space. Using a well-known localization technique, we show maximal regularity of the principal linearization in Section~\ref{sect-MR}. Finally, Section~\ref{sect-nonlinear} proves our main result, Theorem~\ref{th-main}, for the free boundary problem~\eqref{eq-main} with the help of the result obtained in Section~\ref{sect-MR}. In the appendix, we collect technical results needed to execute the above program.
\subsection*{Notation}
As usual, $\BN$, $\BR$, $\BC$ denote the set of all natural, real, and complex number, respectively. Moreover, we also denote $\BR_+ = (0, \infty)$. \par
For $m \in \BN$ and $D \subset \BR^n$, $n = 1, 2, 3$, we set $\ba \cdot \bb = \langle \ba, \bb \rangle = \sum_{j = 1}^m a_j b_j$ for $m$-vectors $\ba = (a_1, \dots, a_m)^\top$ and $\bb = (b_1, \cdots, b_m)^\top$, while we set $(\bff \mid \bg)_D = \int_D \bff (x) \cdot \bg (x) \dx$ for $m$-vector functions $\bff (x) = (f_1 (x), \dots, f_m (x))^\top$ and $\bg (x) = (g_1 (x), \dots, g_m (x))^\top$ on $D$. Besides, $C^\infty_c (D)$ stands for the set of all $C^\infty$-functions on $\BR^n$ whose supports are compact and contained in $D \subset \BR^n$. \par
For $p, q \in [1, \infty]$, $s > 0$, and $D \subset \BR^n$, let $L^q (D)$, $H^{s, q} (D)$, $B^s_{p, q} (D)$, $F^s_{p, q} (D)$ denote the standard $\BK$-valued Lebesgue, Bessel potential, inhomogeneous Besov spaces, and inhomogeneous Triebel-Lizorkin spaces on $D$, respectively, where $\BK \in \{\BR, \BC\}$. For $I \subset \BR$ and $p \in (1, \infty]$, let $L^p (I; X)$ and $H^{1, p} (I; X)$ be the $X$-valued Lebesgue and Sobolev spaces on $I$, respectively. Furthermore, for $p \in (1, \infty)$ and $\delta \in (1\slash p, 1]$, we set
\begin{equation*}
\begin{split}
L^p_\delta (I; X) & := \{f \colon I \to X \mid t^{1 - \delta} f \in L^q (I; X)\}, \\
H^{1, p}_\delta (I; X) & := \{f \in L^p_\delta (I; X) \cap H^{1, 1} (I; X) \mid \pd_t f \in L^p_\delta (I; X)\}.
\end{split}
\end{equation*}
For $p \in (1, \infty)$, $q \in [1, \infty]$, $s \in \BR$, we write the $X$-valued Triebel-Lizorkin spaces with the power wight $\lvert t \rvert^{p(1 - \delta)}$ by $F^s_{p, q, \delta} (I; X)$, where $X$ is a Banach space. Furthermore, ${}_0 H^{1, p}_\delta (I; X)$, and ${}_0 F^s_{p, q, \delta} (I; X)$ denote the subspaces of $H^{1, p}_\delta (I; X)$ and $F^s_{p, q, \delta} (I; X)$, respectively, consisting of all functions having a vanishing trace at $t = 0$, whenever they exist. For $0 < s < 1$, the $X$-valued Bessel-potential spaces $H^{s, p}_\delta (I; X)$ as well as ${}_0 H^{s, p}_\delta (I; X)$ are defined in an analogous way, employing the
complex interpolation method. Namely,
\begin{equation*}
\begin{split}
H^{s, p}_\delta (I; X) & = [L^p_\delta (I; X), H^{1, p}_\delta (I; X)]_s, \\
{}_0 H^{s, p}_\delta (I; X) & = [L^p_\delta (I; X), {}_0 H^{1, p}_\delta (I; X)]_s,
\end{split}
\end{equation*}
where $[\,\cdot, \cdot\,]_\theta$ is the complex interpolation functor with $0 < \theta < 1$. For the precise definition, see, e.g., \cite{MV14}. Here, the following characterizations are known (cf. Pr\"uss and Simonett~\cite[Ch.~3,~6]{PS16}):
\begin{equation*}
\begin{aligned}
f & \in {}_0 H^{1, p}_\delta (\BR_+; X) & \quad & \Leftrightarrow & \quad t^{1 - \delta} f & \in {}_0 H^{1, p} (\BR_+; X), \\
f & \in {}_0 H^{s, p}_\delta (\BR_+; X) & \quad & \Leftrightarrow & \quad t^{1 - \delta} f & \in {}_0 H^{s, p} (\BR_+; X) \\
f & \in {}_0 F^s_{p, q, \delta} (\BR_+; X) & \quad & \Leftrightarrow & \quad t^{1 - \delta}_+ f & \in F^s_{p, q, 1} (\BR; X),
\end{aligned}
\end{equation*}
where we have set $t_+^{1 - \delta} = \max\{t^{1 - \delta}, 0\}$. The symbol $\mathrm{BUC} (I; X)$ stands for the Banach space of all $X$-valued bounded uniformly continuous functions on $I$. In addition, $\mathrm{BUC}^m (I; X)$ is the subset of $\mathrm{BUC} (I; X)$ that has bounded partial derivatives up to order $m \in \BN$. Here, $\mathrm{BUC} (D)$ and $\mathrm{BUC}^m (D)$ are defined similarly as above. In addition, $\mathrm{UC} (I; X)$ denotes the Banach space of all $X$-valued uniformly continuous functions on $I$. For further information on function spaces, we refer to \cite{MV14,PS16,Sbook,T95} and references therein.

\section{Reduction to a fixed reference configuration}
\label{sect-transform}
\noindent
In general, if we study the well-posedness of a free boundary problem it is required to transform the free boundary problem into a domain with a fixed boundary problem since the free boundary is a priori unknown. \par
Suppose that the free interface at time $t$ is given as a graph over the fixed interface $\Gamma_* := \ov{D} \times \{0\}$. Namely, we suppose that there exists a height function $\eta \colon \Gamma_* \times J \to (- H, H)$ such that
\begin{equation}
\label{free-boundary-gamma}
\Gamma (t) = \Gamma_t := \{x \in \Gamma_* \times (- H, H) \mid (x_1, x_2) \in \Gamma_*, \enskip x_3 = \eta (x_1, x_2, t) \}
\end{equation}
for all $t \in J$, where $J = (0, T)$, $T > 0$. Let $\chi \in C^\infty_c (\BR; [0, 1])$ be a bump function such that $\chi (s) = 1$ for $\lvert s \rvert \le r\slash2$ and $\chi (s) = 0$ for $\lvert s \rvert \ge r$, where $0 < r \le H\slash3$. We then define a mapping
\begin{equation*}
\begin{split}
& \Theta_\eta \colon \ov{D} \times (- H, 0) \times J \to \bigcup_{t \in J} \Omega_t \times \{t\}, \\
& \Theta_\eta (x, t) := x + \chi (x_3) \eta (x_1, x_2, t) e_3 =: x + \boldsymbol{\theta}_\eta (x, t).
\end{split}
\end{equation*}
A simple computation yields
\begin{equation*}
\nabla \Theta_\eta = \bI + \begin{pmatrix}
0 & 0 & 0 \\
0 & 0 & 0 \\
(\pd_{x_1} \eta) \chi & (\pd_{x_2} \eta) \chi & \eta \chi'
\end{pmatrix} =: \bI + \boldsymbol{\theta}_\eta'.
\end{equation*}
Hence, if the value of $\lvert \eta \chi' \rvert_{L^\infty (\Gamma_* \times J)}$ is sufficiently small, we see that $\nabla \Theta_\eta$ is a regular matrix and thus $\Theta_\eta$ is invertible. For example, we can achieve this investigation if it holds $\lvert \eta \rvert_{L^\infty (\Gamma_* \times J)} \le (2 \lvert \chi' \rvert_{L^\infty (\BR)})^{- 1}$. We remark that $\lvert \chi' \rvert_{L^\infty (\BR)}$ is bounded by a constant depending only on $r$, and thus $\lvert \eta \rvert_{L^\infty (\Gamma_* \times J)}$ is bounded by a constant depending only on $r$. In the following, we fix $\chi$, choose $r_0 \in (0, (2 \lvert \chi' \rvert_{L^\infty (\BR)})^{- 1})$ suitably small, and suppose $\lvert \eta \rvert_{L^\infty (\Gamma_* \times J)} \le r_0$. Under the conditions stated above, we find that the inverse $\Theta_\eta^{- 1}$ of $\Theta_\eta$ is well-defined and it transforms the free interface $\Gamma_t$ to the flat interface $\Gamma_*$. \par
We denote the \textit{pull-back} of $(\bv, \fp - p_\infty)$ by $(\bu, \pi)$, that is,
\begin{equation*}
\bu (x, t) = \bv (\Theta_\eta (x, t), t), \quad \pi (x, t) = \fp (\Theta_\eta (x, t), t) - p_\infty
\end{equation*}
for $x \in \ov{D} \times [- H, H]$ and $t \in [0, \infty)$. Then, we can compute
\begin{equation*}
\begin{split}
[\nabla \fp] \circ \Theta_\eta & = \nabla_x \pi - M_0 (\eta) \nabla_x \pi, \\
[\dv \bv] \circ \Theta_\eta & = \dv_{\!x} \bu - M_0 (\eta) \dv_{\! x} \bu, \\
[\Delta \bv] \circ \Theta_\eta & = \Delta_x \bu - M_1 (\eta) \colon \nabla^2_x \bu - M_2 (\eta) \nabla_x \bu, \\
[\pd_t \bv] \circ \Theta_\eta & = \pd_t \bu - \chi \pd_t \eta (1 + \chi' \eta)^{- 1} \pd_{x_3} \bu,
\end{split}
\end{equation*}
where we have set
\begin{equation*}
\begin{split}
M_0 (\eta) & := \boldsymbol{\theta}_\eta' (\bI + \boldsymbol{\theta}_\eta')^{- 1}, \\
M_1 (\eta) \colon \nabla^2 \bu & := \Big[2 \mathrm{sym} (\boldsymbol{\theta}_\eta' (\bI + \boldsymbol{\theta}_\eta')^{- 1}) - (\bI + \boldsymbol{\theta}_\eta')^{- 1} \boldsymbol{\theta}_\eta'^\top \boldsymbol{\theta}_\eta' (\bI + \boldsymbol{\theta}_\eta')^{- 1}\Big] \colon \nabla_x^2 \bu, \\
M_2 (\eta) \nabla \bu & := ((\Delta_x \Theta_\eta^{- 1}) \circ \Theta_\eta) \dv \bu.
\end{split}
\end{equation*}
Here, $\mathrm{sym}$ denotes the symmetric part of a matrix. Notice that the similar calculations can be found in, e.g., K\"ohne et al.~\cite[Sec.~2]{KPW13}, Pr\"uss and Simonett~\cite[Ch.~2]{PS16}, and Wilke~\cite[Sec.~1.1]{W17}. Besides, the assumption~\eqref{free-boundary-gamma} implies $V_{\Gamma_t} = (\pd_t \Theta_\eta) \cdot \bn_{\Gamma_t} = (\pd_t \eta) \slash \sqrt{1 + \lvert \nabla_{x'} \eta \rvert^2}$. Hence, we see that $(\bu, \pi, \eta)$ satisfies
\begin{align}
\label{eq-fixed}
\left\{
\begin{aligned}
\pd_t \bu - \mu \Delta \bu + \nabla \pi & = \bF (\bu, \pi, \eta), & \quad & \text{in $\Omega_* \times (0, T)$,} \\
\dv \bu & = F_\mathrm{div} (\bu, \eta), & \quad & \text{in $\Omega_* \times (0, T)$,} \\
\pd_t \eta - u_3 & = D (\bu, \eta), & \quad & \text{in $\Gamma_* \times (0, T)$,} \\
\mu (\pd_3 u_m + \pd_m u_3) & = K_m (\bu, \eta), & \quad & \text{in $\Gamma_* \times (0, T)$,} \\
2 \mu \pd_3 u_3 - \pi - \sigma \Delta_{\Gamma_*} \eta & = K_3 (\bu, \eta), & \quad & \text{in $\Gamma_* \times (0, T)$,} \\
P_{\Sigma_*} (2 \mu \bD (\bu) \bn_{\Sigma_*}) & = \bG (\bu, \eta) & \quad & \text{on $\Sigma_* \times (0, T)$,} \\
\bu \cdot \bn_{\Sigma_*} & = 0 & \quad & \text{on $\Sigma_* \times (0, T)$,} \\
\mu (\pd_3 u_m + \pd_m u_3) & = H_m (\bu, \eta) & \quad & \text{on $B \times (0, T)$,} \\
u_3 & = 0 & \quad & \text{on $B \times (0, T)$,} \\
(\nabla_{\pd D} \eta) \cdot \bn_{\pd D} & = 0, & \quad & \text{on $S_* \times (0, T)$,} \\
\bu (0) & = \bu_0 & \quad & \text{in $\Omega_*$}, \\ 
\eta (0) & = \eta_0, & \quad & \text{on $\Gamma_*$}
\end{aligned}
\right.
\end{align}
with
\begin{equation}
\begin{split}
\bF (\bu, \pi, \eta) & := \chi \pd_t \eta (1 + \chi' \eta)^{- 1} \pd_{x_3} \bu - \mu (M_1 (\eta) \colon \nabla^2 \bu + M_2 (\eta) \nabla \bu) + M_0 (\eta) \nabla \pi, \\
F_\mathrm{div} (\bu, \eta) & := M_0 (\eta) \dv \bu, \\
D (\bu, \eta) & := - \bu' \cdot \nabla_{x'} \eta, \\
K_m (\bu, \eta) & := 2 \mu \bD (\bu) \pd_{x_m} \eta - \lvert \nabla_{x'} \eta \rvert^2 \mu \pd_{x_3} u_m - \Big((1 + \lvert \nabla_{x'} \eta \rvert^2) \mu \pd_{x_3} u_3 - \nabla_{x'} \eta \cdot (\mu \nabla_{x'} u_3)\Big) \nabla_{x'} \eta, \\
K_3 (\bu, \eta) & := \nabla_{x'} \eta \cdot (\mu \nabla_{x'} u_3) + \sum_{m = 1}^2 \pd_{x_m} \eta (\mu \nabla_{x'} u_m) - \lvert \nabla_{x'} \eta \rvert^2 \mu \pd_{x_3} u_3 \\
& \qquad + \sigma \bigg\{\dv_{\!x'} \bigg(\frac{\nabla_{x'} \eta}{\sqrt{1 + \lvert \nabla_{x'} \eta \rvert^2}}\bigg) - \Delta_{x'} \eta \bigg\}, \\
G (\bu, \eta) & := P_{\Sigma_*} \Big(\mu (M_0 (\eta) \nabla \bu + (\nabla \bu)^\top (M_0 (\eta))^\top) \bn_{\Sigma_*} \Big), \\
H_m (\bu, \eta) & := - K_m (\bu, \eta)
\end{split}
\end{equation}
for $m = 1, 2$, where $\bu' = (u_1, u_2)^\top$. Here, $T$ is a positive constant and we have used the assumption $p_0 = p_\infty$ and set
\begin{equation}
\label{def-domain}
\begin{split}
\Omega_* & := \{(x_1, x_2, x_3) \mid (x_1, x_2) \in D, \enskip x_3 \in (- H, 0)\}, \\
\Gamma_* & := \{(x_1, x_2, x_3) \mid (x_1, x_2) \in D, \enskip x_3 = 0\}, \\
\Sigma_* & := \{(x_1, x_2, x_3) \mid (x_1, x_2) \in \pd D, \enskip x_3 \in (- H, 0)\}, \\
B & :=  \{(x_1, x_2, x_3) \mid (x_1, x_2) \in D, \enskip x_3 = - H\}, \\
S_* & := \{(x_1, x_2, x_3) \mid (x_1, x_2) \in \pd D, \enskip x_3 = 0\},
\end{split}
\end{equation}
Notice that the right-hand members in this system are nonlinear and lower order terms. Our aim in this paper is to consider the reformulation of the transformed problem \eqref{eq-fixed} in abstract form $L z = N (z)$ with $z = (\bu, \pi, \eta)$, where $L$ is said to be the \textit{principal linearization}. In order to follow the strategy due to K\"ohne et al.~\cite{KPW13} (and see also Pr\"uss and Simonett~\cite{PS16}), we first show that $L$ has maximal regularity.

\section{Model problems}
\label{sect-model-prob}
\noindent
The proof of maximal regularity of the principal linearization $L$ is based on a localization technique. In fact, using change of coordinates, we may reduce the problem to the following types of model problems:
\begin{enumerate}
\renewcommand{\labelenumi}{(\roman{enumi})}
\item The Stokes equations in the whole space (without any boundary conditions);
\item The Stokes equations in a half space with slip boundary conditions;
\item The Stokes equations in a half space with free boundary conditions;
\item The Stokes equations in a quarter space with slip boundary conditions;
\item The Stokes equations in a quarter space with slip boundary conditions on one part of the boundary and free boundary conditions on the other part.	
\end{enumerate}
Here and in the following, a wedge domain with an angle equal to $\pi\slash2$ is said to be a \textit{quarter space}. Details of a localization procedure will be left to the next section.  \par
This section aims to state the maximal $L^p - L^q$-regularity properties for the linearized problems (i)--(v). To this end, we define
\begin{equation*}
\dot H^{1, q} (D) := \{w \in L^1_\mathrm{loc} (D) \mid \nabla w \in L^q (D)\}
\end{equation*}
for $1 < q < \infty$ and a domain $D \subset \BR^3$. Besides, for $\CS \subset \pd D$, we define
\begin{equation*}
\dot H^{1, q}_\CS (D) := \{w \in L^1_\mathrm{loc} (D) \mid \nabla w \in L^q (D), \enskip w = 0 \enskip \text{on $\CS$}\};
\end{equation*}
in particular, $\dot H^{1, q}_\emptyset (D) := \dot H^{1, q} (D)$. Then, the space $\dot H^{- 1, q}_\CS (D)$ is defined as
\begin{equation*}
\dot H^{- 1, q}_\CS (D) := \Big(\dot H^{1, q'}_{\pd D \backslash \CS} (D)\Big)^*
\end{equation*}
with conventions $\dot H^{- 1, q} (D) = \dot H^{- 1, q}_\emptyset (D)$ and ${}_0 \dot H^{- 1, q} (D) = \dot H^{- 1, q}_{\pd D} (D)$. \par
For simplicity of notation, we write $\pd_j$ instead of $\pd\slash\pd x_j$, $j = 1, 2, 3$, if there is no confusion. 
\subsection{The Stokes equations in the whole space}
Consider the problem
\begin{equation}
\label{eq-whole}
\left\{\begin{split}
\pd_t \bu - \mu \Delta \bu + \nabla \pi & = \bff, & \quad & \text{in $\BR^3 \times J$}, \\
\dv \bu & = f_\mathrm{div}, & \quad & \text{in $\BR^3 \times J$}, \\
\bu (0) & = \bu_0, & \quad & \text{in $\BR^3$}.
\end{split}\right.
\end{equation}
Here and in the following, we use the notation $J = (0, T)$ for $0 < T < \infty$. According to Pr{\" u}ss and Simonett~\cite[Sec. 7.1]{PS16}, we know the following theorem.
\begin{theo}
\label{th-model-whole}
Let $1 < p, q < \infty$, $1\slash p < \delta \le 1$, and $T > 0$. The problem \eqref{eq-whole} has a unique solution $(\bu, \pi)$ satisfying
\begin{equation*}
\begin{split}
\bu & \in H^{1, p}_\delta (J; L^q (\BR^3)^3) \cap L^p_\delta (J; H^{2, q} (\BR^3)^3), \\
\pi & \in L^p_\delta (J; \dot H^{1, q} (\BR^3)),
\end{split}
\end{equation*}
if and only if the data $(\bff, f_\mathrm{div}, \bu_0)$ satisfy the following:
\begin{enumerate}
\renewcommand{\labelenumi}{(\alph{enumi})}
\item $\bff \in L^p_\delta (J; L^q (\BR^3)^3)$;
\item $f_\mathrm{div} \in H^{1, p}_\delta (J; \dot H^{- 1, q} (\BR^3)) \cap L^p_\delta (J; H^{1, q} (\BR^3))$;
\item $\bu_0 \in B^{2 (\delta - 1\slash p)}_{q, p} (\BR^3)^3$ and $\dv \bu_0 = f_\mathrm{div} (0)$.	
\end{enumerate}
The solution $(\bu, \pi)$ depends continuously on the data in the corresponding spaces.	
\end{theo}
\subsection{The Stokes equations in a half space}
We next consider the Stokes equations in a half space 
\begin{equation}
\label{eq-stokes-half-noslip}
\left\{\begin{split}
\pd_t \bu - \mu \Delta \bu + \nabla \pi & = \bff, & \quad & \text{in $\BR^3_+ \times J$}, \\
\dv \bu & = f_\mathrm{div}, & \quad & \text{in $\BR^3_+ \times J$}, \\
\bu & = \bg, & \quad & \text{on $\pd \BR^3_+ \times J$}, \\
\bu (0) & = \bu_0, & \quad & \text{in $\BR^3_+$}.
\end{split}\right.
\end{equation}
\begin{equation}
\label{eq-stokes-half-slip}
\left\{\begin{split}
\pd_t \bu - \mu \Delta \bu + \nabla \pi & = \bff, & \quad & \text{in $\BR^3_+ \times J$}, \\
\dv \bu & = f_\mathrm{div}, & \quad & \text{in $\BR^3_+ \times J$}, \\
\mu (\pd_3 u_m + \pd_m u_3) & = h_m, & \quad & \text{on $\pd \BR^3_+ \times J$}, \\
u_3 & = h_3, & \quad & \text{on $\pd \BR^3_+ \times J$}, \\
\bu (0) & = \bu_0, & \quad & \text{in $\BR^3_+$},
\end{split}\right.
\end{equation}
where $m = 1, 2$ and we have used the notation
\begin{equation*}
\begin{split}
\BR^3_+ & := \{x = (x_1, x_2, x_3) \mid (x_1, x_2) \in \BR^2, \enskip x_3 > 0\}, \\
\pd \BR^3_+ & := \{x = (x_1, x_2, x_3) \mid (x_1, x_2) \in \BR^2, \enskip x_3 = 0\}. 
\end{split}
\end{equation*}
Besides, in order to describe the compatibility conditions for the problems \eqref{eq-stokes-half-noslip} and \eqref{eq-stokes-half-slip}, we introduce the space $\wh H^{- 1, q} (\BR^3_+)$ as the set of all $(\varphi_1, \varphi_2) \in L^q (\BR^3_+) \times B^{2 - 1\slash q}_{q, q} (\pd \BR^3_+)^3$ that satisfy the regularity property $(\varphi_1, \varphi_2 \cdot \bn_{\pd \BR^3_+}) \in \dot H^{- 1, q} (\BR^3_+)$. If we adopt the notation
\begin{equation*}
\langle (\varphi_1, \varphi_2 \cdot \bn_{\pd \BR^3_+}) \mid \phi \rangle_{\BR^3_+} := - (\varphi_1 \mid \phi)_{\BR^3_+} + (\varphi_2 \cdot \bn_{\pd \BR^3_+} \mid \phi)_{\pd \BR^3_+} \quad \text{for any $\phi \in \dot H^{1, q'} (\BR^3_+)$},
\end{equation*}
then from the divergence equation we have the conditions
\begin{equation*}
\begin{split}
\langle (f_\mathrm{div}, \bg \cdot \bn_{\pd \BR^3_+}) \mid \phi \rangle_{\BR^3_+} & = - (\bu \mid \nabla \phi)_{\BR^3_+}, \\
\langle (f_\mathrm{div}, h_3) \mid \phi \rangle_{\BR^3_+} & = - (\bu \mid \nabla \phi)_{\BR^3_+}, \\
\end{split}
\end{equation*}
for any $\phi \in \dot H^{1, q'} (\BR^3_+)$ when we deal with \eqref{eq-stokes-half-noslip} and \eqref{eq-stokes-half-slip}, respectively. Furthermore, let us introduce function spaces
\begin{equation}
\label{def-spaces}
\begin{split}
\BE_{1, \delta} (J; D) & := H^{1, p}_\delta (J; L^q (D)^3) \cap L^p_\delta (J; H^{2, q} (D)^3), \\
\BE_{2, \delta} (J; D) & := L^p_\delta (J; \dot H^{1, q} (D)), \\
\BF_{0, \delta} (J, D) & := L^p_\delta (J; L^q (D)^3), \\
\BF_{1, \delta} (J; D) & := H^{1, p}_\delta (J; \dot H^{- 1, q} (D)) \cap L^p_\delta (J; H^{1, q} (D)), \\
\BF_{2, \delta} (J; \pd D) & := F^{1\slash2 - 1\slash(2q)}_{p, q, \delta} (J; L^q (\pd D)) \cap L^p_\delta (J; B^{1- 1\slash q}_{q, q} (\pd D)), \\
\BF_{3, \delta} (J; \pd D) & := F^{1 - 1\slash(2q)}_{p, q, \delta} (J; L^q (\pd D)) \cap L^p_\delta (J; B^{2 - 1\slash q}_{q, q} (\pd D))
\end{split}
\end{equation}
for $J \subset \BR_+$ and $D \subset \BR^3$. Then the following theorems are well-known (cf. Pr{\" u}ss and Simonett~\cite[Sec.~7.2]{PS16}).
\begin{theo}
\label{th-model-half-dirichlet}
Let $1 < p, q < \infty$, $1\slash p < \delta \le 1$, $1\slash p + 1\slash(2 q) \ne \delta$, $T > 0$, and $J = (0, T)$. Then there exists a unique solution $(\bu, \pi)$ to the equations \eqref{eq-stokes-half-noslip} with regularity $\bu \in \BE_{1, \delta} (J; \BR^3_+)$ and $\pi \in \BE_{2, \delta} (J; \BR^3_+)$
if and only if
\begin{enumerate}
\renewcommand{\labelenumi}{(\alph{enumi})}
\item $\bff \in \BF_{0, \delta} (J; \BR^3_+)$;
\item $f_\mathrm{div} \in \BF_{1, \delta} (J; \BR^3_+)$;
\item $\bg \in \BF^3_{3, \delta} (J; \pd \BR^3_+)$ and $\bg (0) = \mathrm{Tr}_{\pd \BR^3_+} [\bu_0]$ if $1\slash p + 1\slash(2 q) < \delta$;
\item $(f_\mathrm{div}, \bg \cdot \bn_{\pd \BR^3_+}) \in H^{1, p}_\delta (J; \wh H^{- 1, q} (\BR^3_+))$;
\item $\bu_0 \in B^{2 (\delta - 1\slash p)}_{q, p} (\BR^3_+)^3$ and $\dv \bu_0 = f_\mathrm{div} (0)$.		
\end{enumerate}
Furthermore, the solution $(\bu, \pi)$ depends continuously on the data in the corresponding spaces.	
\end{theo}
\begin{theo}
\label{th-model-half-slip}
Suppose $1 < p, q < \infty$, $1\slash p < \delta \le 1$, $1\slash p + 1\slash(2 q) \notin \{\delta - 1\slash2, \delta\}$, $T > 0$, and $J = (0, T)$. The equations~\eqref{eq-stokes-half-slip} admits a unique solution $(\bu, \pi)$ in the class $\bu \in \BE_{1, \delta} (J; \BR^3_+)$, $\pi \in \BE_{2, \delta} (J; \BR^3_+)$ if and only if
\begin{enumerate}
\renewcommand{\labelenumi}{(\alph{enumi})}
\item $\bff \in \BF_{0, \delta} (J; \BR^3_+)$;
\item $f_\mathrm{div} \in \BF_{1, \delta} (J; \BR^3_+)$;
\item $h_m \in \BF_{2, \delta} (J; \pd \BR^3_+)$ and $h_m (0) = \mathrm{Tr}_{\pd \BR^3_+} [\mu (\pd_3 u_{0, m} + \pd_m u_{0, 3})]$ for $m = 1, 2$ if $1\slash p + 1\slash(2 q) < \delta - 1\slash2$;
\item $h_3 \in \BF_{3, \delta} (J; \pd \BR^3_+)$ and $h_3 (0) = \mathrm{Tr}_{\pd \BR^3_+} [u_{0, 3}]$ for $m = 1, 2$ if $1\slash p + 1\slash(2 q) < \delta$;
\item $(f_\mathrm{div}, h_3) \in H^{1, p}_\delta (J; \wh H^{- 1, q} (\BR^3_+))$;
\item $\bu_0 \in B^{2 (\delta - 1\slash p)}_{q, p} (\BR^3_+)^3$ and $\dv \bu_0 = f_\mathrm{div} (0)$.			
\end{enumerate}
In addition, the solution map $(\bff, f_\mathrm{div}, \bh, \bu_0) \mapsto (\bu, \pi)$ is continuous between the corresponding spaces.	
\end{theo}

\begin{rema}
In view of trace theorems, the statements in Pr{\" u}ss and Simonett~\cite[Thm.~7.2.1]{PS16} should be corrected as above. See, e.g., \cite{L19} for the details.
\end{rema}

Consider the Stokes equations with free boundary conditions
\begin{equation}
\label{eq-stokes-half-free}
\left\{\begin{split}
\pd_t \bu - \mu \Delta \bu + \nabla \pi & = \bff, & \quad & \text{in $\BR^3_+ \times J$}, \\
\dv \bu & = f_\mathrm{div}, & \quad & \text{in $\BR^3_+ \times J$}, \\
\pd_t \eta - u_3 & = d, & \quad & \text{on $\pd \BR^3_+ \times J$}, \\
\mu (\pd_3 u_m + \pd_m u_3) & = k_m, & \quad & \text{on $\pd \BR^3_+ \times J$}, \\
2 \mu \pd_3 u_3 - \pi - \sigma \Delta_{\pd \BR^3_+} \eta & = k_3, & \quad & \text{on $\pd \BR^3_+ \times J$}, \\
\bu (0) & = \bu_0, & \quad & \text{in $\BR^3_+$}, \\
\eta (0) & = \eta_0, & \quad & \text{on $\pd \BR^3_+$},
\end{split}\right.
\end{equation}
where $m = 1, 2$ and $\Delta_{\pd \BR^3_+} = \sum_{j = 1}^2 \pd_j^2$. The maximal $L^p - L^q$ regularity theorem for this problem has been studied by Shibata~\cite{S16,S20} (cf. Shibata 
and Shimizu~\cite{SS12}). However, we will show that more optimal regularity results can be obtained. To this end, we introduce a function space
\begin{equation*}
\begin{split}
\BE_{3, \delta} (J; \pd D) & := F^{1\slash2 - 1\slash(2q)}_{p, q, \delta} (J; L^q (\pd D)) \cap L^p_\delta (J; B^{1 - 1\slash q}_{q, q} (\pd D)), \\
\BE_{4, \delta} (J; \pd D) & := F^{2 - 1\slash(2 q)}_{p, q, \delta} (J; L^q (\pd D)) \cap H^{1, p}_\delta (J; B^{2 - 1\slash q}_{q, q} (\pd D)) \cap L^p_\delta (J; B^{3 - 1\slash q}_{q, q} (\pd D))
\end{split}
\end{equation*}
for $J \subset \BR_+$ and $D \subset \BR^3$.
\begin{theo}
\label{th-model-half-free}
Let $1 < p, q <\infty$, $1\slash p < \delta \le 1$, $1\slash p + 1\slash(2 q) \ne \delta$, $T > 0$, and $J = (0, T)$.  The problem~\eqref{eq-stokes-half-free} has a unique solution $(\bu, \pi, \eta)$ with $\bu \in \BE_{1, \delta} (J; \BR^3_+)$, $\pi \in \BE_{2, \delta} (J; \BR^3_+)$, $\mathrm{Tr}_{\pd \BR^3_+} [\pi] \in \BE_{3, \delta} (J; \pd \BR^3_+)$, and $\eta \in \BE_{4, \delta} (J; \pd \BR^3_+)$ if and only if
\begin{enumerate}
\renewcommand{\labelenumi}{(\alph{enumi})}
\item $\bff \in \BF_{0, \delta} (J; \BR^3_+)$;
\item $f_\mathrm{div} \in \BF_{1, \delta} (J; \BR^3_+)$;
\item $d \in \BF_{3, \delta} (J; \pd \BR^3_+)$;
\item $k_j \in \BF_{2, \delta} (J; \pd \BR^3_+)$ for $j = 1, 2, 3$ and $k_m (0) = \mathrm{Tr}_{\pd \BR^3_+} [\mu (\pd_3 u_{0, m} + \pd_m u_{0, 3})]$ for $m = 1, 2$ if $1\slash p + 1\slash(2 q) < \delta - 1\slash2$;
\item $\bu_0 \in B^{2 (\delta - 1\slash p)}_{q, p} (\BR^3_+)^3$ and $\dv \bu_0 = f_\mathrm{div} (0)$;
\item $\eta_0 \in B^{2 + \delta - 1\slash p - 1\slash q}_{q, p} (\pd \BR^3_+)$.	
\end{enumerate}
Besides, the solution $(\bu, \pi, h)$ depends continuously on the data in the corresponding spaces.	
\end{theo}
\begin{proof}
We only prove the sufficient part since the necessary part immediately follows from trace theorems. In order to show the sufficient part, we consider the following shifted problem:
\begin{equation}
\label{eq-stokes-half-free-shift}
\left\{\begin{split}
\pd_t \bu + \omega \bu - \mu \Delta \bu + \nabla \pi & = \bff, & \quad & \text{in $\BR^3_+ \times \BR_+$}, \\
\dv \bu & = f_\mathrm{div}, & \quad & \text{in $\BR^3_+ \times \BR_+$}, \\
\pd_t \eta + \omega \eta - u_3 & = d, & \quad & \text{on $\pd \BR^3_+ \times \BR_+$}, \\
\mu (\pd_3 u_m + \pd_m u_3) & = k_m, & \quad & \text{on $\pd \BR^3_+ \times \BR_+$}, \\
2 \mu \pd_3 u_3 - \pi - \sigma \Delta_{\pd \BR^3_+} \eta & = k_3, & \quad & \text{on $\pd \BR^3_+ \times \BR_+$}, \\
\bu (0) & = \bu_0, & \quad & \text{in $\BR^3_+$}, \\
\eta (0) & = \eta_0, & \quad & \text{on $\pd \BR^3_+$},
\end{split}\right.
\end{equation}
where $m = 1, 2$ and $\omega > 0$ denotes a (possibly large) shift parameter. We first show that there exists $\omega_0 > 0$ such that for each $\omega \ge \omega_0$, the system \eqref{eq-stokes-half-free-shift} admits a unique solution $(\bu, \pi, \eta)$ in the corresponding regularity classes, and then we will prove that \eqref{eq-stokes-half-free} has a unique solution in the right regularity classes --- this is due to the fact that several trace theorems can be applied not for the finite interval $J$ but for the semi-infinite interval $\BR_+$. However, it is known that we can drop off the parameter $\omega$ by restricting the time interval to be finite, and thus it is sufficient to consider \eqref{eq-stokes-half-free-shift}. \par
Without loss of generality, we may assume that $(\bff, f_\mathrm{div}, d, \bu_0, \eta_0) = 0$ and $\bk (0) = 0$ if $1\slash p + 1\slash(2 q) < \delta$. This can be observed as follows: Let us first consider the problem with $\omega = 0$. Define
\begin{equation*}
\begin{split}
\eta_1 (t) & := \Big[2 e^{- (\bI - \Delta_{\pd \BR^3_+})^{1\slash2} t} - e^{- 2 (\bI - \Delta_{\pd \BR^3_+})^{1\slash2} t} \Big] \eta_0 \\
& \quad + \Big[e^{- (\bI - \Delta_{\pd \BR^3_+}) t} - e^{- 2 (\bI - \Delta_{\pd \BR^3_+}) t} \Big] (\bI - \Delta_{\pd \BR^3_+})^{- 1} (u_{0, 3} \vert_{\pd \BR^3_+} + d \vert_{t = 0}) \\
& = \eta_{1, 1} (t) + \eta_{1, 2} (t)
\end{split}
\end{equation*}
for any $t \ge 0$. Since the operators $\mathrm{Tr}_{t = 0} \circ \mathrm{Tr}_{\pd \BR^3_+}$ and $\mathrm{Tr}_{\pd \BR^3_+} \circ \mathrm{Tr}_{t = 0}$ coincide (cf. Lindemulder~\cite[pp.~88]{L19}), we have $d \vert_{t = 0} \in B^{2 (\delta - 1\slash p) - 1\slash q}_{q, p} (\pd \BR^3_+)$. Since $\eta_0 \in B^{2 + \delta - 1\slash p - 1\slash q}_{q, p} (\pd \BR^3_+)$ and $u_{0, 3} \vert_{\pd \BR^3_+} \in B^{2 (\delta - 1\slash p) - 1\slash q}_{q, p} (\pd \BR^3_+)$, by the standard semigroup theory, we see that 
\begin{equation*}
\eta_1 \in \BE_{4, \delta} (\BR_+; \pd \BR^3_+)
\end{equation*}
with $\eta_1 (0) = \eta_0$ and $\pd_t \eta_1 (0) = u_{0, 3} \vert_{\pd \BR^3_+} + d \vert_{t = 0}$. Indeed, we observe that $\eta_{1, 1} (0) = \eta_0$, $\pd_t \eta_{1, 1} (0) = 0$ and $\eta_{1, 2} (0) = 0$, $\pd_t \eta_{1, 2} (0) = u_0 \vert_{\pd \BR^3_+} + d \vert_{t = 0}$. Hence, by \cite[Thm.~4.2]{MV14}, it holds
\begin{equation*}
\eta_{1, 1} \in F^{3 - 1\slash q}_{p, 1, \delta} (\BR_+; L^q (\pd \BR^3_+)) \cap F^0_{p, 1, \delta} (\BR_+; B^{3 - 1\slash q}_{q, 1} (\pd \BR^3_+)) \hookrightarrow \BE_{4, \delta} (\BR_+; \pd \BR^3_+)
\end{equation*}
due to $\eta_0 \in B^{2 + \delta - 1\slash p - 1\slash q}_{q, p} (\pd \BR^3_+)$. Besides, by $(\bI - \Delta_{\pd \BR^3_+})^{- 1} u_{0, 3} \vert_{\pd \BR^3_+} + d \vert_{t = 0} \in B^{2 + 2 (\delta - 1\slash p) - 1\slash q}_{q, p} (\pd \BR^3_+)$, we also have
\begin{equation*}
\eta_{1, 2} \in F^{2 - 1\slash(2 q)}_{p, 1, \delta} (\BR_+; L^q (\pd \BR^3_+)) \cap F^0_{p, 1, \delta} (\BR_+; B^{4 - 1\slash q}_{q, q} (\pd \BR^3_+)) \hookrightarrow \BE_{4, \delta} (\BR_+; \pd \BR^3_+).
\end{equation*}
Hence, we have $\eta_1 \in \BE_{4, \delta} (\BR_+; \pd \BR^3_+)$. \par
We now extend $\bu_0$ to all of $\BR^3$ in the class $B^{2 (\delta - 1\slash p)}_{q, p} (\BR^3)$, and extend $\bff$ by zero, where those are denoted by $e_{\BR^3} [\bu_0]$ and $E^z_{x_3} [\bff]$, respectively. Then for each $\omega > 0$ the problem
\begin{equation*}
\left\{\begin{split}
\pd_t \bu_1 + \omega \bu_1 - \mu \Delta \bu_1 & = E^z_{x_3} [\bff], & \quad & \text{in $\BR^3 \times \BR_+$}, \\
\bu_1 (0) & = e_{\BR^3} [\bu_0], & \quad & \text{in $\BR^3$}
\end{split}\right.
\end{equation*}
has a unique solution $\bu_1 \in H^{1, p}_\delta (\BR_+; L^q (\BR^3)) \cap L^p_\delta (\BR_+; H^{2, q} (\BR^3))$, see, e.g., \cite[Thm.~6.1.8, 4.4.4]{PS16}. Thus $\bu_2 := \bu - \bu_1 \vert_{\BR^3_+}$ and $\pi_2 := \pi$ should solve
\begin{equation*}
\left\{\begin{split}
\pd_t \bu_2 + \omega \bu_2 - \mu \Delta \bu_2 + \nabla \pi_2 & = 0, & \quad & \text{in $\BR^3_+ \times \BR_+$}, \\
\dv \bu_2 & = f_\mathrm{div, 2}, & \quad & \text{in $\BR^3_+ \times \BR_+$}, \\
\bu_2 (0) & = 0, & \quad & \text{in $\BR^3_+$},
\end{split}\right.
\end{equation*}
where we have set $f_\mathrm{div, 2} := f_\mathrm{div} - \dv (\bu_1 \vert_{\BR^3_+}) \in \BF_{1, \delta} (\BR_+; \BR^3_+)$. Here, $f_\mathrm{div, 2}$ has vanishing trace at $t = 0$. Extending $\bu_2$ by even reflection with respect to $x_3$, we consider
\begin{equation*}
\left\{\begin{split}
\pd_t \bu_3 + \omega \bu_3 - \mu \Delta \bu_3 +\nabla \pi_3 & = 0, & \quad & \text{in $\BR^3 \times \BR_+$}, \\
\dv \bu_3 & = E^e_{x_3} [f_\mathrm{div, 2}], & \quad &\text{in $\BR^3 \times \BR_+$}, \\
\bu_3 (0) & = 0, & \quad & \text{in $\BR^3$},
\end{split}\right.
\end{equation*}
where $E^e_{x_3}$ denotes even extension with respect to $x_3$. According to \cite[Sec.~7.1]{PS16}, we can obtain $(\bu_3, \pi_3)$ in the right regularity class. Define $\bu_4 := \bu_2 - \bu_3 \vert_{\BR^3_+}$, $\pi_4 := \pi_2 - \pi_3 \vert_{\BR^3_+}$, and $\eta_4 := \eta - \eta_1$. Then we see that the pair $(\bu_4, \pi_4, \eta_4)$ satisfies the problem \eqref{eq-stokes-half-free-shift} with $(\bff, f_\mathrm{div}, d, \bu_0, \eta_0) = 0$ and the boundary data
\begin{equation*}
\begin{split}
d_4 & = d - \omega \eta_1 + u_{1, 3} \vert_{\BR^3_+} + u_{3, 3} \vert_{\BR^3_+} \\
k_{4, m} & = k_m - \mathrm{Tr}_{\pd \BR^3_+} [\mu (\pd_3 (u_{1, m} \vert_{\BR^3_+} + u_{3, m} \vert_{\BR^3_+}) + \pd_m (u_{1, 3} \vert_{\BR^3_+} + u_{3, 3} \vert_{\BR^3_+}))], \\
k_{4, 3} & = k_3 - \mathrm{Tr}_{\pd \BR^3_+} [(2 \mu \pd_3 (u_{1, 3} \vert_{\BR^3_+} + u_{3, 3} \vert_{\BR^3_+}) - \pi_3 \vert_{\BR^3_+} - \sigma \Delta_{\pd \BR^3_+} \eta_1)].
\end{split}
\end{equation*}
Notice that $\pd_t \eta_4 \vert_{t = 0} = 0$ due to the construction. It is not difficult to observe that $k_{4, 1}$, $k_{4, 2}$, $k_{4, 3}$ have the right regularity and have vanishing traces at $t = 0$ whenever they exist. Summing up, for a fixed parameter $\omega >  0$, we see that $(\bu_4, \pi_4, \eta_4)$ solves
\begin{equation}
\label{eq-stokes-half-free-2}
\left\{\begin{split}
\pd_t \bu_4 + \omega \bu_4 - \mu \Delta \bu_4 + \nabla \pi_4 & = 0, & \quad & \text{in $\BR^3_+ \times \BR_+$}, \\
\dv \bu_4 & = 0, & \quad & \text{in $\BR^3_+ \times \BR_+$}, \\
\pd_t \eta_4 + \omega \eta_4 - u_{4, 3} & = d_4, & \quad & \text{on $\pd \BR^3_+ \times \BR_+$}, \\
\mu (\pd_3 u_{4, m} + \pd_m u_{4, 3}) & = k_{4, m}, & \quad & \text{on $\pd \BR^3_+ \times \BR_+$}, \\
2 \mu \pd_3 u_{4, 3} - \pi_4 - \sigma \Delta_{\pd \BR^3_+} \eta_4 & = k_{4, 3}, & \quad & \text{on $\pd \BR^3_+ \times \BR_+$}, \\
\bu_4 (0) & = 0, & \quad & \text{in $\BR^3_+$}, \\
\eta_4 (0) & = 0, & \quad & \text{on $\pd \BR^3_+$}
\end{split}\right.
\end{equation}
with $m = 1, 2$. This shows that we may suppose $(\bff, f_\mathrm{div}, d, \bu_0, \eta_0) = 0$ and $\bk (0) = 0$ if $1\slash p + 1\slash(2 q) < \delta$. \par
In the following, we suppose $1\slash p + 1\slash(2 q) \ne \delta$. We now consider \eqref{eq-stokes-half-free-shift} with $(\bff, f_\mathrm{div}, d, \bu_0, \eta_0) = 0$, where $\bk$ has vanishing trace at $t = 0$ whenever it exists. Let $E_t$ be an extension operator with respect to $t$. Notice that we can extend $\bk$ trivially to all of $t \in \BR$. We show that the unknowns $\bu$, $\pi$, $\eta$ vanishes on $\{t < 0\}$ if $(\bu, \pi, \eta)$ satisfies
\begin{equation}
\label{eq-stokes-half-free-3}
\left\{\begin{split}
\pd_t \bu + \omega \bu - \mu \Delta \bu + \nabla \pi & = 0, & \quad & \text{in $\BR^3_+ \times \BR$}, \\
\dv \bu & = 0, & \quad & \text{in $\BR^3_+ \times \BR$}, \\
\pd_t \eta + \omega \eta - u_3 & = 0, & \quad & \text{on $\pd \BR^3_+ \times \BR$}, \\
\mu (\pd_3 u_m + \pd_m u_3) & = k_m, & \quad & \text{on $\pd \BR^3_+ \times \BR$}, \\
2 \mu \pd_3 u_3 - \pi - \sigma \Delta_{\pd \BR^3_+} \eta & = k_3, & \quad & \text{on $\pd \BR^3_+ \times \BR$}, \\
\bu (0) & = 0, & \quad & \text{in $\BR^3_+$}, \\
\eta (0) & = 0, & \quad & \text{on $\pd \BR^3_+$}
\end{split}\right.
\end{equation}
with $\bk = 0$ on $\{t < 0\}$. To this end, for $\wt \bff \in L^{p'} (- \infty, T; L^{q'} (\BR^3_+)^3)$ let $(\wt \bu, \wt \pi, \wt \eta)$ be the solution to the \textit{dual backward problem}
\begin{equation*}
\left\{\begin{split}
- \pd_t \wt \bu + \omega \wt \bu - \mu \Delta \wt \bu + \nabla \wt \pi & = \wt \bff, & \quad & \text{in $\BR^3_+ \times (- \infty, T)$}, \\
\dv \wt \bu & = 0, & \quad & \text{in $\BR^3_+ \times (- \infty, T)$}, \\
- \pd_t \wt \eta + \omega \wt \eta - \wt u_3 & = 0, & \quad & \text{on $\pd \BR^3_+ \times (- \infty, T)$}, \\
\mu (\pd_3 \wt u_m + \pd_m \wt u_3) & = 0, & \quad & \text{on $\pd \BR^3_+ \times (- \infty, T)$}, \\
2 \mu \pd_3 \wt u_3 - \wt \pi - \sigma \Delta_{\pd \BR^3_+} \wt \eta & = 0, & \quad & \text{on $\pd \BR^3_+ \times (- \infty, T)$}, \\
\wt \bu (T) & = 0, & \quad & \text{in $\BR^3_+$}, \\
\wt \eta (T) & = 0, & \quad & \text{on $\pd \BR^3_+$}
\end{split}\right.
\end{equation*}
with $T \in \BR$. Integration by parts gives
\begin{equation*}
\begin{split}
0 & = \int_{- \infty}^T (\pd_t \bu + \omega \bu - \mu \Delta \bu + \nabla \pi \mid \wt \bu)_{\BR^3_+} \dt \\
& = \int_{- \infty}^T (\bu \mid - \pd_t \wt \bu + \omega \wt \bu - \mu \Delta \wt \bu)_{\BR^3_+} \dt + \int_{- \infty}^T \Big((u_3 \mid 2 \mu \pd_3 \wt u_3)_{\pd \BR^3_+} - (\sigma\Delta_{\pd \BR^3_+} \eta \mid \wt u_3)_{\pd \BR^3_+} \Big) \dt \\
& = \int_{- \infty}^T (\bu \mid - \pd_t \wt \bu + \omega \wt \bu - \mu \Delta \wt \bu)_{\BR^3_+} \dt + \int_{- \infty}^T \Big((u_3 \mid 2 \mu \pd_3 \wt u_3)_{\pd \BR^3_+} - (\sigma\Delta_{\pd \BR^3_+} \eta \mid - \pd_t \wt \eta + \omega \wt \eta)_{\pd \BR^3_+} \Big) \dt,
\end{split}
\end{equation*}
where we have used the fact $\bu (0) = \wt \bu (T) = 0$ and $\eta (0) = \wt \eta (T) = 0$. We also see that
\begin{equation*}
\begin{split}
\int_{- \infty}^T (\bu \mid \nabla \wt \pi)_{\BR^3_+} \dt & = \int_{- \infty}^T (u_3 \mid \wt \pi)_{\pd \BR^3_+} \dt \\
& = \int_{- \infty}^T (u_3 \mid 2 \mu \pd_3 \wt u_3 - \sigma \Delta_{\pd \BR^3_+} \wt \eta)_{\pd \BR^3_+} \dt \\
& = \int_{- \infty}^T \Big((\pd_t \eta + \omega \eta \mid - \sigma \Delta_{\pd \BR^3_+} \wt \eta)_{\pd \BR^3_+} + (u_3 \mid 2 \mu \pd_3 \wt u_3)_{\pd \BR^3_+} \Big) \dt.
\end{split}
\end{equation*}
Summarizing, we obtain
\begin{equation*}
0 = \int_{- \infty}^T (\bu \mid - \pd_t \wt \bu + \omega \wt \bu - \mu \Delta \wt \bu + \nabla \wt \pi)_{\BR^3_+} \dt = \int_{- \infty}^T (\bu \mid \wt \bff)_{\BR^3_+} \dt.
\end{equation*}
Since $\wt \bff \in L^{p'}_\delta (- \infty, T; L^{q'} (\BR^3_+)^3)$ should be chosen arbitrary, we have $\bu \equiv 0$ on $(-\infty, T)$ for any $T \in \BR$ whenever given data vanishes. Recalling that $\bk$ vanishes on $\{t < 0\}$, we can conclude that $\bu \equiv 0$ on $\{t < 0\}$. In addition, we can also show $\eta = 0$ and $\pi = 0$ on $\{t < 0\}$. Therefore, in the following, let $E_t$ denotes the zero extension operator with respect to $t$. \par
Let $\CL$ be the (bilateral) Laplace transform with respect to a time variable $t$ defined by
\begin{equation*}
\CL [f (t)] (\lambda) = \int_{- \infty}^\infty e^{- \lambda t} f (t) \dt = \int_{- \infty}^\infty e^{- i \tau t} e^{- \gamma t} f (t) \dt = \CF_\tau [e^{- \gamma t} f (t)] (\tau),
\end{equation*}
where $\lambda = \gamma + i \tau \in \BC$ is a parameter with real numbers $\gamma$ and $\tau$. Extending the unknowns to $\{t < 0\}$ trivially by $E_t$, and then applying the Laplace transform to~\eqref{eq-stokes-half-free-shift}, we find
\begin{equation}
\label{eq-stokes-half-free-resolvent}
\left\{\begin{split}
z \CL[E_t[\bu]] - \mu \Delta \CL[E_t[\bu]] + \nabla \CL[E_t[\pi]] & = 0, & \quad & \text{in $\BR^3_+$}, \\
\dv \CL[E_t[\bu]] & = 0, & \quad & \text{in $\BR^3_+$}, \\
z \CL[E_t[\eta]] - \CL[E_t[u_3]] & = 0, & \quad & \text{on $\pd \BR^3_+$}, \\
\mu (\pd_3 \CL[E_t[u_m]] + \pd_m \CL[E_t[u_3]]) & = \CL[E_t [k_{4, m}]
], & \quad & \text{on $\pd \BR^3_+$}, \\
2 \mu \pd_3 \CL[E_t[u_3]] - \CL[E_t[\pi]] - \sigma \Delta_{\pd \BR^3_+} \CL[E_t[\eta]] & = \CL[E_t [k_{4, 3}]], & \quad & \text{on $\pd \BR^3_+$}
\end{split}\right.
\end{equation}
with $z := \lambda + \omega$. To simplify the notation, we may denote $\bK = E_{\pd \BR^3_+} [E_t [\bk_4]]$ and $\bK_\CL = E_{\pd \BR^3_+} [\CL [E_t [\bk_4]]]$. From the result due to Shibata, for any $0 < \theta < \pi\slash2$ and $\CL [E_t [\bk_4]] \in B^{1 - 1\slash q}_{q, q} (\pd \BR^3_+)^3$, there exist a (possibly large) constant $\omega_0 > 0$ and operator families 
\begin{equation*}
\begin{split}
\CU_{\BR^3_+} (z) & \in \Hol (\omega_0 + \Sigma_{\theta + \pi\slash2} \cup \{0\}; \CB (L^q (\BR^3_+)^3 \times H^{1, q} (\BR^3_+)^3, H^{2, q} (\BR^3_+)^3)), \\
\CP_{\BR^3_+} (z) & \in \Hol (\omega_0 + \Sigma_{\theta + \pi\slash2} \cup \{0\}; \CB (L^q (\BR^3_+)^3 \times H^{1, q} (\BR^3_+)^3, \dot H^{1, q} (\BR^3_+))), \\
\CE_{\BR^3_+} (z) & \in \Hol (\omega_0 + \Sigma_{\theta + \pi\slash2} \cup \{0\}; \CB (L^q (\BR^3_+)^3 \times H^{1, q} (\BR^3_+)^3, H^{3, q} (\BR^3_+)))
\end{split}
\end{equation*}
such that for each $\omega > \omega_0$, the triplet
\begin{equation*}
(\CL [E_t[\bu]], \CL [E_t[\pi]], \CL [E_t[\eta]]) := (\CU_{\BR^3_+}, \CP_{\BR^3_+}, \mathrm{Tr}_{\pd \BR^3_+} \circ \CH_{\BR^3_+}) (z) (z^{1\slash2} \bK_\CL, \bK_\CL)
\end{equation*}
is a unique solution to~\eqref{eq-stokes-half-free-resolvent} possessing the estimates
\begin{equation*}
\begin{split}
\CR_{L^q (\BR^3_+)^3 \times H^{1, q} (\BR^3_+)^3 \to H^{2 - j, q} (\BR^3_+)^3} \{(\tau \pd_\tau)^\ell (z^j \CU (z)) \mid z \in \omega_0 + \Sigma_{\theta + \pi \slash2} \cup \{0\} \} & \le C, \\
\CR_{L^q (\BR^3_+)^3 \times H^{1, q} (\BR^3_+)^3 \to L^q (\BR^3_+)^3} \{(\tau \pd_\tau)^\ell (\nabla \CP (z)) \mid z \in \omega_0 + \Sigma_{\theta + \pi \slash2} \cup \{0\} \} & \le C, \\
\CR_{L^q (\BR^3_+)^3 \times H^{1, q} (\BR^3_+)^3 \to H^{3 - j', q} (\BR^3_+)} \{(\tau \pd_\tau)^\ell (z^{j'} \CH (z)) \mid z \in \omega_0 + \Sigma_{\theta + \pi \slash2} \cup \{0\} \} & \le C
\end{split}
\end{equation*}
for $\ell, j' = 0, 1$, $j = 0, 1, 2$, and $\tau$ stands for the imaginary part of $\lambda$, where a positive constant $C$ does not depend on $\lambda$ and $\omega$. Here, we have set $\Sigma_{\theta_0} := \{z \in \BC \backslash \{0\} \mid \lvert \arg (z) \rvert < \theta_0\}$ with $\theta_0 \in (0, \pi)$. Furthermore, for Banach spaces $Z_1$ and $Z_2$, we denote by $\CB (Z_1, Z_2)$  the Banach space of all bounded linear operators from $Z_1$ to $Z_2$, by $\mathrm{Hol} (U; \CB (Z_1, Z_2))$ the set of all $\CB (Z_1,Z_2)$-valued holomorphic functions defined on $U \subset \BC$, and by $\CR_{Z_1 \to Z_2} \{\CT\}$ the $\CR$-bound of a family of operators $\CT \subset \CB (Z_1, Z_2)$. Furthermore, $E_{\pd \BR^3_+}$ stands for an extension operator to $\{x_3 > 0\}$, which makes us to observe $\bK_\CL \in H^{1, q} (\BR^3_+)^3$. Notice that we have taken $\omega_0$ so large that the set of $\lambda$ may include the right half-plane $\{\lambda \in \BC \mid \RE \lambda \ge 0\}$. In the following, we may assume that $\lambda \in \{\lambda \in \BC \mid \RE \lambda \ge 0\}$. Let $\CL^{- 1}$ be the inverse of Laplace transform given by
\begin{equation*}
\CL^{- 1} [g (\lambda)] (t) = \frac{1}{2 \pi i} \int_{\gamma - i \infty}^{\gamma + i \infty} e^{\lambda t} g (\lambda) \,\mathrm{d} \lambda = \frac{1}{2 \pi} \int_{- \infty}^\infty e^{i \tau t} e^{\gamma t} g (\lambda) \,\mathrm{d} \tau = e^{\gamma t} \CF^{- 1}_\tau [g (\lambda)] (t)
\end{equation*}
for $\lambda = \gamma + i \tau \in \BC$, $t \in \BR$, where we choose $\gamma$ such that $\gamma \ge 0$. Setting
\begin{equation*}
\Lambda^{1\slash2}_z [f] (t) := \CL^{- 1} [\lvert z \rvert^{1\slash2} \CL [f] (z)] (t)
\end{equation*}
and using the relation $z^{1\slash2}[g] (z) = \CL [\Lambda^{1\slash2}_z [g] (t)] (z)$, we define $(E_t[\bu], E_t[\pi], \eta) (t)$ by
\begin{equation*}
\begin{split}
E_t[\bu] (t) & = \CL^{- 1} [\CU_{\BR^3_+} (z) (z^{1\slash2} \bK_\CL, \bK_\CL)] \\
& = e^{\gamma t} \CF^{- 1}_\tau [\CU_{\BR^3_+} (z) \CF_\tau [e^{- \gamma t} (\Lambda^{1\slash2}_z [\bK], \bK)] (\tau)] (t), \\
E_t[\pi] (t) & = \CL^{- 1} [\CP_{\BR^3_+} (z) (z^{1\slash2} \bK_\CL, \bK_\CL)] \\
& = e^{\gamma t} \CF^{- 1}_\tau [\CP_{\BR^3_+} (z) \CF_\tau [e^{- \gamma t} (\Lambda^{1\slash2}_z [\bK], \bK)] (\tau)] (t), \\
E_t[\eta] (t) & = \CL^{- 1} [\CE_{\BR^3_+} (z) (z^{1\slash2} \bK_\CL, \bK_\CL)] \\
& = e^{\gamma t} \CF^{- 1}_\tau [\CE_{\BR^3_+} (z) \CF_\tau [e^{- \gamma t} (\Lambda^{1\slash2}_z [\bK], \bK)] (\tau)] (t)
\end{split}
\end{equation*}
for $t \in \BR$, $\RE \lambda \ge 0$. Noting $\lambda \CL [f] (z) = \CL [\pd_t f] (z)$, we see that
\begin{equation*}
\begin{split}
\pd_t E_t[\bu] (t) & = \CL^{- 1} [(z - \omega) \CU_{\BR^3_+} (z) (z^{1\slash2} \bK_\CL, \bK_\CL)] \\
& = e^{\gamma t} \CF^{- 1}_\tau [(z - \omega) \CU_{\BR^3_+} (z) \CF_\tau [e^{- \gamma t} (\Lambda^{1\slash2}_z [\bK], \bK)] (\tau)] (t), \\
\pd_t E_t [\eta] (t) & = \CL^{- 1} [(z - \omega) \CE_{\BR^3_+} (z) (z^{1\slash2} \bK_\CL, \bK_\CL)] \\
& = e^{\gamma t} \CF^{- 1}_\tau [(z - \omega) \CE_{\BR^3_+} (z) \CF_\tau [e^{- \gamma t} (\Lambda^{1\slash2}_z [\bK], \bK)] (\tau)] (t)
\end{split}
\end{equation*}
Since it holds
\begin{equation*}
\begin{split}
\Lambda^{1\slash2}_z \bK & = \CL^{- 1} [\lvert \lambda + \omega \rvert^{1\slash2} \CL [\bK] (\lambda)] (t) \\
& = e^{\gamma t} \CF^{- 1}_\tau \Big[((\gamma + \omega)^2 + \tau^2)^{1\slash4} \CF_\tau [e^{- \gamma t} \bK (t)] (\tau) \Big] \\
& = e^{\gamma t} \CF^{- 1}_\tau \bigg[\bigg(1 + \frac{(\gamma + \omega)^2 - 1}{1 + \tau^2}\bigg)^{1\slash4} (1 + \tau^2)^{1\slash4} \CF_\tau [e^{- \gamma t} \bK (t)] (\tau) \bigg],
\end{split}
\end{equation*}
we see that
\begin{equation*}
\begin{split}
& \lvert e^{- \gamma t} \Lambda_z^{1\slash2} \bK \rvert_{L^p_\delta (\BR; L^q (\BR^3_+))} \\
& \le \bigg\lvert \CF^{- 1}_\tau \bigg[\bigg(1 + \frac{(\gamma + \omega)^2 - 1}{1 + \tau^2}\bigg)^{1\slash4} (1 + \tau^2)^{1\slash4} \CF_\tau [e^{- \gamma t} \bK (t)] (\tau) \bigg] \bigg\rvert_{L^p_\delta (\BR; L^q (\BR^3_+))} \\
& \le \sup_{\tau \in \BR} \bigg\lvert 1 + \frac{(\gamma + \omega)^2 - 1}{1 + \tau^2} \bigg\rvert^{1\slash4} \lvert \CF^{- 1}_\tau [(1 + \tau^2)^{1\slash4} \CF_\tau [e^{- \gamma t} \bK (t)] (\tau)] \rvert_{L^p_\delta (\BR; L^q (\BR^3_+))} \\
& \le (\gamma + \omega)^{1\slash2} \lvert e^{- \gamma t} \bK \rvert_{H^{1\slash2, p}_\delta (\BR; L^q (\BR^3_+))}
\end{split} 
\end{equation*}
provided that $\omega > \max(\omega_0, 1)$. Hence, employing the operator-valued Fourier multiplier theorem of Pr\"uss~\cite{P19}, for each $\omega > \max (\omega_0, 1)$, it holds
\begin{equation*}
\begin{split}
\lvert e^{- \gamma t}\pd_t E_t [\bu] \rvert_{L^p_\delta (\BR; L^q (\BR^3_+))} & \le C \Big(\lvert e^{- \gamma t} \Lambda_z^{1\slash2} [\bK] \rvert_{L^p_\delta (\BR; L^q (\BR^3_+))} + \lvert e^{- \gamma t} \bK \rvert_{L^p_\delta (\BR; H^{1, q} (\BR^3_+))} \Big) \\
& \le C \Big((\gamma + \omega)^{1\slash2} \lvert e^{- \gamma t} \bK \rvert_{H^{1\slash2, p}_\delta (\BR; L^q (\BR^3_+))} + \lvert e^{- \gamma t} \bK \rvert_{L^p_\delta (\BR; H^{1, q} (\BR^3_+))} \Big) 
\end{split}
\end{equation*}
with arbitrary $\gamma \ge 0$. Without loss of generality, we now assume $\gamma = 0$. From the trace theory, there exists a constant $C$ such that
\begin{equation*}
\begin{split}
\lvert \bK \rvert_{H^{1\slash2, p}_\delta (\BR; L^q (\BR^3_+))} & \le \lvert \bK \rvert_{H^{1\slash2, p}_\delta (\BR; L^q (\BR^3_+)) \cap L^p_\delta (\BR; H^{1, q} (\BR^3_+))} \le \lvert \bk \rvert_{\BF_{2, \delta} (\BR_+; \pd \BR^3_+)}, \\
\lvert \bK \rvert_{L^p_\delta (\BR; H^{1, q} (\BR^3_+))} & \le \lvert \bK \rvert_{H^{1\slash2, p}_\delta (\BR; L^q (\BR^3_+)) \cap L^p_\delta (\BR; H^{1, q} (\BR^3_+))} \le \lvert \bk \rvert_{\BF_{2, \delta} (\BR_+; \pd \BR^3_+)}
\end{split}
\end{equation*}
are valid. Thus, for each $\omega > \max (\omega_0, 1)$ we obtain
\begin{equation*}
\lvert \pd_t \bu \rvert_{L^p_\delta (\BR_+; L^q (\BR^3_+))} \le C \omega^{1\slash2} \lvert \bk \rvert_{\BF_{2, \delta} (\pd \BR^3_+)},
\end{equation*}
where a constant $C$ is independent of $\gamma$, $\tau$, and $\omega$. Similarly, we observe
\begin{alignat*}4
\lvert \bu \rvert_{L^p_\delta (\BR_+; H^{2, q} (\BR^3_+))} & \le C \omega^{1\slash2} \lvert \bk \rvert_{\BF_{2, \delta} (\BR_+; \pd \BR^3_+)}, & \qquad \lvert \nabla \pi \rvert_{L^p_\delta (\BR_+; L^q (\BR^3_+))} & \le C \omega^{1\slash2} \lvert \bk\rvert_{\BF_{2, \delta} (\BR_+; \pd \BR^3_+)}, \\
\lvert \pd_t \eta \rvert_{L^p_\delta (\BR_+; B^{2 - 1\slash q}_{q, q} (\pd \BR^3_+))} & \le C \omega^{1\slash2} \lvert \bk \rvert_{\BF_{2, \delta} (\BR_+; \pd \BR^3_+)}, & \qquad \lvert \eta \rvert_{L^p_\delta (\BR_+; B^{3 - 1\slash q}_{q, q} (\pd \BR^3_+))} & \le C \omega^{1\slash2} \lvert \bk \rvert_{\BF_{2, \delta} (\BR_+; \pd \BR^3_+)}.
\end{alignat*}
Then, we now see that the problem \eqref{eq-stokes-half-free-3} admits the unique solution $(\bu, \pi, \eta)$ with $\bu \in \BE_{1, \delta} (\BR_+; \BR^3_+)$, $\pi \in \BE_{2, \delta} (\BR_+; \pd \BR^3_+)$, and
\begin{equation*}
\eta \in H^{1, p}_\delta (\BR_+; B^{2 - 1\slash q}_{q, q} (\pd \BR^3_+)) \cap L^p_\delta (\BR_+; B^{3 - 1\slash q}_{q, q} (\pd \BR^3_+)).
\end{equation*}
Especially, from the equation for $\eta$, we obtain an additional regularity information on $\eta$, and thus we arrive at $\eta \in \BE_{4, \delta} (J; \pd \BR^3_+)$ due to $u_3, d \in F^{1 - 1\slash(2 q)}_{p, q, \delta} (\BR_+; L^q (\pd \BR^3_+)) \cap L^p_\delta (\BR_+; B^{2 - 1\slash q}_{q, q} (\pd \BR^3_+))$. In addition, since it holds 
\begin{equation*}
\pi = k_3 - (2 \mu \pd_3 u_3 - \sigma \Delta_{\pd \BR^3_+} \eta) \qquad \text{on $\pd \BR^3_+$},
\end{equation*}
we also observe that $\mathrm{Tr}_{\pd \BR^3_+} [\pi] \in \BE_{3, \delta}(\BR_+; \pd \BR^3_+)$ from the trace theory. Summing up, there exists a unique solution to \eqref{eq-stokes-half-free-shift}. Finally, a uniqueness part follows from the duality argument, which we have used above in this proof.	
\end{proof}

\subsection{The Stokes equations in quarter-spaces}
\subsubsection{System with slip-slip boundary coditions}
Let us consider the Stokes equations in a quarter-space with slip-slip boundary conditions:
\begin{align}
\label{eq-4.5}
\left\{\begin{aligned}
\pd_t \bu - \mu \Delta \bu + \nabla \pi & = \bff, & \quad & \text{in $\BK^3 \times J$}, \\
\dv \bu & = f_\mathrm{div}, & \quad & \text{in $\BK^3 \times J$}, \\
\mu (\pd_2 u_\ell + \pd_\ell u_2) & = g_\ell, & \quad & \text{in $\pd_2 \BK^3 \times J$}, \\
u_2 & = g_2, & \quad & \text{in $\pd_2 \BK^3 \times J$}, \\
\mu (\pd_3 u_m + \pd_m u_3) & = h_m, & \quad & \text{in $\pd_3 \BK^3 \times J$}, \\
u_3 & = h_3, & \quad & \text{in $\pd_3 \BK^3 \times J$}, \\
\bu (0) & = \bu_0, & \quad & \text{in $\BK^3 \times J$},
\end{aligned}\right.
\end{align}
where $\ell = 1, 3$ and $m = 1, 2$. Here, to simplify the notation, we have used the following notations:
\begin{align*}
\BK^3 & := \{(x_1, x_2, x_3) \mid x_1 \in \BR, \enskip x_2 > 0, \enskip x_3 > 0\}, \\
\pd_2 \BK^3 & := \{(x_1, x_2, x_3) \mid x_1 \in \BR, \enskip x_2 = 0, \enskip x_3 > 0\}, \\
\pd_3 \BK^3 & := \{(x_1, x_2, x_3) \mid x_1 \in \BR, \enskip x_2 > 0, \enskip x_3 = 0\}.
\end{align*}
If there is no confusion, we write $\pd \BK^3 := \pd_2 \BK^3 \cap \pd_3 \BK^3$ for short. Let us introduce the space $\wh H^{- 1, q} (\BK^3)$ as the set of all $(\varphi_1, \varphi_2, \varphi_3) \in L^q (\BK^3) \times B^{2 - 1\slash q}_{q, q} (\pd_2 \BK^3) \times B^{2 - 1\slash q}_{q, q} (\pd_3 \BK^3)$ with $(\varphi_1, \varphi_2, \varphi_3) \in {}_0 \dot H^{- 1, q} (\BK^3)$. Set
\begin{equation*}
\langle (\varphi_1, \varphi_2, \varphi_3) \mid \phi \rangle_{\BK^3} := - (\varphi_1 \mid \phi)_{\BK^3} + (\varphi_2 \mid \phi)_{\pd_2 \BK^3} + (\varphi_3 \mid \phi)_{\pd_3 \BK^3} \quad \text{for any $\phi \in \dot H^{1, q'} (\BK^3)$}.
\end{equation*}
Then, by the divergence equation, we have the condition
\begin{equation*}
\langle (f_\mathrm{div}, g_2, h_3) \mid \phi \rangle_{\BK^3} = - (\bu \mid \nabla \phi)_{\BK^3} \quad \text{for any $\phi \in \dot H^{1, q'} (\BK^3)$}.
\end{equation*}
We now prove the following theorem.
\begin{theo}
\label{th-slipslip}
Let 
\begin{equation}
\label{cond-pqdelta-model1}
1 < p < \infty, \quad 2 < q < \infty, \quad \frac{1}{p} < \delta \le 1, \quad \frac{1}{p} + \frac{1}{2 q} \notin \bigg\{\delta - \frac{1}2, \delta \bigg\}, \quad \frac{1}{p} + \frac{1}{q} \notin \bigg\{\delta - \frac{1}2, \delta \bigg\}.
\end{equation}
Let $T > 0$ and $J = (0, T)$. Then the problem \eqref{eq-4.5} has a unique solution $(\bu, \pi)$ with $\bu \in \BE_{1, \delta} (J; \BK^3)$ and $\pi \in \BE_{2, \delta} (J; \BK^3)$ if and only if
\begin{enumerate}
\renewcommand{\labelenumi}{(\alph{enumi})}
\item $\bff \in \BF_{0, \delta} (J; \BK^3)$;
\item $f_\mathrm{div} \in \BF_{1, \delta} (J; \BK^3)$;
\item $g_\ell \in \BF_{2, \delta} (J; \pd_2 \BK^3)$ for $\ell = 1, 3$;
\item $g_2 \in \BF_{3, \delta} (J; \pd_2 \BK^3)$;
\item $h_m \in \BF_{2, \delta} (J; \pd_3 \BK^3)$ for $m = 1, 2$;
\item $h_3 \in \BF_{3, \delta} (J; \pd_3 \BK^3)$;
\item $(f_\mathrm{div}, g_2, h_3) \in H^{1, p}_\delta (J; \wh H^{- 1, q} (\BK^3))$;
\item $\bu_0 \in B^{2 (\delta - 1\slash p)}_{q, p} (\BK^3)^3$ and $\dv \bu_0 = f_\mathrm{div} (0)$;
\item $g_\ell (0) = \mathrm{Tr}_{\pd_2 \BK^3} [\mu (\pd_2 u_{0, \ell} + \pd_\ell u_{0, 2})]$ for $\ell = 1, 2$ and $h_m (0) = \mathrm{Tr}_{\pd_3 \BK^3} [\mu (\pd_3 u_{0, m} + \pd_m u_{0, 3})]$ for $m = 1, 2$ if $1\slash p + 1\slash(2 q) < \delta - 1\slash2$;
\item $g_2 (0) = \mathrm{Tr}_{\pd_2 \BK^3} [u_{0, 2}]$ and $h_3 (0) = \mathrm{Tr}_{\pd_3 \BK^3} [u_{0, 3}]$ if $1\slash p + 1\slash(2 q) < \delta$;
\item $\mathrm{Tr}_{\pd \BK^3} [g_\ell (0)] = \mathrm{Tr}_{\pd \BK^3} [\mu (\pd_2 u_{0, \ell} + \pd_\ell u_{0, 2})]$ for $\ell = 1, 3$ and $\mathrm{Tr}_{\pd \BK^3} [h_m (0)] = \mathrm{Tr}_{\pd \BK^3} [\mu (\pd_3 u_{0, m} + \pd_m u_{0, 3})]$ for $m = 1, 2$ if $1\slash p + 1\slash q < \delta - 1\slash2$;
\item $\mathrm{Tr}_{\pd \BK^3} [g_2 (0)] = \mathrm{Tr}_{\pd \BK^3} [u_{0, 2}]$ and $\mathrm{Tr}_{\pd \BK^3} [h_3 (0)] = \mathrm{Tr}_{\pd \BK^3} [u_{0, 3}]$ if $1\slash p + 1\slash q < \delta$.	
\end{enumerate}
The solution $(\bu, \pi)$ depends continuously on the data in the corresponding spaces.	
\end{theo}
\begin{rema}
\label{rem-3.7}
The condition $q \in (2, \infty)$ ensures the existence of the trace onto an intersection of the boundaries. For instance, using a similar argument as in the proof of \cite[Thm.~4.6]{L19}, we see that
\begin{equation*}
\mathrm{Tr}_{\pd_3 \BK^3} \Big[\BF_{2, \delta} (J; \pd_2 \BK^3) \Big] = F^{1\slash2 - 1\slash q}_{p, q, \delta} (J; L^q (\pd \BK^3)) \cap L^p_\delta (J; B^{1 - 2\slash q}_{q, q} (\pd \BK^3))
\end{equation*}
if $2 < q < \infty$. Hence, we have compatibility conditions for $g_\ell$ and $h_m$ on $\pd \BK^3$ provided $1\slash p + 1\slash q < \delta - 1\slash2$. Similarly, there are compatibility conditions for $g_2$ and $h_3$ on $\pd \BK^3$ if $1\slash p + 1\slash q < \delta$, cf., \cite[Rem.~3.5]{L19}.
\end{rema}	
\begin{proof}
Define
\begin{equation*}
\begin{split}
\BR^{3 +} & := \{x = (x_1, x_2, x_3) \mid x_1 \in \BR, \enskip x_2 > 0, \enskip x_3 \in \BR\}, \\
\pd \BR^{3 +} & := \{x = (x_1, x_2, x_3) \mid x_1 \in \BR, \enskip x_2 = 0, \enskip x_3 \in \BR\}.
\end{split}
\end{equation*}
Let us extend $\bu_0$ to $\BR^{n +}$ in the class $B^{2 (\delta - 1\slash p)}_{q, p} (\BR^{n +})$ and denote it by $e_{\BR^{n +}} [\bu_0]$. From Theorem~\ref{th-model-half-slip}, there exist functions $\bu_S = (u_{S, 1}, u_{S, 2}, u_{S, 3})^\top \in \BE_{1, \delta} (J; \BR^{n +})$ and  $\pi_S \in \BE_{2, \delta} (J; \BR^{n +})$ satisfying
\begin{align*}
\left\{\begin{aligned}
\pd_t \bu_S - \mu \Delta \bu_S + \nabla \pi_S & = E^e_{x_3} [\bff], & \quad & \text{in $\BR^{3 +} \times J$}, \\
\dv \bu_S & = E^e_{x_3} [f_\mathrm{div}], & \quad & \text{in $\BR^{3 +} \times J$}, \\
\mu (\pd_2 u_{S, \ell} + \pd_\ell u_{S, 2}) & = E^e_{x_3} [g_\ell], & \quad & \text{on $\pd \BR^{3 +} \times J$}, \\
u_{S, 2} & = E_{x_3} [g_2], & \quad & \text{on $\pd \BR^{3 +} \times J$}, \\
\bu_S (0) & = e_{\BR^{3 +}} [\bu_0], & \quad & \text{in $\BR^{3 +}$},
\end{aligned}\right.
\end{align*}
where $E^e_{x_3}$ denotes even extension with respect to $x_3$ and $E_{x_3}$ is an extension operator defined by
\begin{align}
\label{4.6}
E_{x_3} [f (x_1, x_3, t)] := \begin{cases}
E_{\pd_2 \BK^3} [f] (x_1, x_3, t) & \text{if $x_3 > 0$}, \\
2 E_{\pd_2 \BK^3} [f] (x_1, 2 x_3, t) - E_{\pd_2 \BK^3} [f] (x_1, 3 x_3, t) & \text{if $x_3 < 0$}.
\end{cases}
\end{align}
for $f (\cdot, t) \in B^{2 - 1\slash q}_{q, q} (\pd_2 \BK^3)$. We now define $\bg^* = (g^*_1, g^*_2, g^*_3)$ by
\begin{align*}
g^*_m := \mathrm{Tr}_{\pd_3 \BK^3} [\mu (\pd_3 u_{S, m} + \pd_m u_{S, m})], \qquad g^*_3 := \mathrm{Tr}_{\pd_3 \BK^3} [u_{S, 3}]
\end{align*}
with $m = 1, 2$. If we write $\bu = \bu_S + \ov \bu_S$ and $\pi = \pi_S + \ov \pi_S$, we see that a pair of functions $(\ov \bu_S, \ov \pi_S)$ satisfies
\begin{align*}
\left\{\begin{aligned}
\pd_t \ov \bu_S - \mu \Delta \ov \bu_S + \nabla \ov \pi_S & = 0, & \quad & \text{in $\BK^3 \times J$}, \\
\dv \ov \bu_S & = 0, & \quad & \text{in $\BK^3 \times J$}, \\
\mu (\pd_2 \ov u_{S, \ell} + \pd_\ell \ov u_{S, 2}) & = 0, & \quad & \text{on $\pd_2 \BK^3 \times J$}, \\
\ov u_{S, 2} & = 0, & \quad & \text{on $\pd_2 \BK^3 \times J$}, \\
\mu (\pd_3 \ov u_{S, m} + \pd_m \ov u_{S, 3}) & = h_m - g^*_m, & \quad & \text{on $\pd_3 \BK^3 \times  J$}, \\
\ov u_{S, 3} & = h_3 - g_3^*, & \quad & \text{on $\pd_3 \BK^3 \times J$}, \\
\ov \bu_S (0) & = 0, & \quad & \text{in $\BK^3$}.
\end{aligned}\right.
\end{align*}
Notice that we have $\pd_2 \ov u_{S, \ell} = 0$ and $\ov u_{S, 2} = 0$ on the contact line $\pd \BK^3$ from the condition~\eqref{cond-pqdelta-model1}. Hence, we have $\pd_2 (h_\ell - g_\ell^*) = 0$, $\ell = 1, 3$, and $h_2 - g^*_2 = 0$ on $\pd \BK^3 \times \BR_+$, which makes us to extend $h_\ell - g^*_\ell$ by even reflection and $h_2 - g_2^*$ by odd reflection to $\{x_2 < 0\}$. We now deal with the half-space problem
\begin{align}
\label{eq-4.7}
\left\{\begin{aligned}
\pd_t \wt \bu_S^* - \mu \Delta \wt \bu_S^* + \nabla \wt \pi_S^* & = 0, & \quad & \text{in $\BR^3_+ \times J$}, \\
\dv \wt \bu_S^* & = 0, & \quad & \text{in $\BR^3_+ \times J$}, \\
\mu (\pd_3 \wt u_{S, 1}^* + \pd_1 \wt u_{S, 3}^*) & = E^e_{x_2} [h_1 - g^*_1], & \quad & \text{on $\pd \BR^3_+ \times J$}, \\
\mu (\pd_3 \wt u_{S, 2}^* + \pd_2 \wt u_{S, 3}^*) & = E^o_{x_2} [h_2 - g^*_2], & \quad & \text{on $\pd \BR^3_+ \times J$}, \\
\wt u_{S, 3}^* & = E^e_{x_2} [h_3 - g^*_3], & \quad & \text{on $\pd \BR^3_+ \times J$}, \\
\wt \bu_S^* (0) & = 0, & \quad & \text{in $\BR^3_+$}
\end{aligned}\right.
\end{align}
Here, $E^o_{x_2}$ and $E^e_{x_2}$ represent odd and even reflection with respect to $x_2$, respectively. According to Theorem~\ref{th-model-half-slip}, we can find a unique solution to \eqref{eq-4.7}. We also see that restricted function $(\wt \bu_S^* \vert_{\BK^3}, \wt \pi_S^* \vert_{\BK^3}) =: (\ov \bu_S^*, \ov \pi_S^*)$ satisfies the quarter-space problem
\begin{align*}
\left\{\begin{aligned}
\pd_t \ov \bu_S^* - \mu \Delta \ov \bu_S^* + \nabla \ov \pi_S^* & = 0, & \quad & \text{in $\BK^3 \times J$}, \\
\dv \ov \bu_S^* & = 0, & \quad & \text{in $\BK^3 \times J$}, \\
\ov \bu_S^* & = \bg^{**}, & \quad & \text{on $\pd_2 \BK^3 \times J$}, \\
\mu (\pd_3 \ov u_{S, m}^* + \pd_m \ov u_{S, 3}^*) & = h_m - g^*_m, & \quad & \text{on $\pd_3 \BK^3 \times J$}, \\
\ov u_{S, 3}^* & = h_3 - g_3^*, & \quad & \text{on $\pd_3 \BK^3 \times J$}, \\
\ov \bu_S^* (0) & = 0, & \quad & \text{in $\BK^3$},
\end{aligned}\right.
\end{align*}
where $\bg^{**} = (g^{**}_1, g^{**}_2, g^{**}_3)^\top$ is a function defined by $\bg^{**} := \mathrm{Tr}_{\pd_2 \BK^3} [\wt \bu_S^*]$. \par
In the following, we shall consider the following problem
\begin{align*}
\left\{\begin{aligned}
\pd_t \ov \bu_S^{**} - \mu \Delta \ov \bu_S^{**} + \nabla \ov \pi_S^{**} & = 0, & \quad & \text{in $\BK^3 \times J$}, \\
\dv \ov \bu_S^{**} & = 0, & \quad & \text{in $\BK^3 \times J$}, \\
\ov \bu_S^{**} & = - \bg^{**}, & \quad & \text{on $\pd_2 \BK^3 \times J$}, \\
\mu (\pd_3 \ov u_{S, m}^{**} + \pd_m \ov u_{S, 3}^{**}) & = 0, & \quad & \text{on $\pd_3 \BK^3 \times J$}, \\
\ov u_{S, 3}^{**} & = 0, & \quad & \text{on $\pd_3 \BK^3 \times J$}, \\
\ov \bu^{**}_S (0) & = 0, & \quad & \text{in $\BK^3$}
\end{aligned}\right.
\end{align*}
with $\ell = 1, 3$ and $m = 1, 2$. At this stage, the function $\ov \bu_S^*$ is given so that $\bg^{**}$ is given as well. By the compatibility conditions on the contact line $\pd \BK^3$, we have $\pd_3 (- g^{**}_m) = - g^{**}_3 = 0$ on $\pd \BK^3 \times J$ and this makes us to extend $- g_m^{**}$ by even reflection and $- g^{**}_3$ by odd reflection with respect to $x_3$, respectively. Then we consider the reflected half-space problem
\begin{align}
\label{eq-4.8}
\left\{\begin{aligned}
\pd_t \wt \bu_S^{**} - \mu \Delta \wt \bu_S^{**} + \nabla \wt \pi_S^{**} & = 0, & \quad & \text{in $\BR^{n +} \times J$}, \\
\dv \wt \bu_S^{**} & = 0, & \quad & \text{in $\BR^{n +} \times J$}, \\
\wt u_{S, m}^{**} & = - E^e_{x_3} [g^{**}_m], & \quad & \text{in $\pd \BR^{n +} \times J$}, \\
\wt u_{S, 3}^{**} & = - E^o_{x_3} [g^{**}_3], & \quad & \text{in $\pd \BR^{n +} \times J$}, \\
\wt \bu^{**}_S (0) & = 0, & \quad & \text{in $\BR^{n +}$}
\end{aligned}\right.
\end{align}
Notice that we have $\wt u_{S, 3}^{**} = 0$ on $\pd \BK^3 \times J$. According to Theorem~\ref{th-model-half-slip}, we see that \eqref{eq-4.8} admits a unique solution in the right regularity class. We emphasize that the restricted pair $(\wt \bu_S^{**} \vert_{\BK^3}, \wt \pi_S^{**} \vert_{\BK^3})$ satisfies the slip boundary conditions on $\pd_3 \BK^3$:
\begin{align}
\label{slip-b.c.}
\mu (\pd_3 (\wt u_{S, m}^{**} \vert_{\BK^3}) + \pd_m (\wt u_{S, 3}^{**} \vert_{\BK^3})) = 0 \quad \text{and} \quad \wt u_{S, 3}^{**} \vert_{\BK^3} = 0 \quad \text{on $\pd_3 \BK^3 \times J$}.
\end{align}
Indeed, it is easy to verify that the function $(\wt \bv_S, \wt p_S)$ with 
\begin{align*}
\wt \bv_S (x, t) & := (\wt u_{S, 1}^{**} (x_1, x_2, - x_3, t), \wt u_{S, 2}^{**} (x_1, x_2, - x_3, t), - \wt u_{S, 3}^{**} (x_1, x_2, - x_3, t)), \\
\wt p_S (x, t) & := \wt \pi^{**}_S (x_1, x_2, - x_3, t)
\end{align*}
is a solution to \eqref{eq-4.8} in the required regularity class. By Theorem~\ref{th-model-half-dirichlet} we know that a solution to \eqref{eq-4.8} is unique, so that we observe $\wt \bu_S^{**} = \wt \bv_S$ and $\wt \pi_S^{**} = \wt p_S$, i.e.,
\begin{align*}
\wt u^{**}_{S, m} (x_1, x_2, - x_3, t) & = \wt u^{**}_{S, m} (x_1, x_2, x_3, t) \qquad (m = 1, 2), \\
\wt u^{**}_{S, 3} (x_1, x_2, - x_3, t) & = - \wt u^{**}_{S, 3} (x_1, x_2, x_3, t), \\
\wt \pi_S (x_1, x_2, - x_3, t) & = \wt \pi_S (x_1, x_2, - x_3, t).
\end{align*}
This implies that $\wt u^{**}_{S, m}$ are even extension of $\ov u^{**}_{S, m}$ and that $\wt u^{**}_{S, 3}$ is odd extension of $\ov u^{**}_{S, 3}$ with respect to $x_3$, respectively. From these observations, we obtain~\eqref{slip-b.c.} and see that the problem \eqref{eq-4.8} admits a solution defined by $(\ov \bu^{**}_S, \ov \pi^{**}_S) = (\wt \bu_S^{**} \vert_{\BK^3}, \wt \pi_S^{**} \vert_{\BK^3})$. Summing up, we observe that $(\ov \bu_S, \ov \pi_S)$ is given by $\ov \bu_S = \ov \bu_S^* + \ov \bu_S^{**}$ and $\ov \pi_S = \ov \pi_S^* + \ov \pi_S^{**}$, and thus $\bu = \bu_S + \ov \bu_S^* + \ov \bu_S^{**}$ and $\pi = \pi_S + \ov \pi_S^* + \ov \pi_S^{**}$ satisfy \eqref{eq-4.5} in the right regularities. The uniqueness of the solution easily follows from its construction.	
\end{proof}

\subsubsection{System with slip-free boundary conditions}
Consider the Stokes equations in a quarter-space with slip-free boundary conditions:
\begin{align}
\label{eq-4.10}
\left\{\begin{aligned}
\pd_t \bu - \mu \Delta \bu + \nabla \pi & = \bff, & \quad & \text{in $\BK^3 \times J$}, \\
\dv \bu & = f_\mathrm{div}, & \quad & \text{in $\BK^3 \times J$}, \\
\mu (\pd_2 u_\ell + \pd_\ell u_2) & = g_\ell, & \quad & \text{on $\pd_2 \BK^3 \times J$}, \\
u_2 & = g_2, & \quad & \text{on $\pd_2 \BK^3 \times J$}, \\
\pd_t \eta - u_3 & = d, & \quad & \text{on $\pd_3 \BK^3 \times J$}, \\
\mu (\pd_3 u_m + \pd_m u_3) & = k_m, & \quad & \text{on $\pd_3 \BK^3 \times J$}, \\
2 \mu \pd_3 u_3 - \pi - \sigma \Delta_{\pd_3 \BK^3} \eta & = k_3, & \quad & \text{on $\pd_3 \BK^3 \times J$}, \\
\pd_2 \eta & = h, & \quad & \text{on $\pd \BK^3$}, \\
\bu (0) & = \bu_0, & \quad & \text{in $\BK^3$}, \\
\eta (0) & = \eta_0, & \quad & \text{on $\pd_3 \BK^3$},
\end{aligned}\right.
\end{align}
where $\ell = 1, 3$ and $m = 1, 2$. Similarly as we introduced in the previous subsection, we define the space $\wh H^{- 1, q}_{\pd_2 \BK^3} (\BK^3)$ as the set of all $(\varphi_1, \varphi_2) \in L^q (\BK^3) \times B^{2 - 1\slash q}_{q, q} (\pd_2 \BK^3)$ satisfying $(\varphi_1, \varphi_2) \in \dot H^{- 1, q}_{\pd_2 \BK^3} (\BK^3)$. Set
\begin{equation*}
\langle (\varphi_1, \varphi_2) \mid \phi \rangle_{\BK^3} := - (\varphi_1 \mid \phi)_{\BK^3} + (\varphi_2 \mid \phi)_{\pd_2 \BK^3} \quad \text{for any $\phi \in \dot H^{1, q'}_{\pd_3 \BK^3} (\BK^3)$}.
\end{equation*}
Then, by the divergence equation, we obtain the condition
\begin{equation*}
\langle (f_\mathrm{div}, g_2) \mid \phi \rangle_{\BK^3} = - (\bu \mid \nabla \phi)_{\BK^3} \quad \text{for any $\phi \in \dot H^{1, q'}_{\pd_3 \BK^3} (\BK^3)$}.
\end{equation*}
Besides, we introduce a function space
\begin{equation*}
\BF_{4, \delta} (J; S) := F^{3\slash2 - 1\slash q}_{p, q, \delta} (J; L^q (S)) \cap H^{1, p}_\delta (J; B^{1 - 2\slash q}_{q, q} (S)) \cap L^p_\delta (J; B^{2 - 2\slash q}_{q, q} (S))
\end{equation*}
for $S \subset \BR^3$. We will prove the following theorem.
\begin{theo}
\label{th-model-slipfree}
Let $T > 0$ and $J = (0, T)$. Suppose that $p$, $q$, and $\delta$ satisfy
\begin{equation}
\label{cond-pqdelta-model2}
1 < p < \infty, \quad 2 < q < \infty, \quad \frac{1}{p} < \delta \le 1, \quad \frac{1}{p} + \frac{1}{2 q} \ne \delta, \quad \frac{1}{p} + \frac{1}{q} \ne \delta - \frac{1}2.
\end{equation}
Then the problem \eqref{eq-4.10} admits a unique solution $(\bu, \pi, \eta)$ with $\bu \in \BE_{1, \delta} (J; \BK^3)$, $\pi \in \BE_{2, \delta} (J; \BK^3)$, $\mathrm{Tr}_{\pd_3 \BK^3} [\pi] \in \BF_{2, \delta} (J; \pd_3 \BK^3)$, $\eta \in \BE_{4, \delta} (J; \pd_3 \BK^3)$, if and only if
\begin{enumerate}
\renewcommand{\labelenumi}{(\alph{enumi})}
\item $\bff \in \BF_{0, \delta} (J; \BK^3)$;
\item $f_\mathrm{div} \in \BF_{1, \delta} (J; \BK^3)$;
\item $g_\ell \in \BF_{2, \delta} (J; \pd_2 \BK^3)$ for $\ell = 1, 3$;
\item $g_2 \in \BF_{3, \delta} (J; \pd_2 \BK^3)$;	
\item $d \in \BF_{3, \delta} (J; \pd_3 \BK^3)$;
\item $k_j \in \BF_{2, \delta} (J; \pd_3 \BK^3)$ for $j = 1, 2, 3$;
\item $h \in \BF_{4, \delta} (J; \pd \BK^3)$;
\item $(f_\mathrm{div}, g_2, 0) \in H^{1, p}_\delta (J; \wh H^{- 1, q}_{\pd_2 \BK^3} (\BK^3))$;
\item $\bu_0 \in B^{2 (\delta - 1\slash p)}_{q, p} (\BK^3)^3$ and $\dv \bu_0 = f_\mathrm{div} (0)$;
\item $\eta_0 \in B^{2 + \delta - 1\slash p - 1\slash q}_{q, p} (\pd_3 \BK^3)$ and $\pd_2 \eta_0 = h (0)$;
\item $g_\ell (0) = \mathrm{Tr}_{\pd_2 \BK^3} [\mu (\pd_2 u_{0, \ell} + \pd_\ell u_{0, 2})]$ for $\ell = 1, 2$ and $k_m (0) = \mathrm{Tr}_{\pd_3 \BK^3} [\mu (\pd_3 u_{0, m} + \pd_m u_{0, 3})]$ for $m = 1, 2$ if $1\slash p + 1 \slash (2 q) < \delta$;
\item $\mathrm{Tr}_{\pd \BK^3} [g_\ell (0)] = \mathrm{Tr}_{\pd \BK^3} [\mu (\pd_2 u_{0, \ell} + \pd_\ell u_{0, 2})]$ for $\ell = 1, 3$ and $\mathrm{Tr}_{\pd \BK^3} [k_m (0)] = \mathrm{Tr}_{\pd \BK^3} [\mu (\pd_3 u_{0, m} + \pd_m u_{0, 3})]$ for $m = 1, 2$ if $1\slash p + 1\slash q < \delta - 1\slash2$.	
\end{enumerate}
Besides, the solution $(\bu, \pi, \eta)$ depends continuously on the data in the corresponding spaces.	
\end{theo}
\begin{rema}
Since \eqref{cond-pqdelta-model2} infers $\delta + (1 \slash 2) > 1 \slash (2 p) + 1 \slash q$, the trace of $\eta_0$ onto $\pd \BK$ exists, i.e., the compatibility condition $\pd_2 \eta_0 = h (0)$ on $\pd \BK^3$ should be considered.
\end{rema}
\begin{proof}
Necessity follows easily from trace theory. In the following, we will prove sufficiency. Let $e_{\BR^3_+} [\bu_0]$ and $e_{\BR^3_+} [\eta_0]$ denote extensions of $\bu_0$ and $\eta_0$ to $\BR^3_+$ in the class $B^{2 (\delta - 1\slash p)}_{q, p} (\BR^3_+)^3$ and $B^{2 + \delta - 1\slash p - 1\slash q}_{q, p} (\pd \BR^3_+)$, respectively. Using Theorem~\ref{th-model-half-free}, the problem
\begin{align*}
\left\{\begin{aligned}
\pd_t \bu_F - \mu \Delta \bu_F + \nabla \pi_F & = \bff, & \quad & \text{in $\BR^3_+ \times J$}, \\
\dv \bu_F & = f_\mathrm{div}, & \quad & \text{in $\BR^3_+ \times J$}, \\
\pd_t \eta_F - u_{F, 3} & = E_{x_2} [d], & \quad & \text{on $\pd \BR^3_+ \times J$}, \\
\mu (\pd_3 u_{F, m} + \pd_m u_{F, 3}) & = E^e_{x_2} [k_m], & \quad & \text{on $\pd \BR^3_+ \times J$}, \\
2 \mu \pd_3 u_{F, 3} - \pi_F - \sigma \Delta_{\pd \BR^3_+} \eta_F & = E^e_{x_2} [k_3], & \quad & \text{on $\pd \BR^3_+ \times J$}, \\
\bu_F (0) & = e_{\BR^3_+} [\bu_0], & \quad & \text{in $\BR^3_+$}, \\
\eta_F (0) & = e_{\BR^3_+} [\eta_0], & \quad & \text{on $\pd \BR^3_+$}
\end{aligned}\right.
\end{align*}
admits a unique solution $(\bu_F, \pi_F, \eta_F)$, where $E_{x_2}[d]$ is an extension of $d$ to $\{x_2 < 0\}$ given by
\begin{align*}
E_{x_2} d (x_1, x_3) := \begin{cases}
d (x_1, x_2) & \text{if $x_2 > 0$}, \\
2 d (x_1, 2 x_3) - d (x_1, 3 x_3) & \text{if $x_2 < 0$}.
\end{cases}
\end{align*}
Writing $\bu = \bu_F + \ov \bu_F$, $\pi = \pi_F + \ov \pi_F$, and $\eta = \eta_F + \ov \eta_F$, we find that $(\ov \bu_F, \ov \pi_F, \ov \eta_F)$ solves
\begin{align}
\label{eq-4.11}
\left\{\begin{aligned}
\pd_t \ov \bu_F - \mu \Delta \ov \bu_F + \nabla \ov \pi_F & = 0, & \quad & \text{in $\BK^3 \times J$}, \\
\dv \ov \bu_F & = 0, & \quad & \text{in $\BK^3 \times J$}, \\
\mu (\pd_2 \ov u_{F, \ell} + \pd_\ell \ov u_{F, 2}) & = g_\ell - k^*_\ell, & \quad & \text{on $\pd_2 \BK^3 \times J$}, \\
\ov u_{F, 2} & = g_2 - k^*_2, & \quad & \text{on $\pd_2 \BK^3 \times J$}, \\
\pd_t \ov \eta_F - \ov u_{F, 3} & = 0, & \quad & \text{on $\pd_3 \BK^3 \times J$}, \\
\mu (\pd_3 \ov u_{F, m} + \pd_m \ov u_{F, 3}) & = 0, & \quad & \text{on $\pd_3 \BK^3 \times J$}, \\
2 \mu \pd_3 \ov u_{F, 3} - \ov \pi_F - \sigma \Delta_{\pd_3 \BK^3} \ov \eta_F & = 0, & \quad & \text{on $\pd_3 \BK^3 \times J$}, \\
\pd_2 \ov \eta_F & = h - h^*, & \quad & \text{on $\pd \BK^3 \times J$}, \\
\ov \bu_F (0) & = 0, & \quad & \text{in $\BK^3$}, \\
\ov \eta_F (0) & = 0, & \quad & \text{on $\pd_3 \BK^3$}
\end{aligned}\right.
\end{align}
with
\begin{align*}
k^*_\ell := \mathrm{Tr}_{\pd_2 \BK^3} [\mu (\pd_3 u_{F, \ell} + \pd_\ell u_{F, \ell})], \quad k^*_2 := \mathrm{Tr}_{\pd_2 \BK^3} [u_{F, 2}], \quad h^* := - \mathrm{Tr}_{\pd \BK^3} [\pd_2 \eta_F].
\end{align*}
Notice that $d^*$ is well-defined. In fact, by Lemma~\ref{lem-ext1} and \cite[Prop.~3.10]{MV12}, we have
\begin{equation*}
\begin{split}
h^* & \in {}_0 F^{3\slash2 - 1\slash q}_{p, q, \delta} (J; L^q (\pd_3 \BK^3)) \cap {}_0 H^{1, p}_\delta (J; B^{1 - 2\slash q}_{q, q} (\pd_3 \BK^3)) \cap L^p_\delta (J; B^{2 - 2\slash q}_{q, q} (\pd_3 \BK^3))
\end{split}
\end{equation*}
provided the condition \eqref{cond-pqdelta-model2}. From Lemma~\ref{lem-ext2}, we see that there exists $\ov \eta^*_F \in \BE_{4, \delta} (J; \pd_3 \BK^3)$ such that $\pd_2 \ov \eta_F^* = h - h^*$ on $\pd \BK^3 \times J$. We next consider the half-space problem
\begin{align}
\label{eq-4.12}
\left\{\begin{aligned}
\pd_t \ov \bu_F^* - \mu \Delta \ov \bu_F^* + \nabla \ov \pi_F^* & = 0, & \quad & \text{in $\BR^{3 +} \times J$}, \\
\dv \ov \bu_F^* & = 0, & \quad & \text{in $\BR^{3 +} \times J$}, \\
\mu (\pd_2 \ov u_{F, \ell}^* + \pd_\ell \ov u_{F, 2}^*) & = E^e_{x_3} [g_\ell - k^*_\ell], & \quad & \text{on $\pd \BR^{3 +} \times J$}, \\
\ov u_{F, 2}^* & = E_{x_3} [g_2 - k^*_2], & \quad & \text{on $\pd \BR^{3 +} \times J$}, \\
\ov \bu_F^* (0) & = 0, & \quad & \text{in $\BR^{3 +}$},
\end{aligned}\right.
\end{align}
where $E_{x_3}$ is the extension operator defined by \eqref{4.6}. From Theorem~\ref{th-model-half-slip}, there exist a pair $(\ov \bu_F^*, \ov \pi^*_F)$ that is a solution to \eqref{eq-4.12} in the right regularity class. We then see that $\ov \bu_F^{**} := \ov \bu_F - \ov \bu_F^* \vert_{\BK^3}$, $\ov \pi_F^{**} := \ov \pi_F - \ov \pi_F^* \vert_{\BK^3}$, and $\ov \eta_F^{**} := \ov \eta_F - \ov \eta_F^*$ satisfy the system
\begin{align}
\label{eq-4.13}
\left\{\begin{aligned}
\pd_t \ov \bu_F^{**} - \mu \Delta \ov \bu_F^{**} + \nabla \ov \pi_F^{**} & = 0, & \quad & \text{in $\BK^3 \times J$}, \\
\dv \ov \bu_F^{**} & = 0, & \quad & \text{in $\BK^3 \times J$}, \\
\mu (\pd_2 \ov u_{F, \ell}^{**} + \pd_\ell \ov u_{F, 2}^{**}) & = 0, & \quad & \text{on $\pd_2 \BK^3 \times J$}, \\
\ov u_{F, 2}^{**} & = 0, & \quad & \text{on $\pd_2 \BK^3 \times J$}, \\
\pd_t \ov \eta_F^{**} - \ov u_{F, 3}^{**} & = d^{**}, & \quad & \text{on $\pd_3 \BK^3 \times J$}, \\
\mu (\pd_3 \ov u_{F, m}^{**} + \pd_m \ov u_{F, 3}^{**}) & = k^{**}_m, & \quad & \text{on $\pd_3 \BK^3 \times J$}, \\
2 \mu \pd_3 \ov u_{F, 3}^{**} - \ov \pi_F^{**} - \sigma \Delta_{\pd_3 \BK^3} \ov \eta_F^{**} & = k^{**}_3, & \quad & \text{on $\pd_3 \BK^3 \times J$}, \\
\pd_2 \ov \eta_F^{**} & = 0, & \quad & \text{on $\pd \BK^3 \times J$}, \\
\ov \bu_F^{**} (0) & = 0, & \quad & \text{in $\BK^3$}, \\
\ov \eta_F^{**} (0) & = 0, & \quad & \text{on $\pd_3 \BK^3$},
\end{aligned}\right.
\end{align}
where we have set
\begin{align*}
d^{**} := - \pd_t \ov \eta_F^* + \ov u_{F, 3}^*, \quad k^{**}_m := \mu (\pd_3 \ov u^*_{F, m} + \pd_m \ov u^*_{F, 3}), \quad k^{**}_3 := 2 \mu \pd_3 \ov u_{F, 3}^* - \ov \pi_F^* - \sigma \Delta_{\pd_3 \BK^3} \ov \eta_F^*.
\end{align*}
To show the existence of solution to \eqref{eq-4.13}, we consider the reflected half-space problem
\begin{align}
\label{eq-4.14}
\left\{\begin{aligned}
\pd_t \wt \bu_F^{**} - \mu \Delta \wt \bu_F^{**} + \nabla \wt \pi_F^{**} & = 0, & \quad & \text{in $\BR^3_+ \times J$}, \\
\dv \wt \bu_F^{**} & = 0, & \quad & \text{in $\BR^3_+ \times J$}, \\
\pd_t \wt \eta_F^{**} - \wt u_{F, 3}^{**} & = E_{x_2} [d^{**}], & \quad & \text{on $\pd \BR^3_+ \times J$}, \\
\mu (\pd_3 \wt u_{F, m}^{**} + \pd_m \wt u_{F, 3}^{**}) & = E^e_{x_2} [k^{**}_m], & \quad & \text{on $\pd \BR^3_+ \times J$}, \\
2 \mu \pd_3 \wt u_{F, 3}^{**} - \wt \pi_F^{**} - \sigma \Delta_{\pd_3 \BK^3} \wt \eta_F^{**} & = E^e_{x_2} [k^{**}_3], & \quad & \text{on $\pd \BR^3_+ \times J$}, \\
\wt \bu^{**}_F (0) & = 0, & \quad & \text{in $\BR^3_+$}, \\
\wt \eta_F^{**} (0) & = 0, & \quad & \text{on $\pd \BR^3_+$}.
\end{aligned}\right.
\end{align}
According to Theorem~\ref{th-model-half-free}, the problem \eqref{eq-4.14} admits a unique solution $(\wt \bu_F^{**}, \wt \pi_F^{**}, \wt \eta_F^{**})$. In addition, we can observe that the function $(\wt \bv_F, \wt p_F, \wt \zeta_F) = (\wt v_{F, 1}, \wt v_{F, 2}, \wt v_{F, 3}, \wt p_F, \wt \zeta_F)$ defined by
\begin{align*}
\begin{aligned}
\wt v_{F, \ell} (x, t) & := \wt u_{F, \ell}^{**} (x_1, - x_2, x_3, t), & \quad \wt v_{F, 2} (x, t) & := - \wt u_{F, 2}^{**} (x_1, - x_2, x_3, t), \\
\wt p_F (x, t) & := \wt \pi_F^{**} (x_1, - x_2, x_3, t), & \quad \wt \zeta_F (x, t) & := \wt \eta_F^{**} (x_1, - x_2, t)
\end{aligned}
\end{align*}
solves the problem \eqref{eq-4.14}. Since a solution to \eqref{eq-4.14} is unique, it follows that
\begin{align*}
\begin{aligned}
\wt u_{F, \ell}^{**} (x_1, - x_2, x_3, t) & = \wt u_{F, \ell}^{**} (x_1, x_2, x_3, t), & \quad \wt u_{F, 2}^{**} (x_1, - x_2, x_3, t) & = - \wt u_{F, 2}^{**} (x_1, x_2, x_3, t), \\
\wt \pi_F^{**} (x_1, - x_2, x_3, t) & = \wt \pi_F^{**} (x_1, x_2, x_3, t), & \quad \wt \eta_F^{**} (x_1, - x_2, t) & = \wt \eta_F^{**} (x_1, x_2, t).
\end{aligned}
\end{align*}
This makes us to obtain
\begin{align*}
\mu (\pd_2 \wt u_{F, \ell}^{**} + \pd_\ell \wt u_{F, 2}^{**}) = 0, \quad \wt u_{F, 2}^{**} = 0 \quad \text{on $\pd_2 \BK^3 \times J$}
\end{align*}
and $\pd_2 \wt \eta_F^{**} = 0$ on $\pd \BK^3 \times J$. Hence, the restricted function $(\wt \bu_F^{**} \vert_{\BK^3}, \wt \pi_F^{**} \vert_{\BK^3}, \wt \eta_F^{**} \vert_{\pd_3 \BK^3})$ is a solution to~\eqref{eq-4.13}. Thus, a solution to \eqref{eq-4.11} is given by $\ov \bu_F = \ov \bu_F^* + \ov \bu_F^{**}$, $\ov \pi_F = \ov \pi_F^* + \ov \pi_F^{**}$, and $\ov \eta_F = \ov \eta_F^* + \ov \eta_F^{**}$. Summing up, we see that a solution to \eqref{eq-4.10} is given by
\begin{align*}
\bu = \bu_F + \ov \bu_F^* + \ov \bu_F^{**}, \quad \pi = \pi_F + \ov \pi_F^* + \ov \pi_F^{**}, \quad \eta = \eta_F + \ov \eta_F^* + \ov \eta_F^{**}.
\end{align*}
The uniqueness of the solution to \eqref{eq-4.10} can be obtained by its construction. The proof is complete.	
\end{proof}

\section{Maximal Regularity of the principal Linearization}
\label{sect-MR}
\noindent
The principle part of the linearized system reads as follows
\begin{align}
\label{eq-linear}
\left\{
\begin{aligned}
\pd_t \bu - \mu \Delta \bu + \nabla \pi & = \bff, & \quad & \text{in $\Omega_* \times J$,} \\
\dv \bu & = f_\mathrm{div}, & \quad & \text{in $\Omega_* \times J$,} \\
\pd_t \eta - u_3 & = d, & \quad & \text{in $\Gamma_* \times J$,} \\
\mu (\pd_3 u_m + \pd_m u_3) & = k_m, & \quad & \text{in $\Gamma_* \times J$,} \\
2 \mu \pd_3 u_3 - \pi - \sigma \Delta_{\Gamma_*} \eta & = k_3, & \quad & \text{in $\Gamma_* \times J$,} \\
P_{\Sigma_*} (2 \mu \bD (\bu) \bn_{\Sigma_*}) & = P_{\Sigma_*} \bg & \quad & \text{on $\Sigma_* \times J$,} \\
\bu \cdot \bn_{\Sigma_*} & = \bg \cdot \bn_{\Sigma_*} & \quad & \text{on $\Sigma_* \times J$,} \\
\mu (\pd_3 u_m + \pd_m u_3) & = h_m & \quad & \text{on $B \times J$,} \\
u_3 & = h_3 & \quad & \text{on $B \times J$,} \\
(\nabla_{\pd D} \eta) \cdot \bn_{\pd D} & = h_4, & \quad & \text{on $S_* \times J$}, \\
\bu (0) & = \bu_0 & \quad & \text{in $\Omega_*$}, \\ 
\eta (0) & = \eta_0, & \quad & \text{on $\Gamma_*$},
\end{aligned}
\right.
\end{align}
where $m = 1, 2$ and $\bg = (g_1, g_2, g_3)^\top$. To identify a hidden regularity coming from the divergence equation, we define
the space $\wh H^{- 1, q} (\BK^3)$ as the set of all $(\varphi_1, \varphi_2, \varphi_3) \in L^q (\BK^3) \times B^{2 - 1\slash q}_{q, q} (\Sigma_*) \times B^{2 - 1\slash q}_{q, q} (B)$ with $(\varphi_1, \varphi_2, \varphi_3) \in \dot H^{- 1, q}_{\Sigma_* \cup B} (\BK^3)$. Set
\begin{equation*}
\langle (\varphi_1, \varphi_2, \varphi_3) \mid \phi \rangle_{\Omega_*} := - (\varphi_1 \mid \phi)_{\Omega_*} + (\varphi_2 \mid \phi)_{\Sigma_*} + (\varphi_3 \mid \phi)_B \quad \text{for any $\phi \in \dot H^{1, q'}_{\Gamma_*} (\BK^3)$}.
\end{equation*}
By the divergence theorem, the divergence equation yields the condition
\begin{equation*}
\langle (f_\mathrm{div}, g_2, h_3) \mid \phi \rangle_{\BK^3} = - (\bu \mid \nabla \phi)_{\BK^3} \quad \text{for any $\phi \in \dot H^{1, q'}_{\Gamma_*} (\BK^3)$}.
\end{equation*}
\par
The aim of this section is to show maximal regularity of the principal linearization.
\begin{theo}
\label{th-MR-linear}
Set $J = (0, T)$, $T > 0$. Let $p$, $q$, and $\delta$ satisfy \eqref{cond-pqdelta-model2}. There exists a unique solution $(\bu, \pi, \eta)$ of \eqref{eq-linear} with $\bu \in \BE_{1, \delta} (J; \Omega_*)$, $\pi \in \BE_{2, \delta} (J; \Omega_*)$, $\mathrm{Tr}_{\Gamma_*} [\pi] \in \BE_{3, \delta} (J; \Gamma_*)$, $\eta \in \BE_{4, \delta} (J; \Gamma_*)$, if and only if
\begin{enumerate}
\renewcommand{\labelenumi}{(\alph{enumi})}
\item $\bff \in \BF_{0, \delta} (J; \Omega_*)$;
\item $f_\mathrm{div} \in \BF_{1, \delta} (J; \Omega*)$;
\item $P_{\Sigma_*} \bg \in \BF_{2, \delta} (J; \Sigma_*)$ for $\ell = 1, 3$;
\item $\bg \cdot \bn_{\Sigma_*} \in \BF_{3, \delta} (J; \Sigma_*)$;	
\item $d \in \BF_{3, \delta} (J; \Gamma_*)$;
\item $k_j \in \BF_{2, \delta} (J; \Gamma_*)$ for $j = 1, 2, 3$;
\item $h_m \in \BF_{2, \delta} (J; B)$;
\item $h_3 \in \BF_{3, \delta} (J; B)$;
\item $h_4 \in \BF_{4, \delta} (J; S_*)$;
\item $(f_\mathrm{div}, \bg \cdot \bn_{\Sigma_*}, h_3) \in H^{1, p}_\delta (J; \wh H^{- 1, q}_{\Sigma_* \cup B} (\Omega_*))$;
\item $\bu_0 \in B^{2 (\delta - 1\slash p)}_{q, p} (\Omega_*)^3$ and $\dv \bu_0 = f_\mathrm{div} (0)$;
\item $\eta_0 \in B^{2 + \delta - 1\slash p - 1\slash q}_{q, p} (\Gamma_*)$ and $(\nabla_{\pd D} \eta_0) \cdot \bn_{\pd D} = h_4 (0)$;
\item $P_{\Sigma_*} \bg (0) = \mathrm{Tr}_{\Sigma_*} [P_{\Sigma_*} (2 \mu \bD (\bu_0) \bn_{\Sigma_*})]$, $k_m (0) = \mathrm{Tr}_{\Gamma_*} [\mu (\pd_3 u_{0, m} + \pd_m u_{0, 3})]$, and $h_m (0) = \mathrm{Tr}_B [\mu (\pd_3 u_m + \pd_m u_3)]$ if $1\slash p + 1\slash(2 q) < \delta - 1\slash2$;
\item $\bg (0) \cdot \bn_{\Sigma_*} = \mathrm{Tr}_{\Sigma_*} [\bu_0 \cdot \bn_{\Sigma_*}]$ and $h_3 (0) = \mathrm{Tr}_B [u_{0, 3}]$ if $1\slash p + 1\slash(2 q) < \delta$;
\item $\mathrm{Tr}_{S_*} [P_{\Sigma_*} \bg (0)] = \mathrm{Tr}_{S_*} [P_{\Gamma_*} (2 \mu \bD (\bu_0) \bn_{\Sigma_*})]$ and $\mathrm{Tr}_{S_*} [k_m (0)] = \mathrm{Tr}_{S_*} [\mu (\pd_3 u_{0, m} + \pd_m u_{0, 3})]$ if $1\slash p + 1\slash q < \delta - 1\slash2$;
\item $\mathrm{Tr}_{\pd \Sigma_* \cap \pd B} [\bg (0) \cdot \bn_{\Sigma_*}] = \mathrm{Tr}_{\pd \Sigma_* \cap \pd B} [\bu_0 \cdot \bn_{\Sigma_*}]$ and $\mathrm{Tr}_{\pd \Sigma_* \cap \pd B} [h_3 (0)] = \mathrm{Tr}_{\pd \Sigma_* \cap \pd B} [u_{0, 3}]$ if $1\slash p + 1\slash q < \delta$.
\end{enumerate}
In addition, the solution map is continuous between the corresponding spaces.
\end{theo}

\subsection{Bent spaces}
\label{sect-bent}
\noindent
In this subsection, we give results of the Stokes equations in bent spaces.

\subsubsection{Bent half spaces}
Let $\gamma \colon \BR^2 \to \BR$ be a bounded function in $C^3$ class such that $\lvert \nabla_{x'} \gamma \rvert_{L^\infty (\BR^2)} \le c$, where $c$ is a small constant. Let $\BR^3_\gamma$ be the bent half space defined by
\begin{equation}
\label{def-Rgamma}
\BR^3_\gamma := \{(x_1, x_2, x_3) \in \BR^3 \mid (x_1, x_2) \in \BR^2, \enskip x_3 > \gamma (x_1, x_2)\}
\end{equation}
and $\pd \BR^3_\gamma$ be its boundary defined by
\begin{equation*}
\pd \BR^3_\gamma := \{(x_1, x_2, x_3) \in \BR^3 \mid (x_1, x_2) \in \BR^2, \enskip x_3 = \gamma (x_1, x_2)\}.
\end{equation*}
Introducing the transformation $\Phi_\gamma \colon \BR^3_\gamma \to \BR^3_+$ defined by $y = \Phi_\gamma (x) = (x_1, x_2, x_3 - \gamma (x_1, x_2))$, we see that $\Phi_\gamma$ is a bijection with Jacobian equal to $1$. We also find that $\BR^3_\gamma = \Phi_\gamma (\BR^3_+)$ and $\pd \BR^3_\gamma = \Phi_\gamma (\pd \BR^3_+)$. For the unit normal vector of $\pd \BR^3_\gamma$, we have
\begin{equation*}
\bn_{\pd \BR^3_\gamma} = \frac{1}{\sqrt{1 + \lvert \nabla_{x'} \gamma \rvert^2}} (\nabla_{x'} \gamma, - 1)^\top.
\end{equation*}
Let $P_{\BR^3_\gamma}$ be the tangential projection to $\BR^3_\gamma$. Consider the following two systems:
\begin{equation}
\label{eq-stokes-benthalf-slip}
\left\{\begin{split}
\pd_t \bu - \mu \Delta \bu + \nabla \pi & = \bff, & \quad & \text{in $\BR^3_\gamma \times J$}, \\
\dv \bu & = f_\mathrm{div}, & \quad & \text{in $\BR^3_\gamma \times J$}, \\
P_{\BR^3_\gamma}(2 \mu \bD (\bu) \bn_{\pd \BR^3_\gamma}) & = \bh', & \quad & \text{on $\pd \BR^3_\gamma \times J$}, \\
\bu \cdot \bn_{\pd \BR^3_\gamma} & = h_3, & \quad & \text{on $\pd \BR^3_\gamma \times J$}, \\
\bu (0) & = \bu_0, & \quad & \text{in $\BR^3_\gamma$}.
\end{split}\right.
\end{equation}
Let $\wh H^{- 1, q} (\BR^3_\gamma)$ be the set of all $(\varphi_1, \varphi_2) \in L^q (\BR^3_\gamma) \times B^{2 - 1\slash q}_{q, q} (\pd \BR^3_\gamma)$ that satisfy the regularity property $(\varphi_1, \varphi_2) \in {}_0 \dot H^{- 1, q} (\BR^3_\gamma)$. With the notation
\begin{equation*}
\langle (\varphi_1, \varphi_2) \mid \phi \rangle_{\BR^3_\gamma} := - (\varphi_1 \mid \phi)_{\BR^3_\gamma} + (\varphi_2 \mid \phi)_{\pd \BR^3_\gamma} \quad \text{for any $\phi \in \dot H^{1, q'} (\BR^3_\gamma)$},
\end{equation*}
we have the conditions
\begin{equation*}
\langle (\varphi_1, \varphi_2) \mid \phi \rangle_{\BR^3_\gamma} := - (\varphi_1 \mid \phi)_{\BR^3_\gamma} + (\varphi_2 \mid \phi)_{\BR^3_\gamma} \quad \text{for any $\phi \in \dot H^{1, q'} (\BR^3_\gamma)$},
\end{equation*}
in which follows from the divergence equation and the divergence theorem. In view of \cite[Sec.~7.3.2]{PS16}, there exists a unique solution $(\bu, \pi)$ to~\eqref{eq-stokes-benthalf-slip} in maximal regularity class provided $\lvert \nabla_{x'} \gamma \rvert_{L^\infty (\BR^2)}$ is bounded above by a small constant $c$.
\begin{theo}
Suppose $1 < p, q < \infty$, $1\slash p < \delta \le 1$, $1\slash p + 1\slash(2 q) \notin \{\delta - 1\slash2, \delta\}$, $T > 0$, and $J = (0, T)$. Then there exists a constant $c$ such that the problem~\ref{eq-stokes-benthalf-slip} has a unique solution with regularity $\bu \in \BE_{1, \delta} (J; \BR^3_\gamma)$, $\pi \in \BE_{2, \delta} (J; \BR^3_\gamma)$ if and only if
\begin{enumerate}
\renewcommand{\labelenumi}{(\alph{enumi})}
\item $\bff \in \BF_{0, \delta} (J; \BR^3_\gamma)$;
\item $f_\mathrm{div} \in \BF_{1, \delta} (J; \BR^3_\gamma)$;
\item $\bh' \in \BF_{2, \delta}^2 (J; \pd \BR^3_\gamma)$ and $\bh' (0) = \mathrm{Tr}_{\pd \BR^3_\gamma} [P_{\BR^3_\gamma} (2 \mu \bD (\bu_0) \bn_{\pd \BR^3_\gamma})]$ if $1\slash p + 1\slash(2 q) < \delta - 1\slash2$;
\item $h_3 \in \BF_{3, \delta} (J; \pd \BR^3_\gamma)$ and $h_3 (0) = \mathrm{Tr}_{\pd \BR^3_\gamma} [u_{0, 3}]$ if $1\slash p + 1\slash(2 q) < \delta$;
\item $(f_\mathrm{div}, h_3) \in H^{1, p}_\delta (J; \wh H^{- 1, q} (\BR^3_\gamma))$;
\item $\bu_0 \in B^{2 (\delta - 1\slash p)}_{q, p} (\BR^3_\gamma)^3$ and $\dv \bu_0 = f_\mathrm{div} (0)$.				
\end{enumerate}
Furthermore, the solution map $(\bff, f_\mathrm{div}, \bh', h_3, \bu_0) \mapsto (\bu, \pi)$ is continuous between the corresponding spaces.	
\end{theo}

\subsubsection{Bent quarter spaces}
Let $\gamma \colon \BR \to \BR$ be a function of class $BC^3$ such that $\lvert \pd_1 \gamma \rvert_{L^\infty (\BR)} \le c$, where $c$ is a small constant determined later. We define a bent quarter space by
\begin{equation}
\label{def-Kgamma}
\BK^3_\gamma := \{(x_1, x_2, x_3) \mid x_1 \in \BR, \enskip x_2 > \gamma (x_1), \enskip x_3 > 0\}.
\end{equation}
Besides, we denote the boundaries of $\BK^3_\gamma$ by
\begin{equation*}
\begin{split}
\pd_2 \BK^3_\gamma & := \{(x_1, x_2, x_3) \mid x_1 \in \BR, \enskip x_2 = \gamma (x_1), \enskip x_3 > 0\}, \\
\pd_3 \BK^3_\gamma & := \{(x_1, x_2, x_3) \mid x_1 \in \BR, \enskip x_2 > \gamma (x_1), \enskip x_3 = 0\}
\end{split}
\end{equation*}
We also write $\pd \BK^3_\gamma := \pd_2 \BK^3_\gamma \cap \pd_3 \BK^3_\gamma$ if no confusion occurs. Besides, we set
\begin{equation*}
\bn_{\pd_2 \BK^3_\gamma} := b (x) (\pd_1 \gamma, - 1, 0)^\top, \qquad b (x) = \frac{1}{\sqrt{1 + \lvert \pd_1 \gamma \rvert^2}}, 
\end{equation*}
which denotes the outward unit normal field on $\pd_2 \BK^3_\gamma$. \par
First, we consider the Stokes equations in $\BK^3_\gamma$ with slip-slip boundary conditions
\begin{align}
\label{eq-slipslip-bent}
\left\{\begin{aligned}
\pd_t \bu - \mu \Delta \bu + \nabla \pi & = \bff, & \quad & \text{in $\BK^3_\gamma \times J$}, \\
\dv \bu & = f_\mathrm{div}, & \quad & \text{in $\BK^3_\gamma \times J$}, \\
P_{\pd_2 \BK^3_\gamma} (2 \mu \bD (\bu) \bn_{\pd_2 \BK^3_\gamma}) & = \bg', & \quad & \text{in $\pd_2 \BK^3_\gamma \times J$}, \\
\bu \cdot \bn_{\pd_2 \BK^3_\gamma} & = g_2, & \quad & \text{in $\pd_2 \BK^3_\gamma \times J$}, \\
\mu (\pd_3 u_m + \pd_m u_3) & = h_m, & \quad & \text{in $\pd_3 \BK^3_\gamma \times J$}, \\
u_3 & = h_3, & \quad & \text{in $\pd_3 \BK^3_\gamma \times J$}, \\
\bu (0) & = \bu_0, & \quad & \text{in $\BK^3_\gamma$},
\end{aligned}\right.
\end{align}
where $m = 1, 2$ and $\bg' = (g_1, 0, g_3)$. Since the equations \eqref{eq-slipslip-bent} is a perturbation of the half-space problem~\eqref{eq-stokes-half-slip}, we can show the existence of unique solution to \eqref{eq-slipslip-bent} provided $\lvert \pd_1 \gamma \rvert_{L^\infty (\BR)}$ is small enough. As we introduced before, let $\wh H^{- 1, q} (\BK^3_\gamma)$ be the set of all $(\varphi_1, \varphi_2, \varphi_3) \in L^q (\BK^3_\gamma) \times B^{2 - 1 \slash q}_{q, q} (\pd_2 \BK^3_\gamma)\times B^{2 - 1\slash q}_{q, q} (\pd_3 \BK^3_\gamma)$ with $(\varphi_1, \varphi_2, \varphi_3) \in \dot H^{- 1, q} (\BK^3_\gamma)$. Set
\begin{equation*}
\langle (\varphi_1, \varphi_2, \varphi_3) \mid \phi \rangle_{\BK^3_\gamma} := - (\varphi_1 \mid \phi)_{\BK^3_\gamma} + (\varphi_2 \mid \phi)_{\pd_2 \BK^3_\gamma} + (\varphi_3 \mid \phi)_{\pd_3 \BK^3_\gamma} \quad \text{for any $\phi \in \dot H^{1, q'} (\BK^3_\gamma)$}.
\end{equation*}
Then, by the divergence equation, we have the condition
\begin{equation*}
\langle (f_\mathrm{div}, g_2, h_3) \mid \phi \rangle_{\BK^3_\gamma} = - (\bu \mid \nabla \phi)_{\BK^3_\gamma} \quad \text{for any $\phi \in \dot H^{1, q'} (\BK^3_\gamma)$}.
\end{equation*}

\begin{theo}
\label{th-bent-slipslip}
Suppose that $p$, $q$, and $\delta$ satisfy \eqref{cond-pqdelta-model1}. Let $T > 0$ and $J = (0, T)$. Then there exists a unique solution $(\bu, \pi)$ to the problem \eqref{eq-slipslip-bent} with $\bu \in \BE_{1, \delta} (J; \BK^3_\gamma)$ and $\pi \in \BE_{2, \delta} (J; \BK^3_\gamma)$ if and only if
\begin{enumerate}
\renewcommand{\labelenumi}{(\alph{enumi})}
\item $\bff \in \BF_{0, \delta} (J; \BK^3_\gamma)$;
\item $f_\mathrm{div} \in \BF_{1, \delta} (J; \BK^3_\gamma)$;
\item $g_\ell \in \BF_{2, \delta} (J; \pd_2 \BK^3_\gamma)$, $\ell = 1, 3$;
\item $g_2 \in \BF_{3, \delta} (J; \pd_2 \BK^3_\gamma)$;
\item $h_m \in \BF_{2, \delta} (J; \pd_3 \BK^3_\gamma)$;
\item $h_3 \in \BF_{3, \delta} (J; \pd_3 \BK^3_\gamma)$;
\item $(f_\mathrm{div}, g_2, h_3) \in H^{1, p}_\delta (J; \wh H^{- 1, q} (\BK^3_\gamma))$;
\item $\bu_0 \in B^{2 (\delta - 1\slash p)}_{q, p} (\BK^3_\gamma)^3$ and $\dv \bu_0 = f_\mathrm{div} (0)$;
\item $\bg' (0) = \mathrm{Tr}_{\pd_2 \BK^3_\gamma} [P_{\pd_2 \BK^3_\gamma} (2 \mu \bD (\bu_0) \bn_{\pd_2 \BK^3_\gamma})]$ and $h_m (0) = \mathrm{Tr}_{\pd_3 \BK^3_\gamma} [\mu (\pd_3 u_{0, m} + \pd_m u_{0, 3})]$ if $1\slash p + 1\slash(2 q) < \delta - 1\slash2$;
\item $g_2 (0) = \mathrm{Tr}_{\pd_2 \BK^3_\gamma} [\bu_0 \cdot \bn_{\pd_2 \BK^3_\gamma}]$ and $h_3 (0) = \mathrm{Tr}_{\pd_3 \BK^3_\gamma} [u_{0, 3}]$ if $1\slash p + 1\slash(2 q) < \delta$;
\item $\mathrm{Tr}_{\pd \BK^3_\gamma} [\bg' (0)] = \mathrm{Tr}_{\pd \BK^3_\gamma} [P_{\pd_3 \BK^3_\gamma} (2 \mu \bD (\bu_0) \bn_{\pd_2 \BK^3_\gamma})]$ and $\mathrm{Tr}_{\pd \BK^3_\gamma} [h_m (0)] = \mathrm{Tr}_{\pd \BK^3_\gamma} [\mu (\pd_3 u_{0, m} + \pd_m u_{0, 3})]$ if $1\slash p + 1\slash q < \delta - 1\slash2$;
\item $\mathrm{Tr}_{\pd \BK^3_\gamma} [g_2 (0)] = \mathrm{Tr}_{\pd \BK^3_\gamma} [u_{0, 2}]$ and $\mathrm{Tr}_{\pd \BK^3_\gamma} [h_3 (0)] = \mathrm{Tr}_{\pd \BK^3_\gamma} [u_{0, 3}]$ if $1\slash p + 1\slash q < \delta$.		
\end{enumerate}
In addition, the solution $(\bu, \pi)$ depends continuously on the data in the corresponding spaces.	
\end{theo}

\begin{proof}
We first reduce the problem \eqref{eq-slipslip-bent} to the case $\bu_0 = \bff = 0$. To this end, let $e_{\BR^3} [\bu_0]$ and $e_{\BR^3} [\bff]$ denote extensions of $\bu_0$ and $\bff$ to all of $\BR^3$ in the class $B^{2 (\delta - 1\slash p)}_{q, p} (\BR^3)$ and $\BF_{0, \delta} (J; \BR^3)$, respectively, which solve the heat equation
\begin{equation*}
\left\{\begin{aligned}
\pd_t \wt \bu - \mu \Delta \wt \bu & = e_{\BR^3} [\bff], & \quad & \text{in $\BR^3 \times J$}, \\
e_{\BR^3} [\bu] (0) & = e_{\BR^3} [\bu_0], & \quad & \text{in $\BR^3$}
\end{aligned}\right.
\end{equation*}
to obtain a unique solution $\wt \bu \in \BF_{0, \delta} (J; \BR^3)$. If $\bu$ is a solution to \eqref{eq-slipslip-bent}, we see that the restricted function $\bu^* = \bu - \wt \bu$ solves
\begin{equation}
\label{eq-slipslip-bent-modi}
\left\{\begin{aligned}
\pd_t \bu^* - \mu \Delta \bu^* + \nabla \pi & = 0, & \quad & \text{in $\BK^3_\gamma \times J$}, \\
\dv \bu^* & = f_\mathrm{div}^*, & \quad & \text{in $\BK^3_\gamma \times J$}, \\
P_{\pd_2 \BK^3_\gamma} (2 \mu \bD (\bu^*) \bn_{\pd_2 \BK^3_\gamma}) & = \bg'^*, & \quad & \text{in $\pd_2 \BK^3_\gamma \times J$}, \\
\bu^* \cdot \bn_{\pd_2 \BK^3_\gamma} & = g_2^*, & \quad & \text{in $\pd_2 \BK^3_\gamma \times J$}, \\
\mu (\pd_3 u_m^* + \pd_m u_3^*) & = h_m^*, & \quad & \text{in $\pd_3 \BK^3_\gamma \times J$}, \\
u_3^* & = h_3^*, & \quad & \text{in $\pd_3 \BK^3_\gamma \times J$}, \\
\bu^* (0) & = 0, & \quad & \text{in $\BK^3_\gamma$},
\end{aligned}\right.
\end{equation}
with 
\begin{alignat*}4
f_\mathrm{div}^* & = f_\mathrm{div} - \dv \wt \bu \vert_{\BK^3_\gamma} , & \quad \bg'^* & = \bg' - P_{\pd_3 \BK^3_\gamma} (2 \mu \bD (\wt \bu \vert_{\BK^3_\gamma}) \bn_{\pd_2 \BK^3_\gamma}) , \\
g_2^* & = g_2 - \wt \bu \vert_{\BK^3_\gamma} \cdot \bn_{\pd_2 \BK^3_\gamma}, & \quad h_m^* & = h_m - \mu (\pd_3 \wt u_m \vert_{\BK^3_\gamma} + \pd_m \wt u_3 \vert_{\BK^3_\gamma}), & \quad h_3^* & = h_3 - \wt u_3 \vert_{\BK^3_\gamma}.
\end{alignat*}
Considering the time trace at $t = 0$, we have $(f_\mathrm{div}^*, g_2^*, h_3^*) \in H^{1, p}_\delta (J; \wh H^{- 1, q} (\BK^3_\gamma))$ vanishing at $t = 0$.  
\par 
Next, we transform \eqref{eq-slipslip-bent} to the problem in $\BK^3$. To this end, we introduce new variables
\begin{equation*}
\ov x = (\ov x_1, \ov x_2, \ov x_3) = (x_1, x_2 - \gamma (x_1), x_3) \quad \text{for $x \in \BK^3_\gamma$}.
\end{equation*}
Then, we define the new functions
\begin{equation*}
\begin{split}
\ov \bu (\ov x) & := \bu^* (\ov x_1, \ov x_2 + \gamma (\ov x_1), \ov x_3), \\
\ov \pi (\ov x) & := \pi (\ov x_1, \ov x_2 + \gamma (\ov x_1), \ov x_3),
\end{split}
\end{equation*}
where $(\bu^*, \pi)$ is a solution to \eqref{eq-slipslip-bent-modi}. In the same way, we also transform the given data $(f^*_\mathrm{div}, \bg'^*, g^*_2, \bh'^*, h^*_3)$ to $(\ov f_\mathrm{div}, \ov \bg', \ov g_2, \ov \bh', \ov h_3)$, where $\bh'^* = (h_1, h_2)$ and $\ov \bh' = (\ov h_1, \ov h_2)$. Besides, the differential operators are transformed into $\pd_{x_2}^s = \pd_{\ov x_2}^s$, $\pd_{x_3}^s = \pd^s_{\ov x_3}$ for $s = 1, 2$ and
\begin{equation*}
\begin{split}
\pd_{x_1} & = \pd_{\ov x_1} - (\pd_{x_1} \gamma) \pd_{\ov x_2}, \\
\pd_{x_1}^2 & = \pd_{\ov x_1}^2 - 2 (\pd_{x_1} \gamma) \pd_{\ov x_1} \pd_{\ov x_2} - (\pd_{x_1}^2 \gamma) \pd_{\ov x_2} + (\pd_{x_1} \gamma)^2 \pd^2_{\ov x_2}
\end{split}
\end{equation*}
Thus, we find that $(\ov \bu, \ov \pi)$ is a solution to the following problem
\begin{equation}
\label{eq-slipslip-bent-reduced}
\left\{\begin{aligned}
\pd_t \ov \bu - \mu \Delta_{\ov x} \ov \bu + \nabla_{\ov x} \ov \pi & = \bB_1 (\gamma, \ov \bu, \ov \pi), & \quad & \text{in $\BK^3 \times J$}, \\
\dv_{\ov x} \ov \bu & = \ov f_\mathrm{div} + B_2 (\gamma, \ov \bu), & \quad & \text{in $\BK^3 \times J$}, \\
\mu (\pd_{\ov x_2} \ov u_1 + \pd_{\ov x_1} \ov u_2) & = - \ov g_1 + \mu b (x) B_3 (\gamma, \ov \bu), & \quad & \text{in $\pd_2 \BK^3 \times J$}, \\
\mu (\pd_{\ov x_2} \ov u_3 + \pd_{\ov x_3} \ov u_2) & = - \ov g_3 + \mu b (x) B_4 (\gamma, \ov \bu), & \quad & \text{in $\pd_2 \BK^3 \times J$}, \\
\ov u_2 & = - \ov g_2 + \mu b (x) B_5 (\gamma, \ov \bu), & \quad & \text{in $\pd_2 \BK^3 \times J$}, \\
\mu (\pd_{\ov x_3} \ov u_1 + \pd_{\ov x_1} \ov u_3) & = \ov h_1 + B_6 (\gamma, \ov \bu), & \quad & \text{in $\pd_3 \BK^3 \times J$}, \\
\mu (\pd_{\ov x_3} \ov u_2 + \pd_{\ov x_2} \ov u_3) & = \ov h_2, & \quad & \text{in $\pd_3 \BK^3 \times J$}, \\
\ov u_3 & = \ov h_3, & \quad & \text{in $\pd_3 \BK^3 \times J$}, \\
\ov \bu (0) & = 0, & \quad & \text{in $\BK^3$},
\end{aligned}\right.
\end{equation}
where we have set
\begin{equation*}
\begin{split}
\bB_1 (\gamma, \ov \bu, \ov \pi) & = 2 (\pd_{x_1} \gamma) \pd_{x_1} \pd_{\ov x_2} \ov \bu + (\pd^2_{\ov x_1} \gamma) \pd_{\ov x_2} \ov \bu - \lvert \pd_{x_1} \gamma \rvert^2 \pd_{\ov x_2}^2 \ov \bu + (\pd_{\ov x_2} \ov \pi) (\pd_{x_1} \gamma, 0, 0)^\top, \\
B_2 (\gamma, \ov \bu) & = (\pd_{x_1} \gamma) \pd_{\ov x_2} \ov u_1, \\
B_3 (\gamma, \ov \bu) & = 2 (\pd_{x_1} \gamma) \pd_{\ov x_1}\ov u_1 + (\pd_{x_1} \gamma) (1 - \lvert \pd_{x_1} \gamma \rvert^2) \pd_{\ov x_2} \ov u_2 - \lvert \pd_{x_1} \gamma \rvert^2 \pd_{\ov x_2} \ov u_1 \\
& \quad - 2 (\pd_{x_1} \gamma) (1 - \lvert \pd_{x_1} \gamma \rvert^2) (\pd_{x_1} \gamma) \pd_{\ov x_2} \ov u_1 - (\pd_{x_1} \gamma)^2 \pd_{\ov x_1} \ov u_2 \\
& \quad - 2 (\pd_{x_1} \gamma)^3 \pd_{\ov x_1} \ov u_1 + \frac{\lvert \pd_{x_1} \gamma \rvert^2}{1 + \sqrt{1 + \lvert \pd_{x_1} \gamma \rvert^2}} \Big(\pd_{\ov x_2} \ov u_1 + \pd_{\ov x_1} \ov u_2\Big), \\
B_4 (\gamma, \ov \bu) & = (\pd_{x_1} \gamma) (\pd_{\ov x_3} \ov u_1 + \pd_{\ov x_1} \ov u_3) - \lvert \pd_{x_1} \gamma \rvert^2 \pd_{\ov x_3} \ov u_3 + \frac{\lvert \pd_{x_1} \gamma \rvert^2}{1 + \sqrt{1 + \lvert \pd_{x_1} \gamma \rvert^2}} \Big(\pd_{\ov x_2} \ov u_3 + \pd_{\ov x_3} \ov u_2\Big) \\
B_5 (\gamma, \ov \bu) & = (\pd_{x_1} \gamma) \ov u_1 + \frac{\lvert \pd_{x_1} \gamma \rvert^2}{1 + \sqrt{1 + \lvert \pd_{x_1} \gamma \rvert^2}} \ov u_2, \\
B_6 (\gamma, \ov \bu) & = \mu (\pd_{x_1} \gamma) (\pd_{\ov x_2} \ov u_3), 
\end{split}
\end{equation*}
and $\ov \bu = (\ov u_1, \ov u_2, \ov u_3)^\top$. Here, to derive $B_3$ and $B_4$, we have used the following observations: First, by the identity $(P_{\pd_2 \BK^3} w) \cdot \bn_{\pd_2 \BK^3_\gamma} = 0$, $w \in \BR^3$, we see that the second component of $P_{\pd_2 \BK^3_\gamma} w$ is redundant, i.e., the second component of $P_{\pd_2 \BK^3_\gamma} w$ is given by $(P_{\pd_2 \BK^3_\gamma} w) (\pd_{x_1} \gamma) e_1$. Next, we read the boundary condition $P_{\pd_2 \BK^3_\gamma} (2 \mu \bD (\bu^*) \bn_{\pd_2 \BK^3_\gamma}) = \bg'^*$ componentwise, i.e., this boundary condition is replaced by
\begin{equation*}
[P_{\pd_2 \BK^3_\gamma} (2 \mu \bD (\bu^*) \bn_{\pd_2 \BK^3_\gamma})] \cdot e_\ell = \bg'^* \cdot e_\ell
\end{equation*}
for $\ell = 1, 3$. The calculation
\begin{equation*}
2 \bD (\bu^*) = 2 \bD (\ov \bu) - \begin{pmatrix}
(\pd_{x_1} \gamma) \pd_{\ov x_2} \ov u_1 + (\pd_{\ov x_2} \ov u_1) \pd_{x_1} \gamma & (\pd_{\ov x_2} \ov u_2) \pd_{x_1} \gamma & (\pd_{\ov x_3} \ov u_3) \pd_{x_1} \gamma \\
(\pd_{\ov x_2} \ov u_2) \pd_{x_1} \gamma & 0 & 0 \\
(\pd_{\ov x_3} \ov u_3) \pd_{x_1} \gamma & 0 & 0 \\
\end{pmatrix}
\end{equation*}
implies that $2 \bD (\bu^*) \bn_{\pd_2 \BK^3_\gamma}$ is given by
\begin{equation*}
\begin{split}
2 \bD (\bu^*) \bn_{\pd_2 \BK^3_\gamma} \vert_1 & = b (x) \Big[- \Big(\pd_{\ov x_2} \ov u_1 + \pd_{\ov x_1} \ov u_2 \Big) + 2 (\pd_{x_1} \gamma) \pd_{\ov x_1} \ov u_1 + (\pd_{x_1} \gamma) \pd_{\ov x_2} \ov u_2 \\
& \quad - \lvert \pd_{x_1} \gamma \rvert^2 \pd_{\ov x_2} \ov u_1 - (\pd_{x_1} \gamma)^2 \pd_{\ov x_2} \ov u_1\Big], \\
2 \bD (\bu^*) \bn_{\pd_2 \BK^3_\gamma} \vert_2 & = b (x) \Big[- 2 \pd_{\ov x_2} \ov u_2 + (\pd_{x_1} \gamma) \Big(\pd_{\ov x_2} \ov u_1 + \pd_{\ov x_1} \ov u_2 \Big) - \lvert \pd_{x_1} \gamma \rvert^2 \pd_{\ov x_2} \ov u_2 \Big], \\
2 \bD (\bu^*) \bn_{\pd_2 \BK^3_\gamma} \vert_3 & = b (x) \Big[- \Big(\pd_{\ov x_2} \ov u_3 + \pd_{\ov x_3} \ov u_2\Big) + (\pd_{x_1} \gamma) \Big(\pd_{\ov x_3} \ov u_1 + \pd_{\ov x_1} \ov u_3\Big) - \lvert \pd_{x_1} \gamma \rvert^2 \pd_{\ov x_3} \ov u_3 \Big],
\end{split}
\end{equation*}
and hence we obtain
\begin{equation*}
\begin{split}
& [P_{\pd_2 \BK^3_\gamma} (2 \mu \bD (\bu^*) \bn_{\pd_2 \BK^3_\gamma})] \cdot e_1 \\
& \quad = b (x) \Big[- \Big(\pd_{\ov x_2} \ov u_1 + \pd_{\ov x_1} \ov u_2 \Big) + 2 (\pd_{x_1} \gamma) \pd_{\ov x_1}\ov u_1 + (\pd_{x_1} \gamma) (1 - \lvert \pd_{x_1} \gamma \rvert^2) \pd_{\ov x_2} \ov u_2 - \lvert \pd_{x_1} \gamma \rvert^2 \pd_{\ov x_2} \ov u_1 \\
& \quad \quad - 2 (\pd_{x_1} \gamma) (1 - \lvert \pd_{x_1} \gamma \rvert^2) (\pd_{x_1} \gamma) \pd_{\ov x_2} \ov u_1 - (\pd_{x_1} \gamma)^2 \pd_{\ov x_1} \ov u_2 - 2 (\pd_{x_1} \gamma)^3 \pd_{\ov x_1} \ov u_1 \Big], \\
& [P_{\pd_2 \BK^3_\gamma} (2 \mu \bD (\bu^*) \bn_{\pd_2 \BK^3_\gamma})] \cdot e_3 \\
& \quad = b (x) \Big[- \Big(\pd_{\ov x_2} \ov u_3 + \pd_{\ov x_3} \ov u_2\Big) + (\pd_{x_1} \gamma) \Big(\pd_{\ov x_3} \ov u_1 + \pd_{\ov x_1} \ov u_3\Big) - \lvert \pd_{x_1} \gamma \rvert^2 \pd_{\ov x_3} \ov u_3\Big].
\end{split}
\end{equation*}
Notice that all perturbation operators are linear and analytic with respect to $\ov \bu$. The perturbation operators can be estimated as follows:
\begin{equation*}
\begin{split}
\lvert \bB_1 (\gamma, \ov \bu) \rvert_{\BF_{0, \delta} (J; \BK^3)} & \le \lvert \pd_{x_1} \gamma \rvert_{L^\infty (\BR)} \Big((2 + \lvert \pd_{x_1} \gamma \rvert_{L^\infty (\BR)}) \lvert \nabla^2 \ov \bu \rvert_{\BF_{0, \delta} (J; \BK^3)} + \lvert \ov \pi \rvert_{\BE_{2, \delta} (J; \BK^3)} \Big) \\
& \quad + \lvert \pd^2_{x_1} \gamma \rvert_{L^\infty (\BR)} \lvert \nabla \ov \bu \rvert_{\BF_{0, \delta} (J; \BK^3)}, \\
\lvert B_2 (\gamma, \ov \bu) \rvert_{L^p_\delta (J; H^{1, q} (\BK^3))} & \le \lvert \pd_{x_1} \gamma \rvert_{L^\infty (\BR)} \lvert \nabla^2 \ov \bu \rvert_{\BF_{0, \delta} (J; \BK^3)} \\
& \quad + \Big(\lvert \pd_{x_1}^2 \gamma \rvert_{L^\infty (\BR)} + \lvert \pd_{x_1} \gamma \rvert_{L^\infty (\BR)} \Big) \lvert \nabla \ov \bu \rvert_{\BF_{0, \delta} (J; \BK^3)}, \\
\lvert \pd_t B_2 (\gamma, \ov \bu) \rvert_{L^p_\delta (J; \dot H^{- 1, q} (\BK^3))} & \le \lvert \pd_{x_1} \gamma \rvert_{L^\infty (\BR)} \lvert \pd_t \ov \bu \rvert_{\BF_{0, \delta} (J; \BK^3)}, \\
\lvert b (x) B_{3, j} (\gamma, \ov \bu) \rvert_{F^{1\slash2 - 1\slash(2 q)}_{p, q, \delta} (J; L^q (\pd_2 \BK^3))} & \le C \lvert \pd_{x_1} \gamma \rvert_{L^\infty (\BR)} \sum_{\tau = 0}^3 \lvert \pd_{x_1} \gamma \rvert_{L^\infty (\BR)}^\tau \lvert \ov \bu \rvert_{\BE_{1, \delta} (J; \BK^3)} \\
\lvert b (x) B_{3, 3} \rvert_{F^{1\slash2 - 1\slash(2 q)}_{p, q, \delta} (J; L^q (\pd_2 \BK^3))} & \le C \lvert \pd_{x_1} \gamma \rvert_{L^\infty (\BR)} (1 + \lvert \pd_{x_1} \gamma \rvert_{L^\infty (\BR)}) \lvert \ov \bu \rvert_{\BE_{1, \delta} (J; \BK^3)}, \\
\lvert b (x) B_4 (\gamma, \ov \bu) \rvert_{F^{1 - 1\slash(2q)}_{p, q,\delta} (J; L^q (\pd_3 \BK^3))} & \le C \lvert \pd_{x_1} \gamma \rvert_{L^\infty (\BR)} \lvert \ov \bu \rvert_{\BE_{1, \delta} (J; \BK^3)}, \\
\lvert B_{5, j} (\gamma, \ov \bu) \rvert_{F^{1\slash2 - 1\slash(2q)}_{p, q, \delta} (J; L^q (\pd_3 \BK^3))} & \le C \lvert \pd_{x_1} \gamma \rvert_{L^\infty (\BR)} \lvert \ov \bu \rvert_{\BE_{1, \delta} (J; \BK^3)},
\end{split}
\end{equation*}
where constants $C$ are independent of $T$. Here, we use the trace theorem (cf. Lindemulder~\cite[pp.~88]{L19}) to derive the estimates
\begin{equation*}
\begin{split}
\lvert \nabla \ov \bu \rvert_{F^{1\slash2 - 1\slash(2q), p}_{p, q, \delta} (J; L^q (\pd_2 \BK^3))} & \le C \lvert \ov \bu \rvert_{\BF_{2, \delta} (J; \pd_2 \BK)} \le C \lvert \ov \bu \rvert_{\BE_{1, \delta} (J; \BK^3)}, \\
\lvert \nabla \ov \bu \rvert_{F^{1\slash2 - 1\slash(2q), p}_{p, q, \delta} (J; L^q (\pd_3 \BK^3))} & \le C \lvert \ov \bu \rvert_{\BF_{2, \delta} (J; \pd_3 \BK)} \le C \lvert \ov \bu \rvert_{\BE_{1, \delta} (J; \BK^3)}, \\
\lvert \ov \bu \rvert_{F^{1 - 1\slash(2q), p}_{p, q, \delta} (J; L^q (\pd_3 \BK^3))} & \le C \lvert \ov \bu \rvert_{\BF_{3, \delta} (J; \pd_3 \BK)} \le C \lvert \ov \bu \rvert_{\BE_{1, \delta} (J; \BK^3)}.
\end{split}
\end{equation*}
Notice that constants appeared in these estimates are independent of $T$ because $\ov \bu$ vanishes at $t = 0$. To obtain the estimates in $L^p_\delta (J; B^{1 - 1\slash q}_{q, q} (S))$ and $L^p_\delta (J; B^{2 - 1\slash q}_{q, q} (S))$, $S \in \{\pd_2 \BK^3, \pd_3 \BK^3\}$, we use the estimate
\begin{equation*}
\begin{split}
\lvert a \psi \rvert_{B^s_{q, q} (\BR^2)} & \le \lvert a \rvert_{L^\infty (\BR^2)} \lvert \psi \rvert_{B^s_{q, q} (\BR^2)} +  C \lvert \psi \rvert_{B^s_{q, q} (\BR^2)}^\theta \lvert \psi \rvert_{L^q (\BR^2)}^{1 - \theta},
\end{split}
\end{equation*}
where $0 < s < 2$, $1 \le q \le \infty$, $0 \le \theta < 1$, $a \in C^2 (\BR^2)$, and $\psi \in B^s_{q, q} (\BR^2)$, see \cite[Lem.~6.2.8]{PS16}. Here, the constant $C$ depends linearly on $\lvert a \rvert_{B^2_{\infty, \infty} (\BR^2)}$.\footnote{It is known that $C^2 (\BR^2)$ embeds into $B^2_{\infty, \infty} (\BR^2)$. Besides, from $0 < s < 2$, it holds $B^2_{\infty, \infty} (\BR^2) \hookrightarrow B^s_{\infty, q} (\BR^2)$ for all $1 \le q \le \infty$, and hence we obtain $C^2 (\BR^2) \hookrightarrow B^s_{\infty, q}(\BR^2)$ for $1 \le q \le \infty$ and $0 < s < 2$. For the embedding properties, we refer to Triebel~\cite{T95}.} Choosing $\theta$ small such that $\theta \le \lvert \pd_{x_1} \gamma \rvert_{L^\infty (\BR)}$, we observe
\begin{equation*}
\begin{split}
& \lvert b (x) B_{3, 1} (\gamma, \ov \bu) \rvert_{L^p_\delta (J; B^{1 - 1\slash q}_{q, q} (\pd_2 \BK^3))} \\
& \quad \le C \sum_{\tau = 0}^3 \lvert \pd_{x_1} \gamma \rvert_{L^\infty (\BR)}^\tau \Big(\lvert \pd_{x_1} \gamma \rvert_{L^\infty (\BR)} \lvert \nabla \ov \bu \rvert_{L^p_\delta (J; B^{1 - 1\slash q}_{q, q} (\pd_2 \BK^3))} + C_\gamma \lvert \nabla \ov \bu \rvert_{L^p_\delta (J; L^q (\pd_2 \BK^3))} \Big), \\
& \lvert b (x) B_{3, 3} (\gamma, \ov \bu) \rvert_{L^p_\delta (J; B^{1 - 1\slash q}_{q, q} (\pd_2 \BK^3))} \\
& \quad \le C \Big(1 + \lvert \pd_{x_1} \gamma \rvert_{L^\infty (\BR)} \Big) \Big(\lvert \pd_{x_1} \gamma \rvert_{L^\infty (\BR)} \lvert \nabla \ov \bu \rvert_{L^p_\delta (J; B^{1 - 1\slash q}_{q, q} (\pd_2 \BK^3))} + C_\gamma \lvert \nabla \ov \bu \rvert_{L^p_\delta (J; L^q (\pd_2 \BK^3))} \Big), \\
& \lvert b (x) B_4 (\gamma, \ov \bu) \rvert_{L^p_\delta (J; B^{2 - 1\slash q}_{q, q} (\pd_2 \BK^3))} \le C \Big(\lvert \pd_{x_1} \gamma \rvert_{L^\infty (\BR^3)} \lvert \ov \bu \rvert_{L^p_\delta (J; B^{2 - 1\slash q}_{q, q} (\pd_2 \BK^3))} + C_\gamma \lvert \ov \bu \rvert_{L^p_\delta (J; L^q (\pd_2 \BK^3))}  \Big), \\
& \lvert B_5 (\gamma, \ov \bu) \rvert_{L^p_\delta (J, B^{1 - 1\slash q}_{q, q} (\pd_3 \BK^3))} \le C \Big(\lvert \pd_{x_1} \gamma \rvert_{L^\infty (\BR)} \lvert \nabla \ov \bu \rvert_{L^p_\delta (J; B^{1 - 1\slash q}_{q, q} (\pd_3 \BK^3))} + C_\gamma \lvert \nabla \ov \bu \rvert_{L^p_\delta (J; L^q (\pd_3 \BK^3))}\Big),
\end{split}
\end{equation*}
where the constants $C$ and $C_\gamma$ are independent of $T$. Besides, the constant $C_\gamma$ depends on $\gamma$ and $C_\gamma \to 0$ as $\lvert \pd_{x_1} \gamma \rvert_{L^\infty (\BR)} \le c \to 0$. \par
It remains to consider $(B_2 (\gamma, \ov \bu), b(x) B_4 (\gamma, \ov \bu), 0)$ in $H^{1, p}_\delta (J; \wh H^{- 1, q} (\BK^3))$. Integrating by parts with respect to the variable $\ov x_2$, it holds
\begin{equation*}
\begin{split}
\langle (B_2 (\gamma, \ov \bu), b(x) B_4 (\gamma, \ov \bu), 0) \mid \phi \rangle_{\BK^3} & = - (B_2 (\gamma, \ov \bu) \mid \phi)_{\BK^3} + (b(x) B_4 (\gamma, \ov \bu) \mid \phi)_{\pd_2 \BK^3} \\
& = ((\pd_{x_1} \gamma) \ov u_1 \mid \pd_2 \phi)_{\BK^3}
\end{split}
\end{equation*}
for any $\phi \in \dot H^{1, q'} (\BK^3)$, and thus
it is clear that
\begin{equation*}
\lvert (B_2 (\gamma, \ov \bu), B_4 (\gamma, \ov \bu), 0) \rvert_{H^{1, p}_\delta (J; \wh H^{- 1, q} (\BK^3))} \le C \lvert \pd_{x_1} \gamma \rvert_{L^\infty (\BR)} \lvert \ov \bu \rvert_{\BE_{1, \delta} (J; \BK^3)}. 
\end{equation*}
\par
We now solve \eqref{eq-slipslip-bent-reduced}. Let $Z (J; \BK^3) := \BE_{1, \delta} (J; \BK^3) \times \BE_{2, \delta} (J; \BK^3)$ be the solution space, while $Y (J; \BK^3)$ be the product space of given data. In addition, let ${}_0 Z (J; \BK^3)$ and ${}_0 Y (J; \BK^3)$ denote the solution space $Z (J; \BK^3)$ and the data space $Y (J; \BK^3)$ with vanishing time trace at $t = 0$. Define
\begin{equation*}
\begin{split}
z & = (\ov \bu, \ov \pi) \in {}_0 Z (J; \BK^3), \\
F & := (0, \ov f_\mathrm{div}, - \ov \bg', - \ov g_2, \ov \bh', 0) \in {}_0 Y (J; \BK^3), \\
B z & = (\bB_1 (\gamma, \ov \bu, \ov \pi), B_2 (\gamma, \ov \bu), b (x) \bB_3 (\gamma, \ov \bu), b (x) B_4 (\gamma, \ov \bu), B_5 (\gamma, \ov \bu), 0) \colon {}_0 Z (J; \BK^3) \to {}_0 Y (J; \BK^3),
\end{split}
\end{equation*}
where $\bB_3 (\gamma, \ov \bu) = (B_{3, 1} (\gamma, \ov \bu), B_{3, 3} (\gamma, \ov \bu))$. Denoting the left-hand side of \eqref{eq-slipslip-bent-reduced} by $L$, we see that $L$ is isomorphism from ${}_0 Z (J; \BK^3)$ to ${}_0 Y (J; \BK^3)$ and may rewrite \eqref{eq-slipslip-bent-reduced} in the abstract form
\begin{equation}
\label{eq-slipslip-bent-abstract}
L z = B z + F.
\end{equation}
Recalling the estimates for the perturbation operators, we find
\begin{equation}
\label{est-pertubation-bent}
\lvert B z \rvert_{Y (J; \BK^3)} \le C \lvert \pd_{x_1} \gamma \rvert_{L^\infty (\BR)} \lvert z \rvert_Z + M \lvert \bu \rvert_{L^p_\delta (J; H^{1 + 1\slash q, q} (\BK^3))}
\end{equation}
with some constants $C, M > 0$ independent of $T$. By the mixed derivative theorem (cf. \cite[Ch.~4]{PS16}) and the Sobolev embedding (cf. \cite[Cor.~1.4]{MV12}), we have
\begin{equation*}
{}_0 \BE_{1, \delta} (J; \BK^3) \hookrightarrow {}_0 H^{1\slash2 - 1\slash q, p}_\delta (J; H^{1 + 1\slash q, q} (\BK^3)) \hookrightarrow L^{2 p}_\delta (J; H^{1 + 1\slash q, q} (\BK^3)),
\end{equation*}
where the embedding constants are independent of $T$. Hence, the H\"older inequality yields
\begin{equation*}
\lvert \ov \bu \rvert_{L^p_\delta (J; H^{1 + 1\slash q, q} (\BK^3))} \le T^{1\slash(2 p)} \lvert \bu \rvert_{L^{2 p}_\delta (J; H^{1 + 1\slash q, q} (\BK^3))} \le C T^{1\slash(2 p)} \lvert \ov \bu \rvert_{\BE_{1, \delta} (J; \BK^3)},
\end{equation*}
where $C > 0$ does not depend on $T$. Let $c$ be a given (small) constant and assume $\lvert \nabla_{x_1} \gamma \rvert_{L^\infty (\BR)} \le c$. Then, combined with \eqref{est-pertubation-bent}, the Neumann series argument implies that there exists small $T > 0$ such that~\eqref{eq-slipslip-bent-abstract} admits a unique solution $(\ov \bu, \ov \pi) \in {}_0 Z (J; \BK^3)$. Since $J = [0, T]$ is compact, we can solve \eqref{eq-slipslip-bent-reduced} for $J = [0, T]$ by repeating these arguments finitely many times, where $T > 0$ is now arbitrary. This completes the proof of the theorem.
\end{proof}
We next deal with the Stokes equations in $\BK^3_\gamma$ with slip-free boundary conditions:
\begin{align}
\label{eq-slipfree-bent}
\left\{\begin{aligned}
\pd_t \bu - \mu \Delta \bu + \nabla \pi & = \bff, & \quad & \text{in $\BK^3_\gamma \times J$}, \\
\dv \bu & = f_\mathrm{div}, & \quad & \text{in $\BK^3_\gamma \times J$}, \\
P_{\pd_2 \BK^3_\gamma} (2 \mu \bD (\bu) \bn_{\pd_2 \BK^3_\gamma}) & = \bg', & \quad & \text{in $\pd_2 \BK^3_\gamma \times J$}, \\
\bu \cdot \bn_{\pd_2 \BK^3_\gamma} & = g_2, & \quad & \text{in $\pd_2 \BK^3_\gamma \times J$}, \\
\pd_t \eta - u_3 & = d, & \quad & \text{in $\pd_3 \BK^3_\gamma \times J$}, \\
\mu (\pd_3 u_m + \pd_m u_3) & = k_m, & \quad & \text{in $\pd_3 \BK^3_\gamma \times J$}, \\
2 \mu \pd_3 u_3 - \pi - \sigma \Delta_{\pd_3 \BK^3} \eta & = k_3, & \quad & \text{in $\pd_3 \BK^3_\gamma \times J$}, \\
\bu (0) & = \bu_0, & \quad & \text{in $\BK^3_\gamma$}, \\
\eta (0) & = \eta_0, & \quad & \text{on $\pd_3 \BK^3_\gamma$}, \\
\end{aligned}\right.
\end{align}
where $m = 1, 2$ and $\bg' = (g_1, 0, g_3)^\top$. Let the space $\wh H^{- 1, q}_{\pd_2 \BK^3} (\BK^3)$ be the set of all $(\varphi_1, \varphi_2) \in L^q (\BK^3_\gamma) \times B^{2 - 1\slash q}_{q, q} (\pd_2 \BK^3_\gamma)$ satisfying $(\varphi_1, \varphi_2) \in \dot H^{- 1, q}_{\pd_2 \BK^3} (\BK^3_\gamma)$. Setting
\begin{equation*}
\langle (\varphi_1, \varphi_2) \mid \phi \rangle_{\BK^3_\gamma} := - (\varphi_1 \mid \phi)_{\BK^3_\gamma} + (\varphi_2 \mid \phi)_{\pd_2 \BK^3_\gamma} \quad \text{for any $\phi \in \dot H^{1, q'}_{\pd_3 \BK^3_\gamma} (\BK^3_\gamma)$},
\end{equation*}
we observe
\begin{equation*}
\langle (f_\mathrm{div}, g_2) \mid \phi \rangle_{\BK^3_\gamma} = - (\bu \mid \nabla \phi)_{\BK^3_\gamma} \quad \text{for any $\phi \in \dot H^{1, q'}_{\pd_3 \BK^3_\gamma} (\BK^3_\gamma)$}
\end{equation*}
as follows from the divergence equation. The next theorem can be proved in the same manner as in the proof of Theorem~\ref{th-bent-slipslip}, and hence we may omit the details.
\begin{theo}
Let $T > 0$ and $J = (0, T)$. Assume that $p$, $q$, and $\delta$ satisfy \eqref{cond-pqdelta-model2}. The problem \eqref{eq-slipfree-bent} admits a unique solution $(\bu, \pi, \eta)$ with $\bu \in \BE_{1, \delta} (J; \BK^3_\gamma)$, $\pi \in \BE_{2, \delta} (J; \BK^3_\gamma)$, $\mathrm{Tr}_{\pd_3 \BK^3_\gamma} [\pi] \in \BE_{3, \delta} (J; \pd_3 \BK^3_\gamma)$, $\eta \in \BE_{4, \delta} (J; \pd_3 \BK^3_\gamma)$, if and only if
\begin{enumerate}
\renewcommand{\labelenumi}{(\alph{enumi})}
\item $\bff \in \BF_{0, \delta} (J; \BK_\gamma^3)$;
\item $f_\mathrm{div} \in \BF_{1, \delta} (J; \BK_\gamma^3)$;
\item $g_\ell \in \BF_{2, \delta} (J; \pd_2 \BK_\gamma^3)$ for $\ell = 1, 3$;
\item $g_2 \in \BF_{3, \delta} (J; \pd_2 \BK_\gamma^3)$;	
\item $d \in \BF_{3, \delta} (J; \pd_3 \BK_\gamma^3)$;
\item $k_j \in \BF_{2, \delta} (J; \pd_3 \BK_\gamma^3)$ for $j = 1, 2, 3$;
\item $(f_\mathrm{div}, g_2) \in H^{1, p}_\delta (J; \wh H^{- 1, q}_{\pd_2 \BK^3_\gamma} (\BK^3_\gamma))$;
\item $\bu_0 \in B^{2 (\delta - 1\slash p)}_{q, p} (\BK_\gamma^3)^3$ and $\dv \bu_0 = f_\mathrm{div} (0)$;
\item $\eta_0 \in B^{2 + \delta - 1\slash p - 1\slash q}_{q, p} (\pd_3 \BK_\gamma^3)$;
\item $\bg' (0) = \mathrm{Tr}_{\pd_2 \BK^3_\gamma} [P_{\pd_2 \BK^3_\gamma} (2 \mu \bD (\bu_0) \bn_{\pd_2 \BK^3_\gamma})]$ and $k_m (0) = \mathrm{Tr}_{\pd_3 \BK_\gamma^3} [\mu (\pd_3 u_{0, m} + \pd_m u_{0, 3})]$ if $1\slash p + 1\slash(2 q) < \delta - 1\slash2$;
\item $\mathrm{Tr}_{\pd \BK^3_\gamma} [\bg' (0)] = \mathrm{Tr}_{\pd \BK^3_\gamma} [P_{\pd_3 \BK^3_\gamma} (2 \mu \bD (\bu_0) \bn_{\pd_2 \BK^3_\gamma})]$ and $\mathrm{Tr}_{\pd \BK_\gamma^3} [k_m (0)] = \mathrm{Tr}_{\pd \BK_\gamma^3} [\mu (\pd_3 u_{0, m} + \pd_m u_{0, 3})]$ if $1\slash p + 1\slash q < \delta - 1\slash2$.	
\end{enumerate}
Furthermore, the solution map is continuous between the corresponding spaces.
\end{theo}

\subsection{Regularity of the pressure}
In general, the pressure $\pi$ has no more regularity than one given in Theorem~\ref{th-MR-linear}. However, it is possible to obtain additional time-regularity for $\pi$ in a special situation.
\begin{lemm}
\label{lemm-regularity-pressure}
Suppose the assumptions given in Theorem~\ref{th-MR-linear}. Let $(\bu, \pi, \eta)$ be the solutions to \eqref{eq-linear} with
\begin{equation*}
\bu_0 = \eta_0 = f_\mathrm{div} = 0 \quad \text{in $\Omega_*$}, \quad u_3 = 0 \quad \text{on $\Gamma_*$},
\end{equation*}
for a.e. $t \in J$ and $\bff \in {}_0 H^{\alpha, p}_\delta (J; L^q (\Omega_*)^3)$ for some $\alpha \in (0, 1\slash2 - 1\slash(2 q))$. Then the following assertions hold.
\begin{enumerate}
\item If $\Omega_*$ is bounded domain given by \eqref{def-domain}, then it holds $\pi \in {}_0 H^{\alpha, p}_\delta (J; L^q (\Omega_*))$ possessing the estimate
\begin{equation*}
\lvert \pi \rvert_{H^{\alpha, p}_\delta (J; L^q (\Omega_*))} \le C \Big(\lvert \bu \rvert_{\BE_{1, \delta} (J; \Omega_*)} + \lvert \mathrm{Tr}_{\Gamma_*} [\pi] \rvert_{\BE_{3, \delta} (J; \Gamma_*)} + \lvert \bff \rvert_{H^{\alpha, p}_\delta (J; L^q (\Omega_*))} \Big).
\end{equation*}
Here, the constant $C$ does not depend on the length of the interval $J$.
\item If $\Omega_*$ is a full space, a (bent) half space, or a (bent) quarter space, then $(\pi)_K \in {}_0 H^{\alpha, p}_\delta (J; L^q (K))$ for each bounded domain $\Omega_*^R \subset \ov{\Omega_*}$ with $\Omega_*^R = \Omega_* \cap B (0, R)$, $R > 0$ large. In addition, it holds the estimate
\begin{equation*}
\lvert P_{0 R} \pi \rvert_{H^{\alpha, p}_\delta (J; L^q (\Omega_*))} \le C \Big(\lvert \bu \rvert_{\BE_{1, \delta} (J; \Omega_*)} + \lvert \mathrm{Tr}_{\Gamma_*} [\pi] \rvert_{\BE_{3, \delta} (J; \Gamma_*)} + \lvert \bff \rvert_{H^{\alpha, p} (J; L^q (\Omega_*))} \Big) 
\end{equation*}
with a constant $C$ independent of the length of the interval $J$. Here, $P_{0 R} \pi$ denotes the mean zero part of $\pi$ with respect to $\Omega_*^R$ in case the pressure $\pi$ does not appear on the boundary and $P_{0 R} = \bI$ otherwise.
\end{enumerate}
\end{lemm}	
\begin{proof}
(1) Let $\varphi \in L^{q'} (\Omega_*)$ be fixed with mean zero. Consider the elliptic problem
\begin{equation}
\label{pressure-elliptic}
\left\{\begin{aligned}
\Delta \phi & = \varphi, & \quad & \text{in $\Omega_*$}, \\
\phi & = 0, & \quad & \text{on $\Gamma_*$}, \\
\bn_{\Sigma_*} \cdot \nabla \phi & = 0, & \quad & \text{on $\Sigma_*$}, \\
\bn_B \cdot \nabla \phi & = 0, & \quad & \text{on $B$}. 
\end{aligned}\right.
\end{equation}
By Lemma~\ref{lem-A.3}, we have a unique solution $\phi \in H^{2, q} (\Omega_*)$. Then, from integration by parts, we observe
\begin{equation*}
\begin{split}
(\pi \mid \varphi)_{\Omega_*} & = (\pi \mid \Delta \phi)_{\Omega_*} \\
& = - (\nabla \pi \mid \nabla \phi)_{\Omega_*} + (\pi \mid \bn_{\Gamma_*} \cdot \nabla \phi)_{\Gamma_*} \\
& = - (\mu \Delta \bu \mid \nabla \phi)_{\Omega_*} + (\pd_t \bu - \bff \mid \nabla \phi)_{\Omega_*} + (\pi \mid \bn_{\Gamma_*} \cdot \nabla \phi)_{\Gamma_*} \\
& = \int_{\Omega_*} \mu \nabla \bu \colon \nabla^2 \phi \dx - (\mu \bn_{\pd \Omega_*} \cdot \nabla \bu \mid \nabla \phi)_{\pd \Omega_*} - (\bff \mid \nabla \phi)_{\Omega_*} + (\pi \mid \bn_{\Gamma_*} \cdot \nabla \phi)_{\Gamma_*},
\end{split}
\end{equation*}
where we have used $(\bu \mid \nabla \phi)_{\Omega_*} = 0$. According to the mixed derivative theorem and the trace theory~(cf.~\cite[Sec.~4.5]{PS16}), we have $\nabla \bu \in {}_0 H^{1\slash2, p}_{\delta} (J; L^q (\Omega_*)^{n \times n})$, $\pd_k u_\ell \in {}_0 F^{1\slash2 - 1\slash(2 q)}_{p, q, \delta} (J; L^q (\pd \Omega_*))$ for $k, \ell \in \{1, 2, 3\}$, and $\mathrm{Tr}_{\Gamma_*} [\pi] \in {}_0 F^{1\slash2 - 1\slash(2 q)}_{p, q, \delta} (J; L^q (\Gamma_*))$. As $\bu_0 = 0$, we may apply the fractional time derivative $\pd_t^\alpha$ to the above identity, and hence, taking the supremum of the left-hand side over all $\varphi \in L^{q'} (\Omega_*)$ with $\lvert \varphi \rvert_{L^{q'} (\Omega_*)} \le 1$, we arrive at
\begin{equation*}
\begin{split}
\lvert \pd_t^\alpha \pi \rvert_{L^p_\delta (J; L^q (\Omega_*))} & \le C \Big(\lvert \pd^\alpha_t \nabla \bu \rvert_{L^p_\delta (J; L^q (\Omega_*))} + \lvert \pd^\alpha_t (\bn_{\pd \Omega_*} \cdot \nabla \bu) \rvert_{L^p_\delta (J; L^q (\pd \Omega_*))} \\
& \qquad + \lvert \pd_t^\alpha \mathrm{Tr}_{\Gamma*} [\pi] \rvert_{L^p_\delta (J; L^q (\Gamma_*))} + \lvert \pd_t^\alpha \bff \rvert_{L^p_\delta (J; L^q (\Omega_*))} \Big) \\
& \le C \Big(\lvert \bu \rvert_{\BE_{1, \delta} (J; \Omega_*)} + \lvert \mathrm{Tr}_{\Gamma_*} [\pi] \rvert_{\BE_{3, \delta} (J; \Gamma_*)} + \lvert \bff \rvert_{H^{\alpha, p}_\delta (J; L^q (\Omega_*))} \Big)
\end{split}
\end{equation*}
for each $\alpha \in (0, 1\slash2 - 1\slash(2 q))$ because it holds
\begin{equation*}
{}_0 F^{1\slash2 - 1\slash(2 q)}_{p, q, \delta} (J; L^q (\Gamma_*)) \hookrightarrow {}_0 F^{1\slash2 - 1\slash(2 q) - \varepsilon}_{p, 1, \delta} (J; L^q (\Gamma_*)) \hookrightarrow {}_0 H^{1\slash2 - 1\slash(2 q) - \varepsilon, p}_\delta (J; L^q (\Gamma_*))
\end{equation*}
for arbitrary $\varepsilon > 0$, see, e.g., \cite{MV14} for the detail. This shows $\pi \in {}_0 H^{\alpha, p}_\delta (J; L^q (\Omega_*))$. \par
(2) The proof of the second statement is essentially the same as in the proof of the first assertion. Let $\varphi \in L^{q'} (K)$ be fixed with mean zero.  Extend $P_{0 R} \varphi$ by zero to $\wt \varphi \in L^{q'} (\Omega_*)$. Then, by Lemma~\ref{lem-A.3}, the elliptic problem \eqref{pressure-elliptic} with $\wt \varphi$ as an inhomogeneity in the first equation is uniquely solvable. Especially the solution admits the regularity $\nabla \phi \in H^{1, q'} (\Omega_*)$ and the estimate
\begin{equation*}
\lvert \nabla \phi \rvert_{L^{q'} (\Omega_*)} + \lvert \nabla^2 \phi \rvert_{L^{q'} (\Omega_*)} \le C \lvert \wt \varphi \rvert_{L^{q'} (\Omega_*)} \le C \lvert \varphi \rvert_{L^{q'} (K)}. 
\end{equation*}
Furthermore, we have $(P_{0 R} \varphi \mid \varphi)_K = (P_{0 R} \varphi \mid \wt \varphi)_{\Omega_*}$. Thus, employing the same argument employed in the proof of the first assertion yields the desired result.	
\end{proof}

\subsection{Reduction of the data}
\label{sec-reduction_data}
It is convenient to reduce the given data in \eqref{eq-linear} to the case
\begin{equation*}
\bff = f_\mathrm{div} = \bu_0 = \eta_0 = \bg \cdot \bn_{\Sigma_*} = h_3 = 0.
\end{equation*}
In the following, we show that these reductions can be observed. We first extend $\eta_0 \in B^{2 + \delta - 1\slash p - 1\slash q}_{q, p} (\Gamma_*)$ and $u_{0, 3} \vert_{\Gamma_*}, d \vert_{t = 0} \in B^{2 (\delta - 1\slash p) - 1\slash q}_{q, p} (\Gamma_*)$ to $\wt \eta_0 \in B^{2 + \delta - 1\slash p - 1\slash q}_{q, p} (\pd \BR^3_+)$ and $\wt u_{0, 3} \vert_{\pd \BR^3_+}, \wt d \vert_{t = 0} \in B^{2 (\delta - 1\slash p) - 1\slash q}_{q, p} (\pd \BR^3_+)$, respectively. Using these functions, define
\begin{equation*}
\begin{split}
\wt \eta_* (t) & := \Big[2 e^{- (\bI - \Delta_{\pd \BR^3_+})^{1\slash2} t} - e^{- 2 (\bI - \Delta_{\pd \BR^3_+})^{1\slash2} t} \Big] \eta_0 \\
& \quad + \Big[e^{- (\bI - \Delta_{\pd \BR^3_+}) t} - e^{- 2 (\bI - \Delta_{\pd \BR^3_+}) t} \Big] (\bI - \Delta_{\pd \BR^3_+})^{- 1} (\wt u_{0, 3} \vert_{\pd \BR^3_+} + \wt d \vert_{t = 0}).
\end{split}
\end{equation*}
for $t \ge 0$. Then, as we discussed in the proof of Theorem~\ref{th-model-half-free}, the function $\wt \eta_*$ has the regularity $\wt \eta_* \in \BE_{4, \delta} (J; \pd \BR^3_+)$ satisfying $\wt \eta_* (0) = \wt \eta_0$ and $\pd_t \wt \eta_* (0) = \wt u_{0, 3} \vert_{\pd \BR^3_+} + \wt d \vert_{t = 0}$. Setting $\eta_* = \wt \eta_* \vert_{\Gamma_*}$, we find that $\eta_* (0) = \eta_0$ and $\pd_t \eta_* (0) = u_{0, 3} \vert_{\Gamma_*} + d \vert_{t = 0}$. Hence, if we set $\eta_5 := \eta - \eta_*$, we observe $\eta_5 (0) = \pd_t \eta_5 (0) = 0$. \par
Next, let $q_0 := - 2 \mu \pd_3 u_{0, 3} \vert_{\Gamma_*} + \sigma \Delta_{\Gamma_*} \eta_0 + k_3 \vert_{t = 0} \in B^{2 (\delta - 1\slash p) - 1 - 1\slash q}_{q, p} (\Gamma_*)$. Here, it holds
\begin{equation*}
k_3 \vert_{t = 0} \in B^{2 (\delta - 1\slash p) - 1 - 1\slash q}_{q, p} (\Gamma_*)
\end{equation*}
because the trace operator $\mathrm{Tr}_{t = 0}$ and the boundary operator $\mathrm{Tr}_{\Gamma_*} \pd_{x'}$ may commute (cf. Lindemulder~\cite[pp.~88]{L19}). We extend $q_0$ to some $\wt q_0 \in B^{2 (\delta - 1\slash p) - 1 - 1\slash q}_{q, p} (\pd \BR^3_+)$ and define $\wt q_* (t) := e^{- (\bI - \Delta_{\pd \BR^3_+}) t} \wt q_0$. Then, we obtain
\begin{equation*}
\wt q_* \in F^{1\slash2 - 1\slash(2 q)}_{p, q, \delta} (J; L^q (\pd \BR^3_+)) \cap L^p_\delta (J; B^{1 - 1\slash q}_{q, q} (\pd \BR^3_+)),
\end{equation*}
see~\cite[Thm.~4.2]{MV14}. If we set $q_* := \wt q_* \vert_{\Gamma_*}$, it holds
\begin{equation*}
q_* \in F^{1\slash2 - 1\slash(2 q)}_{p, q, \delta} (J; L^q (\Gamma_*)) \cap L^p_\delta (J; B^{1 - 1\slash q}_{q, q} (\Gamma_*))
\end{equation*}
and $q_* (0) = q_0$. Hence, given $q_*$, there exists a unique solution $\pi_* \in L^p_\delta (J; \dot H^{1, q} (\Omega_*))$ of the weak problem
\begin{equation*}
\left\{\begin{aligned}
(\nabla \pi_* \mid \nabla \varphi)_{\Omega_*} & = 0, & \quad & \text{for any $\varphi \in H^{1, q'}_{\Gamma_*} (\Omega_*)$}, \\
\pi_* & = q_*, & \quad & \text{on $\Gamma_*$}
\end{aligned}\right.
\end{equation*}
as follows from Lemma~\ref{lem-A.4}, where we suppose that $2 < q < \infty$. Here, we have defined $H^{1, q'}_{\Gamma_*} (\Omega_*)$ by $H^{1, q'}_{\Gamma_*} (\Omega_*) := \{w \in H^{1, q'} (\Omega_*) \mid w = 0 \enskip \text{on $\Gamma_*$}\}$. \par
Now, we consider the parabolic problem
\begin{equation}
\label{reduced-parabolic}
\left\{\begin{aligned}
\pd_t \bu_* - \mu \Delta \bu_* & = - \nabla \pi_* + \bff, & \quad & \text{in $\Omega_* \times J$}, \\
\mu (\pd_3 u_{*, m} + \pd_m u_{*, 3}) & = k_m, & \quad & \text{on $\Gamma_* \times J$}, \\
2 \mu \pd_3 u_{*, 3} & = k_3 - q_* + \sigma \Delta_{\Gamma_*} \eta_*, & \quad & \text{on $\Gamma_* \times J$}, \\
P_{\Sigma_*} (2 \mu \bD (\bu_*) \bn_{\Sigma_*}) & = P_{\Sigma_*} \bg, & \quad & \text{on $\Sigma_* \times J$}, \\
\bu_* \cdot \bn_{\Sigma_*} & = \bg \cdot \bn_{\Sigma_*}, & \quad & \text{on $\Sigma_* \times J$}, \\
\mu (\pd_3 u_{*, m} + \pd_m u_{*, 3}) & = h_m, & \quad & \text{on $B \times J$}, \\
u_{*, 3} & = h_3, & \quad & \text{on $B \times J$}, \\
\bu_* (0) & = \bu_0, & \quad & \text{in $\Omega_*$},
\end{aligned}\right.
\end{equation}
where $m = 1, 2$. According to Lemma~\ref{lem-A.6}, this system admits a solution $\bu_* \in H^{1, p}_\delta (J; L^q (\Omega_*)^3) \cap L^p_\delta (J; H^{2, q} (\Omega_*)^3)$. Here, all relevant compatibility conditions of the data are valid from the assumption. Setting $\bu_5 = \bu - \bu_*$ and $\pi_5 = \pi - \pi_*$, there is no loss of generality in assuming $\bu_0 = \eta_0 = \bff = 0$. To deal with $f_\mathrm{div}$, let us consider the elliptic problem
\begin{equation}
\label{reduced-elliptic}
\left\{\begin{aligned}
\Delta \psi & = f_\mathrm{div} - \dv \bu_*, & \quad & \text{in $\Omega_*$}, \\
\psi & = 0, & \quad & \text{on $\Gamma_*$}, \\
\bn_{\Sigma_*} \cdot \nabla \psi & = 0, & \quad & \text{on $\Sigma_*$}, \\
\bn_B \cdot \nabla \psi & = 0, & \quad & \text{on $B$}.
\end{aligned}\right.
\end{equation}
Notice that, by the compatibility conditions for $(f_\mathrm{div}, \bg \cdot \bn_{\Sigma_*}, h_3)$, we observe that $\int_{\Omega_*} (f_\mathrm{div} - \dv \bu_*) \dx = 0$ and
\begin{equation*}
f_\mathrm{div} - \mathrm{\bu_*} \in {}_0 H^{1, p}_\delta (J; \dot H^{- 1, q} (\Omega_*)) \cap L^p_\delta (J; H^{1, q} (\Omega_*)).
\end{equation*}
Thus, by Lemma~\ref{lem-A.5}, there exists a solution $\psi$ satisfying $\nabla \psi \in {}_0 \BE_{1, \delta} (J; \Omega_*)$. Then, setting $\bu_6 := \bu_5 - \nabla \psi$ and $\pi_6 = \pi_5 + \pd_t \psi - \mu \Delta \psi$, there is no loss of generality in assuming $f_\mathrm{div} = \bg \cdot \bn_{\Sigma_*} = h_3 = 0$. Furthermore, all the remaining data have vanishing trace at $t = 0$.
\subsection{Localization procedure}
\label{sect-parametrix}
Theorem~\ref{th-MR-linear} can be proved via a standard localization procedure with the results given above. Since the discussion is parallel as in \cite[Sec.~2.3]{W17}, we do not repeat the argument.
\section{Local well-posedness}
\label{sect-nonlinear}
\subsection{Nonlinearity}
As an application of the contraction mapping principle, we construct a unique strong solution to \eqref{eq-fixed}. To this end, we rewrite the system as
\begin{equation}
\label{abstract-system}
L z = N (z), \quad (\bu (0), \eta (0)) = (\bu_0, \eta_0),
\end{equation}
where we have set $z = (\bu, \pi, \eta)$ and $L$ stands for the linear operator representing the left-hand side of \eqref{eq-fixed}. The nonlinear mapping $N = N (z) := (\bN_1, N_2, N_3, \bN_4, \bN_5, \bN_6, 0)$ is given by
\begin{equation*}
\begin{split}
\bN_1 & := \bF (\bu, \pi, \eta), \\
N_2 & := F_\mathrm{div} (\bu, \eta), \\
N_3 & := D (\bu, \eta), \\
\bN_4 & := (K_1 (\bu, \eta), K_2 (\bu, \eta), K_3 (\bu, \eta)), \\
\bN_5 & := \bG (\bu, \eta), \\
\bN_6 & := (H_1 (\bu, \eta), H_2 (\bu, \eta)).
\end{split}
\end{equation*}
For shake of simplicity, we define
\begin{equation*}
\begin{split}
\BE (T) & := \{(\bu, \pi, \eta) \in \BE_{1, \delta} (J; \Omega_*) \times \BE_{2, \delta} (J; \Omega_*) \times  \BE_{5, \delta} (J; \Gamma_*) \mid \mathrm{Tr}_{\Gamma_*} [\pi] \in \BE_{3, \delta} (J; \Gamma_*)\}, \\
\BE_{5, \delta} (J; \Gamma_*) & := F^{2 - 1\slash(2 q)}_{p, q, \delta} (J; L^q (\Gamma_*)) \cap H^{1, p}_\delta (J; B^{2 - 1\slash q}_{q, q} (\Gamma_*)) \\
& \qquad \qquad \cap F^{1\slash2 - 1\slash(2 q)}_{p, q, \delta} (J; H^{2, q} (\Gamma_*)) \cap L^p_\delta (J; B^{3 - 1\slash q}_{q, q} (\Gamma_*)), \\
\BF (T) & := \BF_{0, \delta} (J; \Omega_*) \times \BF_{1, \delta} (J; \Omega_*) \times \BF_{3, \delta} (J; \Gamma_*) \times \BF_{2, \delta}^3 (J; \Gamma_*) \times \BF^2_{2, \delta} (J; \Sigma_*) \times \BF_{2, \delta}^2 (J; B) \times \{0\}
\end{split}
\end{equation*}
for $T > 0$, where $J = (0, T)$. Here, the generic elements of $\BF (T)$ are the function $N (z)$. 
With sufficiently small $\varepsilon > 0$, we set $\BU_T := \{z = (\bu, \pi, \eta) \in \BE (T) \mid \lvert \eta \rvert_{L^\infty (J; L^\infty (\Gamma_*))} < \varepsilon \}$. We first derive suitable estimates for the nonlinearity $N (z)$.
\begin{prop}
\label{prop-nonlinearity}
Let $p$, $q$, $\delta$ satisfy
\begin{equation}
\label{cond-pqdelta}
2 < p < \infty, \quad 3 < q < \infty, \quad \frac{1}{p} + \frac{3}{2 q} < \delta - \frac{1}2 \le \frac{1}2.
\end{equation}
Then it holds
\begin{enumerate}
\item $N$ is a real analytic mapping from $\BU_T$ to $\BF (T)$ and $N (0) = D N (0)$.
\item $D N (z) \in \CB (\BU_T, \BF (T))$ for any $z \in \BE (T)$. 
\end{enumerate}
Here, $D N$ denotes the Fr\'{e}chet derivative of $N$.
\end{prop}
\begin{rema}
\label{rema-cond-pqdelta}
(1) The last condition of \eqref{cond-pqdelta} guarantees the embedding $B^{2 (\delta - 1\slash p)}_{q, p} (\Omega_*) \hookrightarrow \mathrm{BUC} (J; \mathrm{BUC}^1 (\Omega_*))$. Furthermore, this condition also induces the conditions
\begin{equation*}
\frac{1}{p} + \frac{1}{2 q} < \delta - \frac{1}2, \qquad \frac{1}{p} + \frac{1}{q} < \delta - \frac{1}2,
\end{equation*}
and thus we need all compatibility conditions on the boundaries and the contact lines given before whenever $p$, $q$, $\delta$ satisfy \eqref{cond-pqdelta}. 
\end{rema}
To prove Proposition~\ref{prop-nonlinearity}, we introduce the following useful lemma.
\begin{lemm}
\label{embd-nonlinear}	
Assume that $p$, $q$, and $\delta$ satisfy \eqref{cond-pqdelta} and let $J = [0, T]$. Then the following assertions are valid.
\begin{enumerate}
\item $\BE_{1, \delta} (J; \Omega_*) \hookrightarrow \mathrm{BUC}^1 (J; \mathrm{BUC} (\ov{\Omega_*}))$.
\item $\BE_{3, \delta} (J; \Gamma_*) \hookrightarrow \mathrm{BUC} (J; \mathrm{BUC} (\ov{\Gamma_*}))$.
\item $\BE_{4, \delta} (J; \Gamma_*) \hookrightarrow \BE_{5, \delta} (J; \Gamma_*) \hookrightarrow \mathrm{BUC}^1 (J; \mathrm{BUC}^1 (\ov{\Gamma_*})) \cap \mathrm{BUC} (J; \mathrm{BUC}^2 (\ov{\Gamma_*}))$.
\item $\BE_{3, \delta} (J; \Gamma_*)$ and $\BF_{3, \delta} (J; \Gamma_*)$ are multiplication algebras.
\end{enumerate}
Here, in the assertions (1)--(3), the embedding constants are independent of $T > 0$ if the time traces vanish at $t = 0$.
\end{lemm}
\begin{rema}
The assertions of Lemma~\ref{embd-nonlinear} is also valid for the case when the domain $\Omega_* \subset \BR^n$, $n \ge 2$, is surrounded by smooth boundary $\Gamma_*$ (at least of class $C^{2 -}$) if we replace \eqref{cond-pqdelta} by
\begin{equation*}
2 < p < \infty, \quad n < q < \infty, \quad \frac{1}{p} + \frac{n}{2 q} < \delta - \frac{1}2 \le \frac{1}2.
\end{equation*}
\end{rema}
\begin{proof}
Using extensions and restrictions, and employing the standard localization procedure, it suffices to consider the cases $\Gamma_* \times J = \BR^2 \times \BR$, $\Omega_* \times J = \BR^3 \times \BR$, see also \cite{L19}. \par
(1) This is a direct consequence of the trace method of real interpolation (cf. Amann~\cite[Thm.~III.4.10.2]{A95}). Here, the last condition of \eqref{cond-pqdelta} ensures $B^{2 (\delta - 1\slash p)}_{q, p} (\BR^3) \hookrightarrow \mathrm{BUC}^1 (\BR^3)$, see Triebel~\cite[Thm.~2.5.7]{T95} (cf. Sawano~\cite[Prop.~2.4]{Sbook}). \par
(2) From \cite[Thm.~3.1]{MV14}, we have
\begin{equation*}
\BE_{3, \delta} (\BR; \BR^2) \hookrightarrow F^{(1 - \theta_1) (1\slash2 - 1\slash(2 q))}_{p, q, \delta} (\BR; B^{\theta_1 (1 - 1\slash q)}_{q, q} (\BR^2))
\end{equation*}
for any $0 < \theta_1 < 1$, where we have used $L^q (\BR^2) = F^0_{q, 2} (\BR^2)$ and the real interpolation
\begin{equation*}
(F^0_{q, 2} (\BR^2), F^{1 - 1\slash q}_{q, q} (\BR^2))_{\theta_1, q} = B^{\theta_1 (1 - 1\slash q)}_{q, q} (\BR^2),
\end{equation*}
cf., \cite[Thm.~2.4.1]{T95}. Then, by \cite[Prop.~7.4]{MV12} and \cite[Thm.~2.5.7]{T95} (cf. \cite[Prop.~2.4]{Sbook}), we see that
\begin{equation*}
F^{(1 - \theta_1) (1\slash2 - 1\slash(2 q))}_{p, q, \delta} (\BR; B^{\theta_1 (1 - 1\slash q)}_{q, q} (\BR^2)) \hookrightarrow \mathrm{BUC} (\BR; \mathrm{BUC} (\BR^2))
\end{equation*}
if $\theta_1$ satisfies
\begin{equation*}
(1 - \theta_1) \bigg(\frac{1}2 - \frac{1}{2 q}\bigg) - \bigg(1 - \delta + \frac{1}{p}\bigg) > 0, \qquad \theta_1 \bigg(1 - \frac{1}{q}\bigg) - \frac2{q} > 0.
\end{equation*}
Noting the last condition of \eqref{cond-pqdelta}, these both inequalities are equivalent to finding the constant $0 < \theta_1 < 1$ such that
\begin{equation*}
\frac{2 q^{- 1}}{1 - q^{- 1}} < \theta_1 < \frac{2 (\delta - p) - (2 q)^{- 1} - 2^{- 1}}{1 - q^{- 1}}.
\end{equation*}
Notice that this set is nonempty due to the last condition of \eqref{cond-pqdelta}. Hence, for any $\theta_1$ satisfying this condition, we obtain the embedding $\BE_{3, \delta} (\BR; \BR^2) \hookrightarrow \mathrm{BUC} (\BR; \mathrm{BUC} (\BR^2))$. \par	
(3)	By \cite[Thm.~3.1]{MV14}, we find that
\begin{equation*}
\begin{split}
& H^{1, p}_\delta (\BR; B^{2 - 1\slash q}_{q, q} (\BR^2)) \cap L^p_\delta (\BR; B^{3 - 1\slash q}_{q, q} (\BR^2)) \\
& \quad \hookrightarrow F^{1\slash2}_{p, q, \delta} (\BR; B^{2 - 1\slash q}_{q, q} (\BR^2)) \cap F^0_{p, \infty, \delta} (\BR; B^{3 - 1\slash q}_{q, q} (\BR^2)) \\
& \quad \hookrightarrow F^{1\slash2 - 1\slash(2 q)}_{p, q, \delta} (\BR; B^2_{q, 1} (\BR^2)) \\
& \quad \hookrightarrow F^{1\slash2 - 1\slash(2 q)}_{p, q, \delta} (\BR; H^{2, q} (\BR^2)),
\end{split}
\end{equation*}
where we have used the real interpolation $(B^{2 - 1\slash q}_{q, q} (\BR^2), B^{3 - 1\slash q}_{q, q} (\BR^2))_{1\slash q, 1} = B^2_{q, 1} (\BR^2)$ (cf. Triebel~\cite[Thm.~2.4.1]{T95}). This yields $\BE_{4, \delta} (J; \Gamma_*) \hookrightarrow \BE_{5, \delta} (J; \Gamma_*)$. Similarly, from \cite[Thm.~3.1]{MV14}, we see that
\begin{equation*}
\begin{split}
\BE_{5, \delta} (\BR; \BR^2) & \hookrightarrow F^{2 - 1\slash(2 q)}_{p, q, \delta} (\BR; F^0_{q, 2} (\BR^2)) \cap F^1_{p, \infty, \delta} (\BR; B^{2 - 1\slash q}_{q, q} (\BR^2)) \cap F^0_{p, \infty, \delta} (\BR; B^{3 - 1\slash q}_{q, q} (\BR^2)) \\
& \hookrightarrow \Big(F^{2 - 1\slash(2 q)}_{p, q, \delta} (\BR; F^0_{q, 2} (\BR^2)) \cap F^1_{p, q, \delta} (\BR; F^{2 - 1\slash q}_{q, q} (\BR^2))\Big) \\
& \qquad \qquad \cap \Big(F^1_{p, \infty, \delta} (\BR; B^{2 - 1\slash q}_{q, q} (\BR^2)) \cap F^0_{p, \infty, \delta} (\BR; B^{3 - 1\slash q}_{q, q} (\BR^2))\Big) \\
& \hookrightarrow F^{2 - 1\slash(2 q) - \theta_2 (1 - 1\slash(2 q))}_{p, q, \delta} (\BR; B^{\theta_2 (2 - 1\slash q)}_{q, q} (\BR^2)) \cap F^{1 - \theta_3}_{p, \infty, \delta} (\BR; B^{\theta_3 + 2 - 1\slash q}_{q, q} (\BR^2))
\end{split}
\end{equation*}
with $0 < \theta_2, \theta_3 < 1$ because $L^q (\BR^2) = F^0_{q, 2} (\BR^2)$ and $B^{2 - 1\slash q}_{q, q} (\BR^2) = F^{2 - 1\slash q}_{q, q} (\BR^2)$. Here, we have used the interpolation properties 
\begin{equation*}
\begin{split}
(F^0_{q, 2} (\BR^2), F^{2 - 1\slash q}_{q, q} (\BR^2))_{\theta_2, q} & = B^{\theta_2 (2 - 1\slash q)}_{q, q} (\BR^2), \\
(B^{2 - 1\slash q}_{q, q} (\BR^2), B^{3 - 1\slash q}_{q, q} (\BR^2))_{\theta_3, q} & = B^{\theta_3 + 2 - 1\slash q}_{q, q} (\BR^2),
\end{split}
\end{equation*}
see, e.g., \cite[Thm.~2.4.1]{T95}. Using \cite[Prop.~7.4]{MV12} and \cite[Thm.~2.5.7]{T95} (cf. \cite[Prop.~2.4]{Sbook}), we obtain
\begin{equation*}
\begin{split}
F^{2 - 1\slash(2 q) - \theta_2 (1 - 1\slash(2 q))}_{p, q, \delta} (\BR; B^{\theta_2 (2 - 1\slash q)}_{q, q} (\BR^2)) & \hookrightarrow \mathrm{BUC}^1 (\BR; \mathrm{BUC}^1 (\BR^2)), \\
F^{1 - \theta_3}_{p, \infty, \delta} (\BR; B^{\theta_3 + 2 - 1\slash q}_{q, q} (\BR^2)) & \hookrightarrow \mathrm{BUC} (\BR; \mathrm{BUC}^2 (\BR^2))
\end{split}
\end{equation*}
whenever
\begin{equation*}
\begin{aligned}
2 - \frac{1}{2 q} - \theta_2 \bigg(1 - \frac{1}{2 q}\bigg) - \bigg(1 - \delta + \frac{1}{p}\bigg) & > 1, & \qquad \theta_2 \bigg(2 - \frac{1}{q}\bigg) - \frac2{q} & > 1, \\
1 - \theta_3 - \bigg(1 - \delta + \frac{1}{p}\bigg) & > 0, & \qquad \theta_3 + 2 - \frac{1}{q}- \frac2{q} & > 2.
\end{aligned}
\end{equation*}
Using the last condition of \eqref{cond-pqdelta}, we deduce that the constants $\theta_2$ and $\theta_3$ enjoy the conditions
\begin{equation*}
\begin{split}
\frac{2^{- 1} + 3 (2 q)^{- 1} - (2 q)^{- 1}}{1 - (2 q)^{- 1}} & < \theta_2 < \frac{\delta - p^{- 1} - (2 q)^{- 1}}{1 - (2 q)^{- 1}}, \\
\frac{3}{q} & < \theta_3 < \delta - \frac{1}{p},
\end{split}
\end{equation*}
respectively. Hence, for $\theta_2$ and $\theta_3$ both satisfying these inequalities, we arrive at the desired assertion. \par
(4) Let $f, g \in \BE_{3, \delta} (J; \Gamma_*)$. From well-known paraproduct estimates (cf. Bahouri et. al~\cite[Cor.~2.86.]{BCD11}), it holds
\begin{equation}
\label{est-E3-1}
\begin{split}
& \lvert f g \rvert_{L^p_\delta (\BR; B^{1 - 1\slash q}_{q, q} (\BR^2))} \\
& \le C \Big\lvert \lvert f \rvert_{L^\infty (\BR^2)} \lvert g \rvert_{B^{1 - 1\slash q}_{q, q} (\BR^2)} + \lvert f \rvert_{B^{1 - 1\slash q}_{q, q} (\BR^2)} \lvert g \rvert_{L^\infty (\BR^2)} \Big\rvert_{L^p_\delta (\BR)} \\
& \le C \Big(\lvert f \rvert_{L^\infty (\BR; L^\infty (\BR^2))} \lvert g \rvert_{L^p_\delta (\BR; B^{1 - 1\slash q}_{q, q} (\BR^2))} + \lvert f \rvert_{L^p_\delta (\BR; B^{1 - 1\slash q}_{q, q} (\BR^2))} \lvert g \rvert_{L^\infty (\BR; L^\infty (\BR^2))} \Big).
\end{split}
\end{equation}
For $f \in L^p_\delta (\BR; L^q (\BR^2))$ we define
\begin{equation*}
[ f ]_{F^s_{p, q, \delta} (\BR; L^q (\BR^2))}^{(1)} = \bigg\lvert \bigg(\int_0^\infty y^{- s q} \bigg(y^{- 1} \int_{\lvert h \rvert \le y} \lvert \tau_h f - f \rvert_{L^q (\BR^2)} \,\mathrm{d} h \bigg)^q \frac{\mathrm{d} y}{y} \bigg)^{1\slash q} \bigg\rvert_{L^p_\delta (\BR)}
\end{equation*}
with $s = 1\slash2 - 1\slash(2 q) < 1$ and set
\begin{equation*}
\lVert f \rVert_{F^s_{p, q, \delta} (\BR; L^q (\BR^2))}^{(1)} = \lvert f \rvert_{L^p_\delta (\BR; L^q (\BR^2))} + [ f ]_{F^s_{p, q, \delta} (\BR; L^q (\BR^2))}^{(1)}.
\end{equation*}
Here $\{\tau_h\}_{h \in \BR}$ denotes the group of translations defined by
\begin{equation*}
(\tau_h f) (x) = f (x + h) \qquad (x, h \in \BR).
\end{equation*}
Then, according to \cite[Prop.~2.1]{MV14}, we know that $F^s_{p, q, \delta} (\BR; L^q (\BR^2))$ can be characterized as
\begin{equation*}
C^{- 1} \lvert f \rvert_{F^s_{p, q, \delta} (\BR; L^q (\BR^2))} \le \lVert f \rVert_{F^s_{p, q, \delta} (\BR; L^q (\BR^2))}^{(1)} \le C \lvert f \rvert _{F^s_{p, q, \delta} (\BR; L^q (\BR^2))}
\end{equation*}
with some constant $C > 0$. Writing
\begin{equation*}
\tau_h (f g) - f g = \tau_h f (\tau_h g - g) + (\tau_h f - f) g
\end{equation*}
and using the H\"older inequality, we have
\begin{equation*}
[f g]_{F^s_{p, q, \delta} (\BR; L^q (\BR^2))}^{(1)} \le \lvert f \rvert_{L^\infty (\BR; L^\infty (\BR^2))} [ g ]_{F^s_{p, q, \delta} (\BR; L^q (\BR^2))} + [ f ]_{F^s_{p, q, \delta} (\BR; L^q (\BR^2))} \lvert g \rvert_{L^\infty (\BR; L^\infty (\BR^2))}.
\end{equation*}
Noting
\begin{equation*}
\lvert f g \rvert_{L^p_\delta (\BR; L^q (\BR^2))} \le \lvert f \rvert_{L^\infty (\BR; L^\infty (\BR^2))} \lvert g \rvert_{L^p_\delta (\BR; L^q (\BR^2))}, 
\end{equation*}
we achieve at the inequality
\begin{equation}
\label{est-E3-2}
\begin{split}
& \lvert f g \rvert_{F^s_{p, q, \delta} (\BR; L^q (\BR^2))} \\ & \le C \lVert f g \rVert_{F^s_{p, q, \delta} (\BR; L^q (\BR^2))}^{(1)} \\
& \le C \Big(\lvert f \rvert_{L^\infty (\BR; L^\infty (\BR^2))} \lVert g \rVert^{(1)}_{F^s_{p, q, \delta} (\BR; L^q (\BR^2))} + \lVert f \rVert_{F^s_{p, q, \delta} (\BR; L^q (\BR^2))}^{(1)} \lvert g \rvert_{L^\infty (\BR; L^\infty (\BR^2))} \Big) \\
& \le C \Big(\lvert f \rvert_{L^\infty (\BR; L^\infty (\BR^2))} \lvert g \rvert_{F^s_{p, q, \delta} (\BR; L^q (\BR^2))} + \lvert f \rvert_{F^s_{p, q, \delta} (\BR; L^q (\BR^2))} \lvert g \rvert_{L^\infty (\BR; L^\infty (\BR^2))} \Big).
\end{split}
\end{equation}
Consequently, by \eqref{est-E3-1} and \eqref{est-E3-2}, it holds
\begin{equation*}
\lvert f g \rvert_{\BE_{3, \delta} (\BR; \BR^2)} \le C \Big(\lvert f \rvert_{L^\infty (\BR; L^\infty (\BR^2))} \lvert g \rvert_{\BE_{3, \delta} (\BR; \BR^2)} + \lvert f \rvert_{\BE_{3, \delta} (\BR; \BR^2)} \lvert g \rvert_{L^\infty (\BR; L^\infty (\BR^2))} \Big).
\end{equation*}
Since $\BE_{3, \delta} (\BR; \BR^2)$ embeds continuously into $\mathrm{BUC} (\BR; \mathrm{BUC} (\BR^2))$, we see that $\BE_{3, \delta} (\BR; \BR^2)$ is a multiplication algebra. \par
It remains to prove that $\BF_{3, \delta} (\BR; \BR^2)$ is a multiplication algebra. Employing the same argument as above, we only need to prove that $\BF_{3, \delta} (\BR; \BR^2)$ embeds continuously into $\mathrm{BUC} (\BR; \mathrm{BUC} (\BR^2))$. For arbitrary $0 < \theta_4 < \infty$, it holds
\begin{equation*}
\BF_{3, \delta} (\BR; \BR^2) \hookrightarrow F^{(1 - \theta_4) (1 - 1\slash(2 q))}_{p, q, \delta} (\BR; B^{\theta_4 (2 - 1\slash q)}_{q, q} (\BR^2)).
\end{equation*}
If $\theta_4$ satisfies
\begin{equation*}
(1 - \theta_4) \bigg(1 - \frac{1}{2 q}\bigg) - \bigg(1 - \delta + \frac{1}{p}\bigg) > 0, \qquad \theta_4 \bigg(2 - \frac{1}{q}\bigg) - \frac2{q} > 0,
\end{equation*}
we obtain
\begin{equation*}
F^{(1 - \theta_4) (1 - 1\slash(2 q))}_{p, q, \delta} (\BR; B^{\theta_4 (2 - 1\slash q)}_{q, q} (\BR^2)) \hookrightarrow \mathrm{BUC} (\BR; \mathrm{BUC} (\BR^2)).
\end{equation*}
The conditions on $\theta_4$ are rewritten as
\begin{equation*}
\frac{q^{- 1}}{1 - (2 q)^{- 1}} < \theta_4 < \frac{\delta - p^{- 1} - (2 q)^{- 1}}{1 - (2 q)^{- 1}},
\end{equation*}
and thus for any $\theta_4$ satisfying this inequality, we obtain the required property. This completes the proof.
\end{proof}
\begin{proof}[Proof of Proposition~\ref{prop-nonlinearity}]	
Since the mapping $z \mapsto N (z)$ is polynomial, it suffices to show that $N \colon \BU_T \to \BF (T)$ is well-defined and continuous. Noting the mapping properties of the differential operators, cf., \cite[Prop.~3.10]{MV12} and \cite[pp.~88]{L19}, the assertions can be proved as in \cite[Prop.~6.2]{PS10}.
\end{proof}

\subsection{Nonlinear well-posedness}
We are now ready to prove the existence result for the transformed problem~\eqref{eq-fixed}.
\begin{theo}
\label{th-fixed}
Let $T > 0$ be a given constant. Suppose that \eqref{cond-pqdelta} holds. Then there exists a constant $\varepsilon = \varepsilon (T) > 0$ such that for all initial data $(\bu_0, \eta_0) \in B^{2 (\delta - 1\slash p)}_{q, p} (\Omega_*)^3 \times B^{2 + \delta - 1\slash p - 1\slash q}_{q, p} (\Gamma_*)$ satisfying the compatibility conditions
\begin{equation}
\label{compati-fixed}
\left\{\begin{aligned}
\dv \bu_0 & = F_\mathrm{div} (\bu_0, \eta_0), & \quad & \text{in $\Omega_*$}, \\
\mu (\pd_3 u_{0, m} + \pd_m u_{0, 3}) & = K_m (\bu_0, \eta_0), & \quad & \text{on $\ov{\Gamma_*}$}, \\
\bu_0 \cdot \bn_{\Sigma_*} = 0, \quad P_{\Sigma_0} (2 \mu \bD (\bu_0) \bn_{\Sigma_*}) & = \bG (\bu_0, \eta_0), & \quad & \text{on $\ov{\Sigma_*}$}, \\
\quad u_{0, 3} = 0, \quad \mu (\pd_3 u_{0, m} + \pd_m u_3) & = H_m (\bu_0, \eta_0), & \quad & \text{on $\ov B$}, \\
(\nabla_{\pd D} \eta) \cdot \bn_{\pd D} & = 0, & \quad & \text{on $S_*$},
\end{aligned}\right.
\end{equation}
where $m = 1, 2$, and the smallness condition
\begin{equation*}
\lvert \bu_0 \rvert_{B^{2 (\delta - 1\slash p)}_{q, p} (\Omega_*)} + \lvert \eta_0 \rvert_{B^{2 + \delta - 1\slash p - 1\slash q}_{q, p} (\Gamma_*)} \le \varepsilon
\end{equation*}
the transformed problem~\eqref{eq-fixed} has a unique solution $(\bu, \pi, \eta) \in \BE (T)$. Furthermore, the solution $(\bu, \pi)$ is real analytic in $\Omega_* \times J$. Especially, $\CM = \bigcup_{t \in (0, T)} (\Gamma_t \times \{t\})$ is a real analytic manifold.
\end{theo}
\begin{proof}
We introduce an auxiliary function $z_* \in \BE (T)$ that resolves the compatibility conditions and the initial conditions. This makes us to reduce the problem \eqref{abstract-system} into
\begin{equation*}
L z = N (z + z_*) - L z_* =: K_0 (z), \qquad z \in {}_0 \BE (T).
\end{equation*}
Since the mapping $L \colon {}_0 \BE (T) \to {}_0 \BF (T)$ is an isomorphism, we can solve this reduced problem by means of the contraction mapping principle. \par
\noindent \textbf{Step 1.} Set $\BI_0 := B^{2 (\delta - 1\slash p)}_{q, p} (\Omega_*)^3 \times B^{2 + \delta - 1\slash p - 1\slash q}_{q, p} (\Gamma_*)$. Let $(\bu_0, \eta_0) \in \BI_0$ satisfy the compatibility conditions~\eqref{compati-fixed}. Set
\begin{equation*}
q_1 := - 2 \mu \pd_3 + \sigma \Delta_{\Gamma_*} \eta_0 + K_3 (\bu_0, \eta_0).
\end{equation*}
Then, from Proposition~\ref{prop-nonlinearity}, we see that
\begin{equation*}
K_3^0 := K_3 (\bu_0, \eta_0) - q_1 + \sigma \Delta_{\Gamma_*} \eta_0 \in B^{2 (\delta - 1 \slash p) - 1 - 1 \slash q}_{q, p} (\Gamma_*),
\end{equation*}
cf., Section~\ref{sec-reduction_data} for the similar argument. Extend $K_3^0$ to an appropriate function $\wt K_3^0 \in B^{2 (\delta - 1 \slash p) - 1 - 1 \slash q}_{q, p} (\pd_3 \BR^3)$ and define $\wt K^*_3 (t) := e^{- (\bI - \Delta_{\pd_3 \BR^3}) t} \wt K_3^0$. Then, by~\cite[Thm.~4.2]{MV14}, we have
\begin{equation*}
\wt K_3^* \in F^{1 \slash 2 - 1 \slash (2 q)}_{p, q, \delta} (J; L^q (\pd \BR^3)) \cap L^p_\delta (J; B^{1- 1 \slash q}_{q, q} (\pd \BR^3)).
\end{equation*}
Hence, if we set $K^*_3 := \wt K^*_3 \vert_{\Gamma_*}$, it holds
\begin{equation*}
K_*^3 \in F^{1 \slash 2 - 1 \slash (2 q)}_{p, q, \delta} (J; L^q (\Gamma_*)) \cap L^p_\delta (J; B^{1 - 1 \slash q}_{q, q} (\Gamma_*))
\end{equation*}
with $K_3^* (0) = K_3^0$. Similarly, we can construct the functions $K^*_1$, $K^*_2$, $H^*_1$, and $H^*_2$ satisfying
\begin{equation*}
\begin{split}
K^*_1, K^*_2 & \in F^{1 \slash 2 - 1 \slash (2 q)}_{p, q, \delta} (J; L^q (\Gamma_*)) \cap L^p_\delta (J; B^{1 - 1 \slash q}_{q, q} (\Gamma_*)), \\
H_1^*, H_2^* & \in F^{1 \slash 2 - 1 \slash (2 q)}_{p, q, \delta} (J; L^q (B)) \cap L^p_\delta (J; B^{1 - 1 \slash q}_{q, q} (B)),
\end{split}
\end{equation*}
where $K^*_1 (0) = K_1 (\bu_0, \eta_0)$, $K^*_2 (0) = K_2 (\bu_0, \eta_0)$, $H^*_1 (0) = H_1 (\bu_0, \eta_0)$, and $H^*_2 (0) = H_2 (\bu_0, \eta_0)$. \par
We next deal with the compatibility condition on $\Sigma_*$. Since $\Sigma_*$ has a boundary, we consider a compact hypersurface $\wt \Sigma_*$ of class $C^3$ \textit{without} boundary but containing $\Sigma_*$. Let $\Delta_{\wt \Sigma_*}$ be the Laplace-Beltrami operator on $\wt \Sigma_*$. It is well-known that the negative of the operator $\bI - \Delta_{\wt \Sigma_*}$ generates an exponentially stable analytic semigroup $\{e^{- (\bI - \Delta_{\wt \Sigma_*}) t}\}_{t \ge 0}$ on $L^q (\wt \Sigma_*)$, see \cite[Thm.~6.4.3]{PS16}. As we have seen in Section~\ref{sec-reduction_data}, we see that $G_j (\bu_0, \eta_0) \in B^{2 (\delta - 1 \slash p) - 1 - 1 \slash q}_{q, p} (\Sigma_*)$, which can be derived from Proposition~\ref{prop-nonlinearity} and the fact that $\mathrm{Tr}_{t = 0}$ and the boundary operator commute (cf. Lindemulder~\cite[pp.~88]{L19}). Now, we extend $\bG (\bu_0, \eta_0)$ to $\wt \bG (\bu_0, \eta_0) \in B^{2 (\delta - 1 \slash p) - 1 - 1 \slash q}_{q, p} (\wt \Sigma_*)^3$. Then, it follows from \cite[Thm.~4.2]{MV14} that
\begin{equation*}
e^{- (\bI - \Delta_{\wt \Sigma_*}) t} \wt G_j (\bu_0, \eta_0) \in F^{1 \slash 2 - 1 \slash (2 q)}_{p, q, \delta} (J; L^q (\wt \Sigma_*)) \cap L^p_\delta (J; B^{1- 1 \slash q}_{q, q} (\wt \Sigma_*))
\end{equation*}
for $j = 1, 2, 3$, where $\wt \bG = (\wt G_1, \wt G_2, \wt G_3)^\top$. Restricting $e^{- (\bI - \Delta_{\wt \Sigma_*}) t} \wt \bG (\bu_0, \eta_0)$ to $\Sigma_*$, denoted by $\bG^*$, it holds
\begin{equation*}
\bG^* \in F^{1 \slash 2 - 1 \slash (2 q)}_{p, q, \delta} (J; L^q (\Sigma_*)) \cap L^p_\delta (J; B^{1- 1 \slash q}_{q, q} (\Sigma_*))
\end{equation*}
with $\bG^* (0) = \bG (\bu_0, \eta_0)$. \par
Consider the parabolic problem
\begin{equation*}
\left\{\begin{split}
\pd_t \bu^{* *} - \mu \Delta \bu^{* *} & = 0, & \quad & \text{in $\Omega_* \times J$}, \\
\mu (\pd_3 u_1^{* *} + \pd_1 u_3^{* *}) & = K_1^*, & \quad & \text{on $\Gamma_* \times J$}, \\
\mu (\pd_3 u_2^{* *} + \pd_2 u_3^{* *}) & = K^*_2, & \quad & \text{on $\Gamma_* \times J$}, \\
2 \mu \pd_3 u_3^{* *} & = K_3^*, & \quad & \text{on $\Gamma_* \times J$}, \\
\bu^{* *} \cdot \bn_{\Sigma_*} & = 0, & \quad & \text{on $\Sigma_* \times J$}, \\
P_{\Sigma_*} (2 \mu \bD (\bu^{* *}) \bn_{\Sigma_*}) & = \bG^*, & \quad & \text{on $\Sigma_* \times J$}, \\
u_3^{* *} & = 0, & \quad & \text{on $B \times J$}, \\
\mu (\pd_3 u_1^{* *} + \pd_1 u_3^{* *}) & = H^*_1, & \quad & \text{on $B \times J$}, \\
\mu (\pd_3 u_2^{* *} + \pd_2 u_3^{* *}) & = H^*_2, & \quad & \text{on $B \times J$}, \\
\bu^{* *} (0) & = \bu_0, & \quad & \text{in $\Omega_*$}.
\end{split}\right.
\end{equation*}
Then, by Lemma~\ref{lem-A.6}, there exists a unique solution $\bu^{* *} \in \BE_{1, \delta} (J; \Omega_*)$ due to $\bu_0 \in B^{2 (\delta - 1/p)}_{q, p} (\Omega_*)^3$. Using this solution, we define $F_\mathrm{div}^* := \dv \bu^{* *}$ with $F_\mathrm{div}^* (0) = \dv \bu_0$, where $F_\mathrm{div}^*$ belongs to the same regularity class as $f_\mathrm{div}$ described in Theorem~\ref{th-MR-linear}. Notice that, from the compatibility condition of $F_\mathrm{div}$, it holds $F_\mathrm{div}^* (0) = F_\mathrm{div} (\bu_0, \eta_0)$. Besides, in the following, we set $\bK^* := (K_1^*, K_2^*, K^*_3)^\top$ and $\bH^* := (H_1^*, H_2^*)^\top$. Employing the similar argument in Section~\ref{sec-reduction_data}, we know $D (\bu_0, \eta_0) \in B^{2 (\delta - 1 \slash p) - 1 \slash q}_{q, p} (\Gamma_*)$ due to Proposition~\ref{prop-nonlinearity}. Extend $D (\bu_0, \eta_0)$ to $\wt D (\bu_0, \eta_0) \in B^{2 (\delta - 1 \slash p) - 1 \slash q}_{q, p} (\pd \BR^3)$. From \cite[Thm.~4.2]{MV14}, we deduce that
\begin{equation*}
\wt D^* := e^{- (\bI - \Delta_{\pd \BR^3}) t} \wt D (\bu_0, \eta_0) \in \BF_{3, \delta} (J; \pd \BR^3).
\end{equation*}
Setting $D^* := \wt D^* \vert_{\Gamma_*}$, it holds $D^* \in \BF_{3, \delta} (J; \Gamma_*)$ and $D^* (0) = D (\bu_0, \eta_0)$.
\par
From Theorem~\ref{th-MR-linear}, we can find the unique solution $z^* \in \BE (T)$ of the linear problem
\begin{equation*}
L z^* = (0, F^*_\mathrm{div}, D^*, \bK^*, \bG^*, \bH^*, 0), \qquad z^* (0) = (\bu_0, \eta_0) \in \BI_0,
\end{equation*}
where $z^*$ resolves the compatibility conditions \eqref{compati-fixed}. Notice that $z^*$ enjoys the estimate
\begin{equation*}
\lvert z^* \rvert_{\BE (T)} \le C_0 \lvert (\bu_0, \eta_0) \rvert_{\BI_0} 
\end{equation*}
with some constant $C_0$ that does not depend on $(\bu_0, \eta_0)$. \par
\noindent \textbf{Step 2.} By Theorem~\ref{th-MR-linear}, the operator $L \colon {}_0 \BE (T) \to {}_0 \BF (T)$ is an isomorphism, and hence the solution of \eqref{eq-fixed} is given by $z = L^{- 1} K_0 (z)$. From the constructions of $(0, F^*_\mathrm{div}, D^*, \bK^*, \bG^*, \bH^*, 0)$, we see that $K_0 (z) \in {}_0 \BF (T)$ for any $z \in {}_0 \BE (T)$. Hence, by Proposition~\ref{prop-nonlinearity}, the mapping $K_0$ is real analytic from ${}_0 \BE (T)$ to ${}_0 \BF (T)$ yielding that $L^{- 1}_0 K_0 \colon {}_0 \BE (T) \to {}_0 \BE (T)$ is well-defined and smooth. \par
In the following, for a given Banach space $W$, we set
\begin{equation*}
c \BB_W := \{w \in W \mid \lvert w \rvert_W \le c \},
\end{equation*}
where $c$ is a positive number. Let $M_0 := \lvert L^{- 1} \rvert_{\CB ({}_0 \BF (T), {}_0 \BE (T))}$ and $M_1 := \lvert L \rvert_{\CB (\BE (T), \BF (T))}$. Based on Proposition~\ref{prop-nonlinearity}, there exists $\delta_1 > 0$ such that
\begin{equation*}
\lvert D N (z + z^*) \rvert_{\CB (\BE (T), \BF (T))} \le \frac{1}{2 M_0}
\end{equation*}
for $(z + z^*) \in \delta_1 \BB_{\BE (T)}$. We further suppose that $z \in (\delta_1 \slash 2) \BB_{\BE (T)}$. Then, for $(\bu_0, \eta_0) \in \varepsilon \BB_{\BI_0}$, we obtain
\begin{equation*}
\lvert z + z^* \rvert_{\BE (T)} \le \lvert z \rvert_{\BE (T)} + \lvert z^* \rvert_{\BE (T)} \le \frac{\delta_1}2 + C_0 \varepsilon < \delta_1.
\end{equation*}
provided that $\varepsilon$ is sufficiently small such that
\begin{equation*}
\varepsilon < \frac{\delta_1}{2 C_0 (1 + 2 M_0 M_1)}.
\end{equation*}
Therefore, by the mean value theorem, we have
\begin{equation*}
\begin{split}
\lvert L^{- 1} K_0 (z) \rvert_{\BE (T)} & \le M_0 \lvert K_0 (z) \rvert_{\BF (T)} \\
& \le M_0 \Big(\lvert N (z + z^*) \rvert_{\BF (T)} + M_1 \lvert z^* \rvert_{\BE (T)} \Big) \\
& \le M_0 \bigg(\frac{1}{2 M_0} \lvert z + z^*\rvert_{\BE (T)} + M_1 C_0 \lvert (\bu_0, \eta_0) \rvert_{\BI_0}\bigg) \\
& \le \frac{\delta_1}{4} + \bigg(\frac12 + M_0 M_1 \bigg) C_0 \varepsilon \\
& \le \frac{\delta_1}2.
\end{split}
\end{equation*}
Hence, the mapping $L^{- 1} K_0 \colon (\delta_1 \slash 2) \BB_{{}_0 \BE (T)} \to (\delta_1 \slash 2) \BB_{{}_0 \BE (T)}$ is a self-mapping. In addition, using the mean value theorem again, it holds
\begin{equation*}
\lvert L^{- 1} K_0 (z_1) - L^{- 1} K_0 (z_2) \rvert_{\BE (T)} \le \frac{1}2 \lvert z_1 - z_2 \rvert_{\BE (T)}
\end{equation*}
for all $z_1, z_2 \in (\delta_1 \slash 2) \BB_{{}_0 \BE (T)}$. This shows that $L^{- 1} K_0 \colon (\delta_1 \slash 2) \BB_{{}_0 \BE (T)} \to (\delta_1 \slash 2) \BB_{{}_0 \BE (T)}$ is contractive. Hence, the contraction mapping principle implies a unique fixed point $z_* \in (\delta_1 \slash 2) \BB_{{}_0 \BE (T)}$ of $L^{- 1} K_0$, i.e., $z_* = L^{- 1} K_0 (z_*)$. Accordingly, we find that $\wt z := z^* + z_*$ solves $L (\wt z) = N (\wt z)$.
\par
\noindent \textbf{Step 3.} From Lemma~\ref{prop-nonlinearity}, the right-hand side of \eqref{eq-fixed} is real analytic. Therefore, by the parameter trick as employed in~\cite[Thm.~6.3]{PS10} (cf.~\cite[Sec.~8]{EPS03} and~\cite[Ch.~9]{PS16}), we can show that the solution $(\bu, \pi, \eta)$ is real analytic as well. This completes the proof.
\end{proof}

Finally, we provide the proof of our main result, Theorem~\ref{th-main}.
\begin{proof}[Proof of Theorem~\ref{th-main}]
First, we see that the compatibility conditions of Theorem~\ref{th-fixed} are satisfied if and only if \eqref{compati-fixed} is satisfied. Define the mapping $\Theta_{\eta_0}$ by $\Theta_{\eta_0} (x) = x + \boldsymbol{\theta}_{\eta_0} (x') = x + \chi (x_3) \eta_0 (x_1, x_2)$. This defines a $C^2$-diffeomorphism from $\Omega_*$ onto $\Omega_0$ with inverse $\Theta_{\eta_0}^{- 1} (x) := (x - \boldsymbol{\theta}_{\eta_0} (x'))$, which follows from the Sobolev embedding theorem. Therefore, there exists some constant $C_{\eta_0}$ depending on $\eta_0$ such that
\begin{equation*}
C^{- 1}_{\eta_0} \lvert \bv_0 \rvert_{B^{2 (\delta - 1 \slash p)}_{q, p} (\Omega_0)} \le \lvert \bu_0 \rvert_{B^{2 (\delta - 1 \slash p)}_{q, p} (\Omega_*)} \le C_{\eta_0} \lvert \bv_0 \rvert_{B^{2 (\delta - 1 \slash p)}_{q, p} (\Omega_0)}.
\end{equation*}
Hence, there exists $\varepsilon_0 > 0$ such that the smallness condition of Theorem~\ref{th-main} implies the smallness condition in Theorem~\ref{th-fixed}. Theorem~\ref{th-fixed} yields the existence of a unique solution $(\bu, \pi, \eta) \in \BE (T)$ of \eqref{eq-fixed} that satisfies the additionally regularity properties denoted in Theorem~\ref{th-fixed}. Furthermore, setting
\begin{equation*}
(\bv, \fp) (x, t) = (\bu, \pi) (x - \boldsymbol{\theta}_\eta (x, t), t), \qquad (x, t) \in \CO := \{(x, t) \in \Omega_t \times (0, T)\},
\end{equation*}
we find that $(\bv, \fp, \eta)$ is a unique solution to \eqref{eq-main}, where $(\bv, \fp)$ is real analytic in $\CO$ and $\CM = \bigcup_{t \in (0, T)} (\Gamma_t \times \{t\})$ is a real analytic manifold.
\end{proof}  

\subsection*{Acknowledgments}
\noindent
The author would like to thank Giovanni Paolo Galdi for the warm hospitality when the author stayed at the University of Pittsburgh in the summer of 2019.

\appendix
\section{Extension operators}
\noindent
We collect some technical results that are needed for the execution of the Stokes equations with slip-free boundary conditions (cf. Theorem~\ref{th-model-slipfree}). The following lemmas are generalization of Propositions A.1 and~A.2 in \cite{W17}.
\begin{lemm}
\label{lem-ext1}
Let $1 < p < \infty$, $2 < q < \infty$, $1\slash p < \delta \le 1$, and $J = (0, T)$. Then there exists a bounded linear extension operator $\mathrm{ext}$ from
\begin{equation*}
{}_0 F^{1\slash2 - 1\slash q}_{p, q, \delta} (J; L^q (\BR)) \cap L^p_\delta (J; B^{1 - 2\slash q}_{q, q} (\BR)) := X_1
\end{equation*}
to
\begin{equation*}
{}_0 F^{1\slash2 - 1\slash(2q)}_{p, q, \delta} (J; L^q (\BR \times \BR_+)) \cap L^p_\delta (J; B^{1 - 1\slash q}_{q, q} (\BR \times \BR_+)) =: Y_1
\end{equation*}
such that $\mathrm{ext} [v] \vert_{\BR \times \{0\}} = v$ for all $v \in Y_1$. Furthermore, if $v = v (x_1, y, t) \in Y_1$, then $\mathrm{Tr}_{y = 0} [v] \in X_1$ such that the trace map is bounded from $Y_1$ to $X_1$.
\end{lemm}
\begin{proof}
This is a direct consequence of \cite[Thm.~4.6]{L19}.
\end{proof}
\begin{lemm}
\label{lem-ext2}
Suppose $1 < p < \infty$, $2 < q < \infty$, $1\slash p < \delta \le 1$ and set $J = (0, T)$ with $T > 0$. If
\begin{equation*}
f \in {}_0 F^{3\slash2 - 1\slash q}_{p, q, \delta} (J; L^q (\BR)) \cap {}_0 H^{1, p}_\delta (J; B^{1 - 2\slash q}_{q, q} (\BR)) \cap L^p_\delta (J; B^{2 - 2\slash q}_{q, q} (\BR)) =: X_2,
\end{equation*}
then there exists $g = g (t, x_1, y)$ with the regularity
\begin{equation*}
g \in {}_0 F^{2 - 1\slash(2 q)}_{p, q, \delta} (J; L^q (\BR \times \BR_+)) \cap {}_0 H^{1, p}_\delta (J; B^{2 - 1\slash q}_{q, q} (\BR \times \BR_+)) \cap L^p_\delta (J; B^{3 - 1\slash q}_{q, q} (\BR \times \BR_+)) =: Y_2
\end{equation*}
such that $\pd_y g = f$ at $y = 0$. 
\end{lemm}
\begin{proof}
Consider the operator $L_0 := \pd_t - \pd_{x_1}^2$ in the space $X_0 := L^p_\delta (J; L^q (\BR))$ with domain
\begin{equation*}
\mathsf{D} (L_0) = {}_0 H^{1, p}_\delta (J; L^q (\BR)) \cap L^p_\delta (J; H^{2, q} (\BR)),
\end{equation*}
in which $L_0$ is invertible on $X_0$ and admits  bounded imaginary powers with power angle not exceeding $\pi\slash2$. It is not difficult to see that $L_0$ is the negative generator of an analytic semigroup $\{e^{- y L_0^{1\slash2}}\}_{y \ge 0}$ in $X_0$ with domain $\mathsf{D} (L_0^{1\slash2}) = [X_0, \mathsf{D} (L_0)]_{1\slash2}$, where $[\,\cdot, \cdot\,]_\theta$ is the complex interpolation functor with $\theta \in (0, 1)$, see, e.g., \cite[pp.~255]{PS16}. Here, $L^{1\slash2}_0$ has bounded imaginary powers with angle not larger than $\pi\slash4$.
Now, set $g (y) = - e^{y L^{1\slash2}_0} L^{- 1\slash2}_0 f$ for $f \in X_2$, $y > 0$. Recalling $f \in X_2$, we easily see that
\begin{equation}
\label{emb-f}
f, \pd_t f, L^{- 1\slash2}_0 f, L^{- 1\slash2}_0 \pd_t f \in {}_0 F^{1\slash2 - 1\slash q}_{p, q, \delta} (J; L^q (\BR)) \cap L^p_\delta (J; B^{1 - 2\slash q}_{q, q} (\BR)).
\end{equation}
In fact, the operator $L^{- 1\slash2}_0$ induces an isomorphism from $D_{L_0} (1\slash2 - 1\slash(2 q), p)$ onto $D_{L_0} (1 - 1\slash(2 q), p)$, where we have set
\begin{equation*}
\begin{split}
D_{L_0} (1\slash2 - 1\slash(2 q), p) & = {}_0 F^{1\slash2 - 1\slash(2 q)}_{p, p, \delta} (J; L^q (\BR)) \cap L^p_\delta (J; B^{1\slash2 - 1\slash(2 q)}_{q, q} (\BR)), \\
D_{L_0} (1 - 1\slash(2 q), p) & = {}_0 F^{1 - 1\slash(2 q)}_{p, p, \delta} (J; L^q (\BR)) \cap L^p_\delta (J; B^{1 - 1\slash(2 q)}_{q, q} (\BR))
\end{split}
\end{equation*}
and used \cite[Thm.~1.15.2 (e)]{T95} and \cite[Thm.~3.4.7]{PS16}. Here, we have also set $D_{L_0} (\alpha, r) := (X, \mathsf{D} (L_0))_{\alpha, r}$ for $\alpha \in (0, 1)$ and $r \in [1, \infty]$, where $(\,\cdot, \cdot\,)_{\alpha, r}$ is the real interpolation functor. Then, recalling $q > 2$, we have the embedding properties 
\begin{equation*}
\begin{split}
& {}_0 F^{1\slash2 - 1\slash q}_{p, q, \delta} (J; L^q (\BR)) \cap L^p_\delta (J; B^{1 - 2\slash q}_{q, q} (\BR)) \\
& \quad \hookrightarrow {}_0 F^{1\slash2 - 1\slash(2 q)}_{p, p, \delta} (J; L^q (\BR)) \cap L^p_\delta (J; B^{1\slash2 - 1\slash(2 q)}_{q, q} (\BR)), \\
& {}_0 F^{1 - 1\slash(2 q)}_{p, p, \delta} (J; L^q (\BR)) \cap L^p_\delta (J; B^{1 - 1\slash(2 q)}_{q, q} (\BR)) \\
& \quad \hookrightarrow {}_0 F^{1\slash2 - 1\slash q}_{p, q, \delta} (J; L^q (\BR)) \cap L^p_\delta (J; B^{1 - 2\slash q}_{q, q} (\BR)),
\end{split}
\end{equation*}
yielding that $L^{- 1\slash2}_0$ maps continuously from ${}_0 F^{1\slash2 - 1\slash q}_{p, q, \delta} (J; L^q (\BR)) \cap L^p_\delta (J; B^{1 - 2\slash q}_{q, q} (\BR))$ to itself. This infers~\eqref{emb-f}. Using Lemma~\ref{lem-ext1}, from $\pd_y g = f$ at $y = 0$, we deduce that
\begin{equation*}
g, \pd_t g \in {}_0 F^{1 - 1\slash(2 q)}_{p, q, \delta} (J; L^q (\BR \times \BR_+)) \cap L^p_\delta (J; B^{2 - 1\slash q}_{q, q} (\BR \times \BR_+)),
\end{equation*}
see also \cite[pp.~88]{L19}. Especially, from \cite[Prop.~3.10]{MV12}, we arrive at
\begin{equation*}
g \in {}_0 F^{2 - 1\slash(2 q)}_{p, q, \delta} (J; L^q (\BR \times \BR_+)) \cap H^{1, p}_\delta (J; B^{2 - 1\slash q}_{q, q} (\BR \times \BR_+)).
\end{equation*}
Finally, by $\pd_y g = f$ at $y = 0$ and $f \in L^p_\delta (J; B^{2 - 2\slash q}_{q, q} (\BR))$, we have $g \in L^p_\delta (J; B^{3 - 1\slash q}_{q, q} (\BR \times \BR_+))$ in view of the trace theory. Summing up, we observe $g \in Y_2$. This completes the proof.
\end{proof}
\section{Auxiliary elliptic and parabolic problems}
\subsection{Elliptic problems}
\noindent
We start with considering the auxiliary problem
\begin{equation}
\label{eq-zaremba}
\left\{\begin{split}
- \Delta \phi & = f, & \quad & \text{in $\Omega_*$}, \\
\phi & = g, & \quad & \text{on $\Gamma_*$}, \\
\bn_{\Sigma_*} \cdot \nabla \phi & = h_1, & \quad & \text{on $\Sigma_*$}, \\
\bn_B \cdot \nabla \phi & = h_2, & \quad & \text{on $B$}.
\end{split}\right.
\end{equation}
The domain $\Omega_*$ stands for either a whole space, a (bent) half space, a (bent) quarter space, or the cylindrical domain given in \eqref{def-domain}. We look for solutions $\phi$ satisfying $\nabla \phi \in H^{1, q} (\Omega_*)$ since we cannot expect to seek $\phi \in L^q (\Omega_*)$ when $\Omega_*$ is unbounded. Notice that if $\phi$ is a solution to~\eqref{eq-zaremba} with $g = 0$ and $\nabla \phi \in H^{1, q} (\Omega_*)$, then we have $f \in L^q (\Omega_*)$, $h_1 \in B^{1 - 1\slash q}_{q, q} (\Sigma_*)$, and $h_2 \in B^{1 - 1\slash q}_{q, q} (B)$. As we introduced before, let $\wh H^{- 1, q}_{\Sigma_* \cup B} (\Omega_*)$ be the set of all $(f, h_1, h_2) \in L^q (\Omega_*) \times B^{1 - 1 \slash q}_{q, q} (\Sigma_*) \times B^{1 - 1 \slash q}_{q, q} (B)$ such that $(f, h_1, h_2)$ belongs to $\dot H^{- 1, q}_{\Sigma_* \cup B} (\Omega_*)$. The following lemma can be shown along the same lines of Lemmas~A.6 in \cite{W17}, and hence we do not repeat the argument.
\begin{lemm}
\label{lem-A.3}
Let $2 < q < \infty$. Assume the compatibility conditions
\begin{equation}
\label{compati-elliptic}
\begin{aligned}
\mathrm{Tr}_{S_*} [g] & = \mathrm{Tr}_{S_*} [h_1] & \quad & \text{on $S_*$}, \\
\mathrm{Tr}_{\pd \Sigma_* \cap \pd B} [h_1] & = \mathrm{Tr}_{\pd \Sigma_* \cap \pd B} [h_2] & \quad & \text{on $\pd \Sigma_* \cap \pd B$}
\end{aligned}
\end{equation}
when $\Omega_*$ is a (bent) quarter space or the cylindrical domain defined in \eqref{def-domain}. Then the following assertions hold.
\begin{enumerate}
\item If $\Omega_*$ is a whole space, a (bent) half space, or a (bent) quarter space, then the problem \eqref{eq-zaremba} has a unique solution $\phi$ with $\nabla \phi \in H^{1, q} (\Omega_*)$ and $g = 0$ if and only if $(f, h_1, h_2) \in \wh H^{- 1, q}_{\Sigma_* \cup B} (\Omega_*)$.
\item If $\Omega_*$ is the cylindrical domain defined in \eqref{def-domain}, then there \eqref{eq-zaremba} admits a unique solution $\phi \in H^{2, q} (\Omega_*)$ with $g = h_1 = h_2 = 0$ if and only if $f$ satisfies
\begin{equation*}
f \in L^q_0 (\Omega_*) = \{f \in L^q (\Omega_*) \mid \int_{\Omega_*} f \dx = 0\}.
\end{equation*}	
\end{enumerate}		
\end{lemm}
\begin{rema}
Although the result in Lemma~A.6 in \cite{W17} includes the case $q = 2$, we emphasize that it is required and crucial to assume $q > 2$. Indeed, we need this assumption to employ the reflection argument because there exists the trace onto the contact lines whenever $q > 2$.
\end{rema}
As a corollary of Lemma~\ref{lem-A.3}, we can prove the existence of weak solution to \eqref{eq-zaremba} provided that $h_1 = h_2 = 0$ and \eqref{compati-elliptic}.
\begin{lemm}
\label{lem-A.4}
Let $2 < q < \infty$ and $\Omega_*$ be the cylindrical domain defined in \eqref{def-domain}. Suppose the compatibility conditions \eqref{compati-elliptic} with $h_1 = h_2 = 0$. Furthermore, let $f \in \dot H^{- 1, q}_{\Sigma_* \cup B} (\Omega_*)$ and $g \in B^{1 - 1\slash q}_{q, q} (\Gamma_*)$ be given. Then there exists a unique solution $\phi \in \dot H^{1, q} (\Omega_*)$ to the weak version of \eqref{eq-zaremba}
\begin{equation*}
\left\{
\begin{aligned}
(\nabla \phi \mid \nabla \varphi)_{\Omega_*} & = (f \mid \varphi)_{\Omega_*}, \\
\phi & = g, & \quad \text{on $\Gamma_*$}
\end{aligned}\right.
\end{equation*}
for any $\varphi \in H^{1, q'}_{\Gamma_*} (\Omega_*)$, where we have set $H^{1, q'}_{\Gamma_*} (\Omega_*) := \{w \in H^{1, q'} (\Omega_*) \mid w = 0 \enskip \text{on $\Gamma_*$}\}$.	
\end{lemm}
\begin{proof}
We follow the proof of \cite[Lem.~A.7]{W17}. However, the argument in \cite[Lem.~A.7]{W17} requires the result of~\cite[Sec.~8]{KPW13}, so that we instead use the result due to~\cite[Sec.~7.4]{PS16}. Besides, the space of test functions in~\cite[Lem.~A.7]{W17} was $W^{1, p'} (\Omega)$, but we point out that the test functions have to vanish on the boundary where the pressure term appears, which will play an important role in integration by parts; related to the divergence equation. In fact, owing to this investigation, by~\eqref{eq-zaremba} and integration by parts, we find that $(\nabla \phi \mid \nabla \varphi)_{\Omega_*} = (f \mid \varphi)_{\Omega_*}$ holds for any $\varphi \in \dot H^{1, q'}_{\Gamma_*} (\Omega_*) \hookrightarrow H^{1, q'}_{\Gamma_*} (\Omega_*)$ provided $h_1 = h_2 = 0$. Then, the required property can be shown in a same way as in the proof of \cite[Lem.~A.7]{W17}.
\end{proof}
As a consequence of Lemma~\ref{lem-A.4}, we can show the higher regularity result for the solution $\phi$.
\begin{lemm}
\label{lem-A.5}
Let $2 < q < \infty$ and $J = (0, T)$. Suppose that $g = h_1 = h_2 = 0$. Then the following statements are valid.
\begin{enumerate}
\item If $\Omega_*$ is a whole space, a (bent) half space, or a (bent) quarter space, then the problem \eqref{eq-zaremba} admits a unique solution $\phi$ with 
\begin{equation*}
\nabla \phi \in {}_0 H^{1, p}_\delta (J; H^{1, q} (\Omega_*)) \cap L^p_\delta (J; H^{3, q} (\Omega_*)),
\end{equation*}
if and only if
\begin{equation*}
f \in {}_0 H^{1, p}_\delta (J; \dot H^{- 1, q}_{\Sigma_* \cup B} (\Omega_*) \cap L^q (\Omega_*)) \cap L^p_\delta (J; H^{1, q} (\Omega_*)).
\end{equation*}
\item If $\Omega_*$ is the cylindrical domain defined in \eqref{def-domain}, then there \eqref{eq-zaremba} has a unique solution
\begin{equation*}
\phi \in {}_0 H^{1, p}_\delta (J; H^{1, q} (\Omega_*)) \cap L^p_\delta (J; H^{3, q} (\Omega_*))
\end{equation*}
if and only if
\begin{equation*}
f \in {}_0 H^{1, p}_\delta (J; \dot H^{- 1, q}_{\Sigma_* \cup B} (\Omega_*)) \cap L^p_\delta (J; H^{1, q} (\Omega_*)).
\end{equation*}
\end{enumerate}		
\end{lemm}
\begin{proof}
From Lemma~\ref{lem-A.3} and Lemma~\ref{lem-A.4}, we immediately obtain the regularity $\nabla \phi \in {}_0 H^{1, p} (J; H^{1, q} (\Omega_*))$ in the first assertion and $\phi \in {}_0 H^{1, p}_\delta (J; H^{1, q} (\Omega_*))$ in the second assertion, respectively. As for the additional spacial regularity of $\phi$, we use Lemmas~\ref{lem-A.3} and~\ref{lem-A.4}. By means of local coordinates, we may reduce each localized problem to one of the model problems in a whole space, a (bent) half space, and a (bent) quarter space. Employing the perturbation and reflection arguments as in the proof of~\cite[Lem.~A.3]{W17} (cf.~\cite[Thm.~8.6]{KPW13}), we readily obtain the desired result. Hence, we will not repeat the arguments.
\end{proof}
\subsection{Parabolic problems}
Finally, we provide the existence and uniqueness of solutions to the parabolic problem
\begin{equation}
\label{eq-parabolic}
\left\{\begin{aligned}
\pd_t \bu - \mu \Delta \bu & = \bff, & \quad & \text{in $\Omega_* \times J$}, \\
\mu (\pd_3 u_m + \pd_m u_3) & = k_m, & \quad & \text{on $\Gamma_* \times J$}, \\
2 \mu \pd_3 u_3 & = k_3, & \quad & \text{on $\Gamma_* \times J$}, \\
\bu \cdot \bn_{\Sigma_*} & = \bg \cdot \bn_{\Sigma_*}, & \quad & \text{on $\Sigma_* \times J$}, \\
P_{\Sigma_*} (2 \mu \bD (\bu) \bn_{\Sigma_*}) & = P_{\Sigma_*} \bg, & \quad & \text{on $\Sigma_* \times J$}, \\
u_3 & = h_3, & \quad & \text{on $B \times J$}, \\
\mu (\pd_3 u_m + \pd_m u_3) & = h_m, & \quad & \text{on $B \times J$}, \\
\bu (0) & = \bu_0, & \quad & \text{in $\Omega_*$},
\end{aligned}\right.
\end{equation}
where $m = 1, 2$. The following lemma can be established in the same way as~\cite[Lem.~A.10]{W17} with the help of the argument in Theorem~\ref{th-slipslip}.

\begin{lemm}
\label{lem-A.6}
Let $T > 0$ and $J = (0, T)$. Assume that $p$, $q$, $\delta$ satisfy \eqref{cond-pqdelta-model1}. Then, there exists a unique solution
\begin{equation*}
\bu \in H^{1, p}_\delta (J; L^q (\Omega_*)^3) \cap L^p_\delta (J; H^{2, q} (\Omega_*)^3)
\end{equation*}
of \eqref{eq-parabolic} if and only if
\begin{enumerate}
\renewcommand{\labelenumi}{(\alph{enumi})}
\item $\bff \in \BF_{0, \delta} (J; \Omega_*)$;
\item $P_{\Sigma_*} \bg \in \BF_{2, \delta} (J; \Sigma_*)$ for $\ell = 1, 3$;
\item $\bg \cdot \bn_{\Sigma_*} \in \BF_{3, \delta} (J; \Sigma_*)$;	
\item $k_j \in \BF_{2, \delta} (J; \Gamma_*)$ for $j = 1, 2, 3$;
\item $h_m \in \BF_{2, \delta} (J; B)$ for $m = 1, 2$;
\item $h_3 \in \BF_{3, \delta} (J; B)$;
\item $\bu_0 \in B^{2 (\delta - 1\slash p)}_{q, p} (\Omega_*)^3$;
\item $P_{\Sigma_*} \bg (0) = \mathrm{Tr}_{\Sigma_*} [P_{\Sigma_*} (2 \mu \bD (\bu_0) \bn_{\Sigma_*})]$, $k_m (0) = \mathrm{Tr}_{\Gamma_*} [\mu (\pd_3 u_{0, m} + \pd_m u_{0, 3})]$, $m = 1, 2$, $k_3 (0) = \mathrm{Tr}_{\Gamma_*} [2 \mu \pd_3 u_{0, 3}]$, and $h_m (0) = \mathrm{Tr}_B [\mu (\pd_3 u_{0, m} + \pd_m u_{0, 3})]$, $m = 1, 2$, if $1\slash p + 1\slash(2 q) < \delta - 1\slash2$;
\item $\bg (0) \cdot \bn_{\Sigma_*} = \mathrm{Tr}_{\Sigma_*} [\bu_0 \cdot \bn_{\Sigma_*}]$ and $h_3 (0) = \mathrm{Tr}_B [u_{0, 3}]$ if $1\slash p + 1\slash(2 q) < \delta$;
\item $\mathrm{Tr}_{S_*} [P_{\Sigma_*} \bg (0)] = \mathrm{Tr}_{S_*} [P_{\Gamma_*} (2 \mu \bD (\bu_0) \bn_{\Sigma_*})]$, $\mathrm{Tr}_{S_*} [k_m (0)] = \mathrm{Tr}_{S_*} [\mu (\pd_3 u_{0, m} + \pd_m u_{0, 3})]$ for $m = 1, 2$, and $\mathrm{Tr}_{S_*} [k_3 (0)] = \mathrm{Tr}_{S_*} [2 \mu \pd_3 u_{0, 3}]$ if $1\slash p + 1\slash q < \delta - 1\slash2$;
\item $\mathrm{Tr}_{\pd \Sigma_* \cap \pd B} [\bg (0) \cdot \bn_{\Sigma_*}] = \mathrm{Tr}_{\pd \Sigma_* \cap \pd B} [\bu_0 \cdot \bn_{\Sigma_*}]$ and $\mathrm{Tr}_{\pd \Sigma_* \cap \pd B} [h_3 (0)] = \mathrm{Tr}_{\pd \Sigma_* \cap \pd B} [u_{0, 3}]$ if $1\slash p + 1\slash q < \delta$.
\end{enumerate}
Especially, the solution map is continuous between the corresponding spaces.
\end{lemm}



\begin{bibdiv}
\begin{biblist}	
	
\bib{A95}{book}{
	author={Amann, Herbert},
	title={Linear and quasilinear parabolic problems. Vol. I},
	series={Monographs in Mathematics},
	volume={89},
	note={Abstract linear theory},
	publisher={Birkh\"{a}user Boston, Inc., Boston, MA},
	date={1995},
}	
	
\bib{BCD11}{book}{
	author={Bahouri, Hajer},
	author={Chemin, Jean-Yves},
	author={Danchin, Rapha\"{e}l},
	title={Fourier analysis and nonlinear partial differential equations},
	volume={343},
	publisher={Springer, Heidelberg},
	date={2011},
}

	
\bib{BKK}{article}{
	author={Bothe, Dieter},
	author={Kashiwabara, Takahito},
	author={K\"{o}hne, Matthias},
	title={Strong well-posedness for a class of dynamic outflow boundary
		conditions for incompressible Newtonian flows},
	journal={J. Evol. Equ.},
	volume={17},
	date={2017},
	number={1},
	pages={131--171},
}
	
	


\bib{DHP}{article}{
	author={Denk, Robert},
	author={Hieber, Matthias},
	author={Pr\"{u}ss, Jan},
	title={Optimal $L^p$-$L^q$-estimates for parabolic boundary value
		problems with inhomogeneous data},
	journal={Math. Z.},
	volume={257},
	date={2007},
	number={1},
	pages={193--224},
}

\bib{EPS03}{article}{
	author={Escher, Joachim},
	author={Pr\"{u}ss, Jan},
	author={Simonett, Gieri},
	title={Analytic solutions for a Stefan problem with Gibbs-Thomson
		correction},
	journal={J. Reine Angew. Math.},
	volume={563},
	date={2003},
	pages={1--52},
}
	

\bib{FKB19}{article}{
	author={Fricke, Mathis},
	author={K\"{o}hne, Matthias},
	author={Bothe, Dieter},
	title={A kinematic evolution equation for the dynamic contact angle and
		some consequences},
	journal={Phys. D},
	volume={394},
	date={2019},
	pages={26--43},
}


\bib{GT18}{article}{
	author={Guo, Yan},
	author={Tice, Ian},
	title={Stability of contact lines in fluids: 2D Stokes flow},
	journal={Arch. Ration. Mech. Anal.},
	volume={227},
	date={2018},
	number={2},
	pages={767--854},
}

\bib{H81}{article}{
	author={Hanzawa, Ei-ichi},
	title={Classical solutions of the Stefan problem},
	journal={Tohoku Math. J. (2)},
	volume={33},
	date={1981},
	number={3},
	pages={297--335},
}



\bib{K13}{book}{
	author={K\"{o}hne, Matthias},
	title={${L_p}$-theory for incompressible Newtonian flows},
	note={Energy preserving boundary conditions, weakly singular domains;
		Dissertation, Technische Universit\"{a}t Darmstadt, Darmstadt, 2012},
	publisher={Springer Spektrum, Wiesbaden},
	date={2013},
}

\bib{KPW13}{article}{
	author={K\"{o}hne, Matthias},
	author={Pr\"{u}ss, Jan},
	author={Wilke, Mathias},
	title={Qualitative behaviour of solutions for the two-phase Navier-Stokes
		equations with surface tension},
	journal={Math. Ann.},
	volume={356},
	date={2013},
	number={2},
	pages={737--792},
}

\bib{L19}{article}{
	author={Lindemulder, Nick},
	title={Maximal regularity with weights for parabolic problems with
		inhomogeneous boundary conditions},
	journal={J. Evol. Equ.},
	volume={20},
	date={2020},
	number={1},
	pages={59--108},
}

\bib{MS17}{article}{
	author={Maryani, Sri},
	author={Saito, Hirokazu},
	title={On the $\CR$-boundedness of solution operator families for
		two-phase Stokes resolvent equations},
	journal={Differential Integral Equations},
	volume={30},
	date={2017},
	number={1-2},
	pages={1--52},
}


\bib{MV12}{article}{
	author={Meyries, Martin},
	author={Veraar, Mark},
	title={Sharp embedding results for spaces of smooth functions with power
		weights},
	journal={Studia Math.},
	volume={208},
	date={2012},
	number={3},
	pages={257--293},
}

\bib{MV14}{article}{
	author={Meyries, Martin},
	author={Veraar, Mark},
	title={Traces and embeddings of anisotropic function spaces},
	journal={Math. Ann.},
	volume={360},
	date={2014},
	number={3-4},
	pages={571--606},
}

\bib{MS91}{article}{
	author={Mogilevski\u{\i}, I. Sh.},
	author={Solonnikov, V. A.},
	title={On the solvability of a free boundary problem for the
		Navier-Stokes equations in the H\"{o}lder space of functions},
	conference={
		title={Nonlinear analysis},
	},
	book={
		series={Sc. Norm. Super. di Pisa Quaderni},
		publisher={Scuola Norm. Sup., Pisa},
	},
	date={1991},
	pages={257--271},
}


\bib{PS02}{article}{
	author={Padula, M.},
	author={Solonnikov, V. A.},
	title={On the global existence of nonsteady motions of a fluid drop and
		their exponential decay to a uniform rigid rotation},
	conference={
		title={Topics in mathematical fluid mechanics},
	},
	book={
		series={Quad. Mat.},
		volume={10},
		publisher={Dept. Math., Seconda Univ. Napoli, Caserta},
	},
	date={2002},
	pages={185--218},
}



\bib{P19}{article}{
	author={Pr\"{u}ss, Jan},
	title={Vector-valued Fourier multipliers in $L_p$-spaces with power
		weights},
	journal={Studia Math.},
	volume={247},
	date={2019},
	number={2},
	pages={155--173},
}	


\bib{PSW14}{article}{
	author={Pr\"{u}ss, Jan},
	author={Shimizu, Senjo},
	author={Wilke, Mathias},
	title={Qualitative behaviour of incompressible two-phase flows with phase
		transitions: the case of non-equal densities},
	journal={Comm. Partial Differential Equations},
	volume={39},
	date={2014},
	number={7},
	pages={1236--1283},
}

\bib{PS10}{article}{
	author={Pr\"{u}ss, Jan},
	author={Simonett, Gieri},
	title={On the two-phase Navier-Stokes equations with surface tension},
	journal={Interfaces Free Bound.},
	volume={12},
	date={2010},
	number={3},
	pages={311--345},
}

\bib{PS11}{article}{
	author={Pr\"{u}ss, Jan},
	author={Simonett, Gieri},
	title={Analytic solutions for the two-phase Navier-Stokes equations with
		surface tension and gravity},
	conference={
		title={Parabolic problems},
	},
	book={
		series={Progr. Nonlinear Differential Equations Appl.},
		volume={80},
		publisher={Birkh\"{a}user/Springer Basel AG, Basel},
	},
	date={2011},
	pages={507--540},
}

\bib{PS16}{book}{
	author={Pr\"{u}ss, Jan},
	author={Simonett, Gieri},
	title={Moving interfaces and quasilinear parabolic evolution equations},
	series={Monographs in Mathematics},
	volume={105},
	publisher={Birkh\"{a}user\slash Springer, [Cham]},
	date={2016},
}

\bib{SS20pre}{article}{
	author={Saito, Hirokazu},
	author={Shibata, Yoshihiro},
	title={On the global wellposedness for free boundary problem for
		the Navier-Stokes systems with surface tension},
	note={Available at 
		\texttt{arXiv:1912.10121}
	},
}
	
\bib{RE07}{article}{
	author={Ren, Weiqing},
	author={E, Weinan},
	title={Boundary conditions for the moving contact line problem},
	journal={Phys. Fluids},
	volume={19},
	date={2007},
	pages={022101},
}

\bib{Sbook}{book}{
	author={Sawano, Yoshihiro},
	title={Theory of Besov spaces},
	series={Developments in Mathematics},
	volume={56},
	publisher={Springer, Singapore},
	date={2018},
}


\bib{S14}{article}{
	author={Shibata, Yoshihiro},
	title={On the $\scr{R}$-boundedness of solution operators for the Stokes
		equations with free boundary condition},
	journal={Differential Integral Equations},
	volume={27},
	date={2014},
	number={3-4},
	pages={313--368},
}

\bib{S16}{article}{
	author={Shibata, Yoshihiro},
	title={On the $\mathcal R$-bounded solution operator and the maximal
		$L_p$-$L_q$ regularity of the Stokes equations with free boundary
		condition},
	conference={
		title={Mathematical fluid dynamics, present and future},
	},
	book={
		series={Springer Proc. Math. Stat.},
		volume={183},
		publisher={Springer, Tokyo},
	},
	date={2016},
	pages={203--285},
}

\bib{S20}{article}{
	author={Shibata, Yoshihiro},
	title={$\mathcal R$ boundedness, maximal regularity and free boundary
		problems for the Navier Stokes equations},
	conference={
		title={Mathematical analysis of the Navier-Stokes Equations},
	},
	book={
		series={Lecture Notes in Math.},
		volume={2254},
		publisher={Springer, Cham},
	},
	date={2020},
	pages={{193--462}},
}

\bib{SS20}{article}{
	author={Shibata, Yoshihiro},
	author={Saito, Hirokazu}
	title={Global well-posedness for incompressible--incompressible
		two-phase problem},
	conference={
		title={Fluids under pressure},
	},
	book={
		series={Advances in Mathematical Fluid Mechanics},
		publisher={Birkh{\"a}user, Cham.},
	},
	date={2020},
	pages={157--347},
}

\bib{SS11}{article}{
	author={Shibata, Yoshihiro},
	author={Shimizu, Senjo},
	title={Maximal $L_p$-$L_q$ regularity for the two-phase Stokes equations;
		model problems},
	journal={J. Differential Equations},
	volume={251},
	date={2011},
	number={2},
	pages={373--419},
}

\bib{SS12}{article}{
	author={Shibata, Yoshihiro},
	author={Shimizu, Senjo},
	title={On the maximal $L_p$-$L_q$ regularity of the Stokes problem with
		first order boundary condition; model problems},
	journal={J. Math. Soc. Japan},
	volume={64},
	date={2012},
	number={2},
	pages={561--626},
}

\bib{SY17}{article}{
	author={Shimizu, Senjo},
	author={Yagi, Shintaro},
	title={On local $L_p$-$L_q$ well-posedness of incompressible two-phase
		flows with phase transitions: non-equal densities with large initial
		data},
	journal={Adv. Differential Equations},
	volume={22},
	date={2017},
	number={9-10},
	pages={737--764},
}


\bib{S84}{article}{
	author={Solonnikov, V. A.},
	title={Solvability of the problem of evolution of an isolated amount of a
		viscous incompressible capillary fluid},
	language={Russian, with English summary},
	note={Mathematical questions in the theory of wave propagation, 14},
	journal={Zap. Nauchn. Sem. Leningrad. Otdel. Mat. Inst. Steklov. (LOMI)},
	volume={140},
	date={1984},
	pages={179--186},
}


\bib{S89}{article}{
	author={Solonnikov, V. A.},
	title={Unsteady motions of a finite isolated mass of a self-gravitating
		fluid},
	language={Russian},
	journal={Algebra i Analiz},
	volume={1},
	date={1989},
	number={1},
	pages={207--249},
	translation={
		journal={Leningrad Math. J.},
		volume={1},
		date={1990},
		number={1},
		pages={227--276},
	},
}

\bib{S91}{article}{
	author={Solonnikov, V. A.},
	title={Solvability of a problem on the evolution of a viscous
		incompressible fluid, bounded by a free surface, on a finite time
		interval},
	language={Russian},
	journal={Algebra i Analiz},
	volume={3},
	date={1991},
	number={1},
	pages={222--257},
	translation={
		journal={St. Petersburg Math. J.},
		volume={3},
		date={1992},
		number={1},
		pages={189--220},
	},
}

\bib{S95}{article}{
	author={Solonnikov, V. A.},
	title={On some free boundary problems for the Navier-Stokes equations
		with moving contact points and lines},
	journal={Math. Ann.},
	volume={302},
	date={1995},
	number={4},
	pages={743--772},
}


\bib{S11}{article}{
	author={Solonnikov, V. A.},
	title={On the linear problem arising in the study of a free boundary
		problem for the Navier-Stokes equations},
	journal={Algebra i Analiz},
	volume={22},
	date={2010},
	number={6},
	pages={235--269},
	translation={
		journal={St. Petersburg Math. J.},
		volume={22},
		date={2011},
		number={6},
		pages={1023--1049},
	},
}

\bib{T95}{book}{
	author={Triebel, Hans},
	title={Interpolation theory, function spaces, differential operators},
	edition={2},
	publisher={Johann Ambrosius Barth, Heidelberg},
	date={1995},
}

\bib{W17}{thesis}{
	author={Wilke, Mathias},
	title={Rayleigh-Taylor instability for the two-phase Navier--Stokes
		equations with surface tension in cylindrical domains},
	date={2017},
	note={Habilitations--Schrift Universit\"at Halle,
		Naturwissenschaftliche Fakult\"at II (2013), arXiv:1703.05214}
}

\bib{ZT17}{article}{
	author={Zheng, Yunrui},
	author={Tice, Ian},
	title={Local well posedness of the near-equilibrium contact line problem
		in 2-dimensional Stokes flow},
	journal={SIAM J. Math. Anal.},
	volume={49},
	date={2017},
	number={2},
	pages={899--953},
}
		
\end{biblist}
\end{bibdiv}
\end{document}